\documentclass[a4paper,11pt]{article}

\usepackage{scalerel}[2016/12/29]
\usepackage{color}
\usepackage[latin1]{inputenc}
\usepackage[T1]{fontenc}
\usepackage[normalem]{ulem}
\usepackage[english]{babel}
\usepackage{verbatim}
\usepackage{graphicx}
\usepackage{enumerate,enumitem}
\usepackage{amsmath,amssymb,amsfonts,amsthm,mathrsfs}
\usepackage{rotating,relsize}
\usepackage{float}
\usepackage{array}
\usepackage{todonotes}
\usepackage{MnSymbol}
\usepackage{wasysym}
\usepackage[all,cmtip]{xy}
\usepackage{tikz}
\usepackage{tikz-cd}

\usepackage[hidelinks]{hyperref}
\usepackage[clock]{ifsym}
\usepackage{stackengine}
\usepackage{scalerel}
\usepackage{stmaryrd}

\stackMath
\xyoption{all}

\newtheorem{theorem}{Theorem}[section]
\newtheorem{theoremA}{Theorem}

\newtheorem{corA}[theoremA]{Corollary}
\newtheorem*{theorem*}{Theorem}
\newtheorem*{cor*}{Corollary}

\newtheorem{lemma}[theorem]{Lemma}
\newtheorem{prop}[theorem]{Proposition}
\newtheorem{cor}[theorem]{Corollary}

\theoremstyle{definition}
\newtheorem*{remark*}{Remark}
\newtheorem{defn}[theorem]{Definition}
\newtheorem{rem}[theorem]{Remark}
\newtheorem{example}[theorem]{Example}

\DeclareMathOperator{\id}{id}
\DeclareMathOperator{\im}{Im}

\DeclareMathOperator{\tr}{tr}

\DeclareMathOperator{\coker}{coker}

\DeclareMathOperator{\map}{Map}

\DeclareMathOperator*{\colim}{colim}

\DeclareMathOperator{\THH}{THH}

\DeclareMathOperator{\TR}{TR}

\DeclareMathOperator{\Sp}{Sp}

\DeclareMathOperator{\sh}{sh}
\DeclareMathOperator{\Set}{Set}

\DeclareMathOperator{\Z}{\mathbb{Z}}
\DeclareMathOperator{\F}{\mathbb{F}}

\DeclareMathOperator{\proj}{proj}

\DeclareMathOperator{\Ab}{Ab}

\DeclareMathOperator{\bimod}{biMod}

\DeclareMathOperator{\tlog}{tlog}

\DeclareMathOperator{\FF}{\mathcal{F}}
\DeclareMathOperator{\PgcSp}{PgcSp}
\DeclareMathOperator{\Res}{Res}
\DeclareMathOperator{\qfgen}{qfgen}
\DeclareMathOperator{\infl}{infl}

\newcommand{\cyctens}[3]{{#2}\lower.6ex\hbox{${}^{\underset{\scaleto{ #1}{3pt}}{\circledcirc} p^{#3}}$}}
\newcommand{\cycsmash}[3]{{#2}\lower.6ex\hbox{${}^{\underset{\scaleto{ #1}{3pt}}{\varowedge} p^{#3}}$}}
\newcommand{\tens}[3]{{#2}\lower.6ex\hbox{${}^{\underset{\scaleto{ #1}{3pt}}{\otimes} p^{#3}}$}}
\newcommand{\smashp}[3]{{#2}\lower.6ex\hbox{${}^{\underset{\scaleto{ #1}{3pt}}{\wedge} p^{#3}}$}}

\newcommand{\pow}[1]{\left[\kern-0.5ex\left[{#1}\right]\kern-0.5ex\right]}

\begin{document}
\begin{center}\LARGE{Witt vectors with coefficients and TR}
\end{center}

\begin{center}\Large{Emanuele Dotto, Achim Krause, Thomas Nikolaus, Irakli Patchkoria}
\end{center}

\vspace{.05cm}

\abstract{We give a new construction of $p$-typical Witt vectors with coefficients in terms of ghost maps and show that this construction is isomorphic to the one defined in terms of formal power series from the authors' previous paper. We show that our construction recovers  Kaledin's polynomial Witt vectors in the case of vector spaces over a perfect field of characteristic $p$. We then identify the components of the $p$-typical $\TR$ with coefficients, originally defined by Lindenstrauss and McCarthy and later reworked by the second and third authors in joint work with McCandless, with the $p$-typical Witt vectors with coefficients. This extends a celebrated result of Hesselholt and Hesselholt-Madsen relating the components of TR with the Witt vectors. 
As an application, we given an algebraic description of the components of the Hill-Hopkins-Ravenel norm for cyclic $p$-groups in terms of $p$-typical Witt vectors with coefficients. 
}

\vspace{.05cm}


\tableofcontents

\section*{Introduction}
\phantomsection\addcontentsline{toc}{section}{Introduction}

In the paper \cite{DKNP1} we define the Witt vectors $W(R;M)$ of a ring $R$ with coefficients in an $R$-bimodule $M$. This construction extends the usual big Witt vectors of a commutative ring, recovering it in the case where $M=R$. Our approach is analogous to the construction of Witt vectors of a commutative ring 
 in terms of power series (see e.g. \cite{Cartier}), by replacing it with the completed tensor algebra.
In the present paper we give an alternative description of the $p$-typical Witt vectors with coefficients which aligns with the original construction of Witt \cite{Witt}  based on Witt polynomials. In particular we show that for $M=R$ our construction recovers Hesselholt's definition of $p$-typical Witt vectors for non-commutative rings of \cite{HesselholtncW}. We also compare our construction to Kaledin's definition of polynomial Witt vectors from  \cite{KaledinWpoly}, as well as describing the components of topological restriction homology TR with coefficients as defined in \cite{LMcC} and \cite{polygonic}, and in particular of the Hill-Hopkins-Ravenel norm for cyclic $p$-groups, in terms of p-typical Witt vectors.

Let us fix a prime $p$. The Witt polynomials are the $n$-variable polynomials $w_j\in \Z[a_0,\dots,a_{n-1}]$, defined for $0\leq j<n$ as
\[
w_j=a_{0}^{p^j}+pa_{1}^{p^{j-1}}+p^2a_{1}^{p^{j-2}}+\dots +p^{j-1}a_{j-1}^{p}+p^ja_j.
\]
The ring of $n$-truncated $p$-typical Witt vectors $W_n(A)$ of a commutative ring $A$ can be characterised as the unique ring structure on the set $A^{\times n}$ which is functorial in $A$ and with the property that the ``ghost maps'' $w_j\colon A^{\times n}\to A$ defined by the Witt polynomials are ring homomorphisms for every $0\leq j<n$. Given any ring $R$ and $R$-bimodule $M$, let us define a variant of these ghost maps, by formally replacing the $p$-th powers in the ghost map with tensor powers of $M$. For a bimodule $M$ over a ring $R$, we define an $R$-bimodule $\tens{R}{M}{i}$ and an abelian group $\cyctens{R}{M}{j}$ respectively by
\[
\tens{R}{M}{i}=\underbrace{M\otimes_R M\otimes_R\dots\otimes_R M}_{p^i} \ \ \ \ \ \  \ \ \mbox{and} \ \ \ \  \ \ \ \  \cyctens{R}{M}{j}=\tens{R}{M}{j}/[R,\tens{R}{M}{j}]
\vspace{-.3cm}
\]
where $[R,\tens{R}{M}{j}]$ is the abelian subgroup generated by the elements $rm-mr$ for $r\in R$ and $m\in \tens{R}{M}{j}$. We think of $\cyctens{R}{M}{j}$ as $p^j$ copies of $M$ tensored together around a circle, and these define an abelian group with a natural action of the cyclic group $C_{p^j}$. We then define an analogue of the $j$-th ghost map
\[
w_j\colon \prod_{i=0}^{n-1}\tens{R}{M}{i}\longrightarrow (\cyctens{R}{M}{j})^{C_{p^j}},
\]
by sending a sequence $a_0,a_1,\dots,a_{n-1}$ to the invariant represented by
\[
w_j(a_0,a_1,\dots,a_{n-1}):=a^{\otimes p^j}_0+\sum_{\sigma\in C_{p^j}/C_{p^{j-1}}}\sigma(a^{\otimes p^{j-1}}_1)+\sum_{\sigma\in C_{p^j}/C_{p^{j-2}}}\sigma(a^{\otimes p^{j-2}}_2)+\dots+\sum_{\sigma\in C_{p^j}}\sigma(a_j)
\]
where $\sigma\colon \cyctens{R}{M}{j}\to \cyctens{R}{M}{j}$ is the automorphism given by the cyclic action. We then define an equivalence relation on $\prod_{i=0}^{n-1}\tens{R}{M}{i}$ by forcing the ghost map to be injective on free bimodules (see Definition \ref{def:rel}), and define the $p$-typical Witt vectors, as a set, as the quotient
\[
W_{n,p}(R;M):=\left(\prod_{i=0}^{n-1}\tens{R}{M}{j}\right)/\sim
\]
by this relation.
This situation is analogous to the Witt vectors for non-commutative rings $W_n(R)$ of \cite{HesselholtncW}, which as a set is a certain quotient of the product $\prod_{i=0}^{n-1}R$.
The following is the main result of \S\ref{secexist}.

\begin{theoremA}
Let $p$ be a prime, $n\geq 1$ an integer, $R$ a ring and $M$ an $R$-bimodule. There is a unique abelian group structure and lax symmetric monoidal structure on $W_{n,p}(R;M)$ such that $w_j$ is additive and monoidal for all $0\leq j<n$, and a natural monoidal isomorphism of abelian groups 
\[W_{n,p}(R;M)\cong W_{\langle p^{n-1}\rangle}(R;M),\]
where $W_{\langle p^{n-1}\rangle}(R;M)$ is the group of $p$-typical $(n-1)$-truncated Witt vectors of \cite{DKNP1}.
\end{theoremA}
When $M=R$ with the canonical $R$-bimodule structure, there is a natural isomorphism of abelian groups between
$
W_{n,p}(R;R)
$
and Hesselholt's group of $p$-typical $n$-truncated Witt vectors of non-commutative rings of \cite{HesselholtncW} (see Corollary \ref{CompLars}). This follows from the fact that
\[(\cyctens{R}{R}{j})^{C_{p^j}}\cong R/[R,R]\] and under this isomorphism the ghost maps are given by the usual Witt polynomials.
As an immediate consequence of the symmetric monoidal structure $W_{n,p}$ extends to a functor from the category of algebras over commutative rings to rings, which sends commutative algebras to commutative rings. In particular when $R$ is commutative the isomorphism above identifies $W_{n,p}(R;R)$ with the usual ring of Witt vectors.

The groups $W_{n,p}(R;M)$ also  extend Kaledin's construction of  polynomial Witt vectors from \cite{KaledinWpoly} and \cite{KaledinncW}, as follows.  We let $Q_n(R;M)$ be the abelian group defined as the cofiber of the transfer map
\[
Q_n(R;M):=\coker\big(\tr_{e}^{C_{p^n}}=\sum_{\sigma\in C_{p^n}}\sigma\colon (\cyctens{R}{M}{n})_{C_{p^n}}\longrightarrow  (\cyctens{R}{M}{n})^{C_{p^n}}\big).
\]
Kaledin defines in \cite{KaledinWpoly} and \cite{KaledinncW} a functor $\widetilde{W}_n$ of ``polynomial Witt vectors'' on the category of vector spaces over a perfect field $k$ of characteristic $p$, in terms of the functor $Q_n$. The following theorem, proved in  \S\ref{secKaledin}, provides a similar description for $W_{n,p}$, showing that $W_{n,p}$ restricts to Kaledin's construction on the subcategory of $k$-vectors spaces.
\begin{theoremA}
For every prime $p$ and integer $n\geq 1$, there is a surjective lax symmetric monoidal natural transformation
\[w_{n}\colon W_{n,p}(R/p;M/p)\longrightarrow Q_n(R;M).\]
It is an isomorphism when $R$ is a commutative ring with no $p$-power torsion, $R/p$ is perfect, and $M$ is a free $R$-module. It follows that $W_{n,p}(k;V)$ is isomorphic to the polynomial Witt vectors $\widetilde{W}_n(V)$ of \cite{KaledinWpoly} 
for every $k$-vector space $V$ and perfect field $k$ of characteristic $p$.
\end{theoremA}
When $M=R$ the isomorphism of the theorem recovers the fact that if $R$ has no  $p$-power torsion and is commutative, with perfect $R/p$, then $W_n(R/p)\cong R/p^n$.

In homotopy theory, the ring of Witt vectors arises when one considers the cyclotomic structure on topological Hochschild homology. Hesselholt and Madsen show in \cite{Wittvect} and \cite{HesselholtncW} that the $p$-typical Witt vectors of a ring $R$ are isomorphic to $\pi_0$ of the $p$-typical topological restriction homology spectrum $\TR(R)$. In \cite{LMcC} Lindenstrauss and McCarthy define a version of TR with coefficients in an $R$-bimodule, as the derived cyclic invariants
\[
\TR_{\langle p^n \rangle}(R;M)=(\cycsmash{R}{M}{n})^{C_{p^n}}
\]
of a genuine $C_{p^n}$-spectrum $\cycsmash{R}{M}{n}$, which is the derived analogue of the algebraic cyclic tensor product  used in the definition of the Witt vectors.
The foundations of this theory have been reworked in \cite{polygonic} by McCandless and the second and third authors, in a way analogous to the approach to topological cyclic homology of \cite{NikSchol}. In particular for every prime $p$ and integer $n\geq 0$, the authors give a description of TR as an equaliser
\[
\TR_{\langle p^n\rangle}(R;M)=eq\left(\xymatrix{
\prod_{i=0}^n\THH(R;M^{\wedge_R p^i})^{hC_{p^i}}\ar@<.5ex>[r]\ar@<-.5ex>[r]
&
\prod_{i=0}^{n-1}(\THH(R;M^{\wedge_R p^{i+1}})^{tC_p})^{hC_{p^i}}
}
\right)
\]
where $\THH(R;M)$ is the usual topological Hochschild homology spectrum with coefficients, and $\THH(R;M^{\wedge_R p^{i}})$ carries a certain action of the cyclic group $C_{p^i}$. The maps of the equaliser are defined from the canonical map from homotopy fixed-points to the Tate construction, and from certain Frobenius maps
$\THH(R;M)\to \THH(R;M^{\wedge_R p})^{tC_p}$ (see \S\ref{sec:TR}). From this equaliser formula one can easily deduce, for every spectrum $A$, an equivalence
\[
\TR_{\langle p^n\rangle}(\mathbb{S};A)=(N^{C_{p^n}}_eA)^{C_{p^n}}
\]
with the genuine fixed-points of the norm construction of cyclic $p$-groups of Stolz \cite{Sto} and Hill-Hopkins-Ravenel \cite{HHR}. The components of this norm have been computed by Mazur when $A=H\F_p$ (see \cite[Proposition 5.23]{BGHL}), and for  $p=2$ and $n=1$ in \cite[Proposition 5.5]{THRmodels}. In Theorem \ref{pi0norm} we extend this calculation to all connective bimodules, showing the following:

\begin{theoremA}
Let $R$ be a connective ring spectrum and $M$ a connective $R$-bimodule. There is a canonical isomorphism
\[
W_{\langle p^n\rangle}(\pi_0R;\pi_0M)\cong\pi_0\TR_{\langle p^n\rangle}(R;M),
\]
which is moreover natural in $(R;M)$ and monoidal.
In particular for every connective spectrum $A$, this gives an isomorphism $W_{\langle p^n\rangle}(\Z;\pi_0 A)\cong \pi_0(N_{e}^{C_{p^n}}A)^{C_{p^n}}$ with the components of the Hill-Hopkins-Ravenel norm construction, which is a ring isomorphism when $A$ is a ring spectrum.
\end{theoremA}

The proof of this theorem is somewhat similar to Hesselholt and Madsen's proof of the isomorphism between the usual $p$-typical Witt vectors $W_{n+1}(A)$ and the components of the $C_{p^n}$-fixed points of $\THH(A)$, for a commutative ring $A$. The argument is by induction, by comparing certain fibre sequences for TR established in \cite{polygonic}  with the exact sequences for the Witt vectors from \cite[\S 1.5]{DKNP1}. A similar description of the components of the norm for any finite group $G$ has been obtained by Read in \cite{Read}, with a variation of our Witt vectors construction. 

%

By a result of \cite{polygonic}, the spectrum $\TR_{\langle p^{k}\rangle}(R;M^{\wedge_Rp^{n-k}})$ is the $C_{p^k}$-fixed-points of a genuine $C_{p^n}$-spectrum $\underline{\TR}_{\langle p^n\rangle}(R;M)$, for every $0\leq k\leq n$. For example, when $R$ is the sphere spectrum this is the cyclic norm construction
\[
\underline{\TR}_{\langle p^n\rangle}(\mathbb{S};A)=N^{C_{p^n}}_eA
\]
for every spectrum $A$.
In Proposition \ref{structure compatible} we identify the Mackey structure on $\pi_0\underline{\TR}_{\langle p^n\rangle}(R;M)$ in terms of the Witt vectors operators introduced in \cite{DKNP1} and in \S\ref{secop}.
The characterisation of this Mackey structure suggests a relationship between the Witt vectors with coefficients and the free Tambara functor on a commutative ring.
In \cite[Theorem B]{BrunTamb} Brun describes the free $C_{p^n}$-Tambara functor on a commutative ring with \textit{trivial} $C_{p^n}$-action in terms of the usual ring of Witt vectors $W_{n+1}(A)$. In \S\ref{secbrun} we show that the Witt vectors with coefficients in fact compute the free $C_{p^n}$-Tambara functor on every commutative ring.

\begin{corA}
Let $A$ be a commutative ring, $p$ a prime and $n\geq 0$ an integer. The association $C_{p^i}\mapsto W_{i+1}(\Z;A^{\otimes p^{n-i}})$ equipped with the operators  $F$, $V$ and $N$ of \S\ref{secop} form a $C_{p^n}$-Tambara functor, which is the free $C_{p^n}$-Tambara functor on the commutative ring $A$.
\end{corA}

\subsection*{Acknowledgements}

The authors would like to thank Jonas McCandless and Christian Wimmer for helpful conversations. 

The first and fourth authors were supported by the German Research Foundation Schwerpunktprogramm 1786. 
The second and third author were funded by the Deutsche Forschungsgemeinschaft (DFG, German Research Foundation) -- Project-ID 427320536--SFB 1442, as well as under Germany's Excellence Strategy EXC 2044 390685587, Mathematics M\"unster: Dynamics-Geometry-Structure. 

For the purpose of open access, the authors have applied a Creative Commons Attribution (CC-BY) licence to any Author Accepted Manuscript version arising from this submission.

\section{The $p$-typical Witt vectors with coefficients}

In \cite{DKNP1} we defined the Witt vectors with coefficients in a way analogous to the definition of the (big) Witt vectors of a commutative ring in terms of power series \cite{Cartier}. In this section we give an alternative description of the $p$-typical Witt vectors with coefficients, more in line with the usual construction of $p$-typical Witt vectors for commutative rings of \cite{Witt}.

We start by recalling the definition of Witt vectors with coefficients of \cite{DKNP1}. Let $R$ be a ring and $M$ an $R$-bimodule. We let $\widehat{T}(R;M)=\prod_{n\geq 0}M^{\otimes_R n}$ be the completed tensor algebra, and $\widehat{S}(R;M)=1\times \prod_{n\geq 1}M^{\otimes_R n}\subset\widehat{T}(R;M)$ the multiplicative subgroup of special units. We denote the elements of this group by power series 
\[1+m_1t+m_2t^2+\dots\]
 with $m_i\in M^{\otimes_R i}$, and we let $\tau\colon M\to \widehat{S}(R;M)$ be the map that sends $m$ to the power series $1-mt$. The (big) Witt vectors of $R$ with coefficients in $M$ are defined in \cite{DKNP1} as the group
\[W(R;M)= \frac{ \widehat{S}(R;M)^{\mathrm{ab}}}{ \tau(rm) \sim \tau(mr)}\]
where the relation runs over all $m\in M$ and $r\in R$, and the abelianisation and the quotient are taken in Hausdorff topological groups, that is we quotient by the closure of the normal subgroup generated by the relations.

Given a truncation set $S\subset\mathbb{N}$, one can define the $S$-truncated Witt-vectors as a quotient of $W(R;M)$. In the present papers we will be interested in the truncation sets consisting of the powers of a prime, and in this case the truncated Witt vectors are defined as follows.

\begin{defn}[\cite{DKNP1}]
Let $p$ be a prime and $n\geq 0$ an integer. The $p$-typical $(n+1)$-truncated Witt vectors of $R$ with coefficients in $M$ is the abelian group $W_{\langle p^n\rangle}(R;M)$ defined as the quotient of $W(R;M)$ by the closed subgroup generated by the elements of the form 
\[1-m_1\otimes\dots\otimes m_kt^k,\]
 where $k\notin \{1,p,p^2,\dots,p^n\}$ and $m_1,\dots, m_k\in M$.
\end{defn}

The truncated Witt vectors have operators analogous to those of the usual Witt vectors, which will play a crucial role in the rest of the paper. The functor $W_{\langle p^n\rangle}$ from the category $\bimod$ of pairs $(R;M)$ has a lax symmetric monoidal structure \cite[Proposition 1.27]{DKNP1}.  The operators are defined in \cite[\S 1.3,1.5]{DKNP1}, and in the truncated case above they take the form of natural transformations
\[
\begin{array}{ll}
F=F_p\colon W_{\langle p^n\rangle}(R;M)\longrightarrow W_{\langle p^{n-1}\rangle}(R;\tens{R}{M}{})&\hspace{1cm} V=V_p\colon W_{\langle p^{n-1}\rangle}(R;\tens{R}{M}{})\longrightarrow W_{\langle p^n\rangle}(R;M)
\\
\\
R\colon W_{\langle p^n\rangle}(R;M)\longrightarrow W_{\langle p^{n-1}\rangle}(R;M) &\hspace{1cm} \tau=\tau_1\colon M\longrightarrow W_{\langle p^n\rangle}(R;M)
\\
\\
\sigma\colon W_{\langle p^n\rangle}(R;M^{\otimes_Rk})\longrightarrow W_{\langle p^n\rangle}(R;M^{\otimes_Rk})
\end{array}
\]
The maps $F$ and $R$ are respectively called the Frobenius and restriction map, and they are monoidal. The map $V$ is called the Verschiebung, and it is additive, whereas $\tau$ is called the Teichm\"uller map and it is monoidal (with respect to the tensor product over $\Z$ on the source). The map $\sigma$ is an automorphism of order $k$, which we call the Weyl action of $C_k$. These maps satisfy certain relations which are detailed in \cite[Proposition 1.31]{DKNP1}.


%
%
%

\subsection{The $p$-typical Witt vectors with coefficients in Witt coordinates}\label{secexist}

In this section we will give a description for $W_{\langle p^n\rangle}(R;M)$ in terms of sequences of $p$-powers of $M$, which is more in line with the classical definition of the $p$-typical Witt vectors (see e.g. \cite{Witt}), as well as Hesselholt's construction of $p$-typical Witt vectors for non-commutative rings.
Let us define a map
\[
\gamma\colon \prod_{i=0}^{n}\tens{R}{M}{i}\longrightarrow W_{\langle p^n\rangle}(R;M)
\]
by sending a sequence $(m_0,m_1,\dots,m_n)$ to the equivalence class of $\prod_{i=0}^n(1-m_it^{p^i})$, where the product is taken in the completed tensor algebra (see e.g. \cite[Prop 1.14]{LarsBig}  for the case $M=R$ commutative but notice the different sign convention).

In order to analyse this map we make use of a version with coefficients of the ghost map. The cyclic tensor power of $M$ is defined for every $j\geq 0$ as the abelian group
\[
 \cyctens{R}{M}{j}=\tens{R}{M}{j}/[R,\tens{R}{M}{j}],
\]
where $[R,\tens{R}{M}{j}]$ is the abelian subgroup generated by the elements $rm-mr$ for $r\in R$ and $m\in \tens{R}{M}{j}$. The cyclic group $C_{p^j}$ of order $p^j$ acts on this abelian group by cyclically permuting the tensor factors, and for all $l\leq j$ we write $(\cyctens{R}{M}{j})^{C_{p^l}}$ for the subgroup of invariants of a cyclic subgroup $C_{p^l}\leq C_{p^j}$. The transfer maps for this cyclic action are denoted by
\[
\tr_{C_{p^{l}}}^{C_{p^j}}\colon (\cyctens{R}{M}{j})^{C_{p^l}}\longrightarrow (\cyctens{R}{M}{j})^{C_{p^j}}.
\]
 We also let $(\bullet)^{\otimes p^k}$ be the composite
\[
(\bullet)^{\otimes p^k}\colon \tens{R}{M}{i}\longrightarrow (\tens{R}{(\tens{R}{M}{i})}{k})^{C_{p^k}}\cong (\tens{R}{M}{i+k})^{C_{p^k}}\longrightarrow(\cyctens{R}{M}{i+k})^{C_{p^{k}}}
\]
where the first map sends $x$ to $x^{\otimes p^k}$, the isomorphism is the canonical associativity isomorphism of the monoidal structure on $R$-bimodules, and the last map is the canonical projection onto the quotient.
\begin{defn}\label{defghost}
The $j$-th $p$-typical ghost map is the map $w_j\colon  \prod_{i=0}^{n}\tens{R}{M}{i}\to(\cyctens{R}{M}{j})^{C_{p^j}}$ defined by
\[
w_j(m_0,\dots,m_{n}):=\sum_{i=0}^{j}\tr_{C_{p^{j-i}}}^{C_{p^j}}(m_i^{\otimes p^{j-i}}),
\]
for every $0\leq j<n+1$.
The product of these maps $w\colon  \prod_{i=0}^{n}\tens{R}{M}{i}\to\prod_{j=0}^{n}(\cyctens{R}{M}{j})^{C_{p^j}}$ is called the $p$-typical ghost map.
\end{defn}

We observe that for $M=R$ there are canonical isomorphisms of abelian groups  $\tens{R}{R}{j}\cong R$ and $(\cyctens{R}{R}{j})^{C_{p^j}}\cong R/[R;R]$, and $w_j$ corresponds to the usual Witt polynomial 
\[
w_j(r_0,\dots,r_{n})=\sum_{i=0}^{j}p^ir_i^{p^{j-i}}.
\]

We recall that Hesselholt's definition of the $p$-typical Witt vectors of a non-commutative ring $R$ of \cite{HesselholtncW} is a complicated quotient of the product of copies of $R$. Thus if we want our construction to specialise to his when $M=R$, we need to define an equivalence relation on $\prod_{i=0}^{n}\tens{R}{M}{i}$. Informally, we define the smallest equivalence relation which makes the ghost map injective when $(R;M)$ is free, and such that the resulting functor commutes with reflexive coequalisers of bimodules. 

We let $\bimod$ be the category of bimodules, whose objects are pairs $(R;M)$ of a ring $R$ and an $R$-bimodule $M$, and a morphism $(R;M)\to (R';M')$ is a pair $(\alpha;f)$ of a ring homomorphism $\alpha\colon R\to R'$ and a map of $R$-bimodules $f\colon M\to \alpha^\ast M'$. 
Clearly the tensor power and cyclic tensor powers introduced above are functors on the category $\bimod$ by applying the map $f$  factorwise on elementary tensors, and we will often drop $\alpha$ from the notation and denote a morphism only by $f\colon M\to M'$.
We say that a bimodule $(R;M)$  is free if $R$ is a free ring and $M$ is a free $R$-bimodule. A free resolution of  $(R;M)$ is a reflexive coequaliser diagram
\[\xymatrix@C=12pt{(\overline{S};\overline{Q})\ar@<.5ex>[r]^{f}\ar@<-.5ex>[r]_{g}&(S;Q)\ar[l]\ar@{->>}[r]^-{\epsilon}&(R;M)}\]
in the category $\bimod$, where $(S;Q)$ and $(\overline{S};\overline{Q})$ are free. For more details on free resolutions in $\bimod$ we refer to \cite[\S 1.1]{DKNP1}, and we recall in particular that reflexive coequalisers in $\bimod$ are computed on the underlying pairs of sets.

\begin{defn}\label{def:rel}
Let $M$ be an $R$-bimodule and $\xymatrix@C=12pt{(\overline{S};\overline{Q})\ar@<.5ex>[r]^{f}\ar@<-.5ex>[r]_{g}&(S;Q)\ar[l]\ar@{->>}[r]^-{\epsilon}&(R;M)}$ a free resolution. We let $\mathcal{R}$ be the equivalence relation on $\prod_{i=0}^{n-1}\tens{R}{M}{i}$ generated by 
\[
a\sim b \Leftrightarrow \left\{
\begin{tabular}{c}
there exist $q$ and $u$ in $\prod_{i=0}^{n-1}\tens{S}{Q}{i}$ and $z$ in $\prod_{i=0}^{n-1}\tens{\overline{S}}{\overline{Q}}{i}$ such that:
\\
$\epsilon_* q=a$ and $\epsilon_* u=b$ in $\prod_{i=0}^{n-1}\tens{R}{M}{i}$,
\\
$w(q)= w(f_*(z))$ and $w(u)= w(g_*(z))$ in $\prod_{j=0}^{n-1}(\cyctens{S}{Q}{j})^{C_{p^j}}$
.
\end{tabular}\right.
\]
We denote the orbits of this relation by $W_{n,p}(R;M):=(\prod_{i=0}^{n-1}\tens{R}{M}{i})/\mathcal{R}$.
\end{defn}

\begin{prop}\label{Wwelldef}
The equivalence relation $\mathcal{R}$ is independent of the choice of free resolution. Every bimodule homomorphism $f\colon M\to M'$ induces a map
\[
f_*\colon W_{n,p}(R;M)\longrightarrow  W_{n,p}(R';M')
\]
defined as the quotient of the product map $\prod_{i=0}^{n-1}f^{\otimes_R p^i}$, making $W_{n,p}\colon \bimod\to Set$ into a functor. This functor commutes with  reflexive coequalisers, and in particular a free resolution of $(R;M)$ induces a reflexive coequaliser of sets
\[
\xymatrix{W_{n,p}(\overline{S};\overline{Q})\ar@<.5ex>[r]\ar@<-.5ex>[r]&W_{n,p}(S;Q)\ar[l]\ar@{->>}[r]&W_{n,p}(R;M)}.
\]
The ghost map $w$ descends to a natural transformation  $w\colon W_{n,p}(R;M)\to \prod_{j=0}^{n-1}(\cyctens{R}{M}{j})^{C_{p^j}}$,
which is injective when $(R;M)$ is free.
 \end{prop}

 \begin{proof}
Let $\xymatrix@C=10pt{(\overline{S};\overline{Q})\ar@<.5ex>[r]\ar@<-.5ex>[r]&(S;Q)\ar[l]\ar@{->>}[r]&(R;M)}$ and $\xymatrix@C=10pt{(\overline{S'};\overline{Q'})\ar@<.5ex>[r]\ar@<-.5ex>[r]&(S';Q')\ar[l]\ar@{->>}[r]&(R;M)}$ be two free resolutions of $(R;M)$. We claim that there are vertical arrows
 \[
\xymatrix{(\overline{S'};\overline{Q'})\ar[d]_k\ar@<.5ex>[r]^-{f'}\ar@<-.5ex>[r]_-{g'}&(S';Q')\ar[d]^h\ar@{->>}[r]^-{\epsilon'}&(R;M)\ar@{=}[d]
\\
(\overline{S};\overline{Q})\ar@<.5ex>[r]^-{f}\ar@<-.5ex>[r]_-{g}&(S;Q)\ar@{->>}[r]_-{\epsilon}&(R;M)
}
 \]
such that $hf'=fk$, $hg'=gk$ and $\epsilon h=\epsilon'$. By applying the functor $ \prod_{i=0}^{n-1}\tens{(-)}{(-)}{i}$ to this diagram this immediately implies that the equivalence relation induced by the first resolution is coarser than the second. By reversing the roles of the resolutions the two relations are equal. Let us write $(S';Q')=F(X,Y)$ as the free bimodule generated by a pair of sets $(X,Y)$. That is, $S'=\Z\{X\}$ is the free ring on $X$ and $Q'=(\Z\{X\}\otimes \Z\{X\}^{op})(Y)$ is the free $S'$-bimodule on $Y$.
We define $h$ as the adjoint of the map of pairs of sets
\[
h\colon (X,Y)\stackrel{\epsilon'}{\longrightarrow} (UR,UM)\stackrel{t}{\longrightarrow}(US,UQ)
\]
where $t$ is a section for the map $U\epsilon$ in the category $Set\times Set$. In order to define $k$,  let us write $(\overline{S'};\overline{Q'})=F(Z,W)$ for some sets $(Z,W)$, and define the adjoint map $k\colon (Z,W)\to U(\overline{S};\overline{Q})$ on a pair of elements $(z,w)$ as follows.
Since $f, g$ and their common section are additive (both on bimodules and underlying rings), the relation on $Q$
\[
q\sim q'\mbox{\ if there is $\overline{q}\in \overline{Q}$ such that \ }f(\overline{q})=q\ \mbox{and}\ g(\overline{q})=q'
\]
whose quotient is $M$ is already an equivalence relation, and similarly for the ring components. Since $\epsilon h f'(z,w)=\epsilon h g'(z,w)$ and $(R;M)$ is the coequaliser of the second resolution, there is an element $(\overline{s},\overline{q})\in (\overline{S};\overline{Q})$ such that $f(\overline{s},\overline{q})=h f'(z,w)$ and $g(\overline{s},\overline{q})=h g'(z,w)$, and we set $k(z,w):=(\overline{s},\overline{q})$. By construction the diagram above commutes, and this concludes the proof of the independence on the free resolution.

Every bimodule $(R;M)$ has a canonical free resolution provided by the free-forgetful adjunction to $\Set\times \Set$, given by
\[
\xymatrix@C=10pt{FUFU(R;M)\ar@<.5ex>[r]\ar@<-.5ex>[r]&FU(R;M)\ar[l]\ar@{->>}[r]&(R;M)}
\]
(see \cite[\S1.1]{DKNP1}), which is functorial in $(R;M)$. Using this resolution to compute the quotient it is immediate to verify that the tensor power of a morphism is well-defined on the quotient, and therefore that $W_{n,p}$ is a functor. If $(R;M)$ is free, by computing the quotient using the constant resolution we see that $W_{n,p}(R;M)$ is exactly the quotient of the product which makes the ghost map injective. It follows that the $w_j$ are well-defined maps out of $W_{n,p}(R;M)$ for free $(R;M)$, and that they define a natural transformations on free modules.

Let us show that $W_{n,p}$ of a free resolution is a reflexive coequaliser. This will in particular imply that $w_j$ descends to a well-defined natural transformation on $W_{n,p}(R;M)$ for all $(R;M)$. Given a general free resolution as above,
 we choose a common section $s\colon S\to \overline{S}$, and we regard $\overline{Q}$ as an $S$-bimodule via this map. Since reflexive coequalisers commute with tensor powers and with products of $S$-bimodules, we obtain a reflexive coequaliser of  abelian groups
\[
\xymatrix{\prod_{i=0}^{n-1}\tens{S}{\overline{Q}}{i}\ar@<.5ex>[r]\ar@<-.5ex>[r]&\prod_{i=0}^{n-1}\tens{S}{Q}{i}\ar[l]\ar@{->>}[r]&\prod_{i=0}^{n-1}\tens{S}{M}{i}=\prod_{i=0}^{n-1}\tens{R}{M}{i}},
\]
where $M$ is regarded as an $S$-bimodule via the surjection $S\to R$. Since $\tens{S}{\overline{Q}}{i}\to \tens{\overline{S}}{\overline{Q}}{i}$ is surjective the diagram
 \[
\xymatrix{\prod_{i=0}^{n-1}\tens{\overline{S}}{\overline{Q}}{i}\ar@<.5ex>[r]\ar@<-.5ex>[r]&\prod_{i=0}^{n-1}\tens{S}{Q}{i}\ar[l]\ar@{->>}[r]&\prod_{i=0}^{n-1}\tens{R}{M}{i}}
\]
 is also a coequaliser of abelian groups, and therefore of sets. By modding out the relation which makes the ghost map injective from the first two sets, we obtain a commutative diagram
\[
\xymatrix@R=12pt{\prod_{i=0}^{n-1}\tens{\overline{S}}{\overline{Q}}{i}\ar@<.5ex>[r]\ar@<-.5ex>[r]\ar@{->>}[d]&\prod_{i=0}^{n-1}\tens{S}{Q}{i}\ar[l]\ar@{->>}[r]\ar@{->>}[d]&\prod_{i=0}^{n-1}\tens{R}{M}{i}\ar@{->>}[d]
\\
W_{n,p}(\overline{S};\overline{Q})\ar@<.5ex>[r]\ar@<-.5ex>[r]&W_{n,p}(S;Q)\ar[l]\ar@{->>}[r]&C
}
\]
 where $C$ is the coequaliser of the bottom row. By definition, $W_{n,p}(R;M)$ is the quotient of $\prod_{i=0}^{n-1}\tens{R}{M}{i}$ by the equivalence relation which makes the right vertical map injective, and thus bijective.

Now let $\xymatrix@C=10pt{(T;P)\ar@<.5ex>[r]\ar@<-.5ex>[r]&(S;N)\ar[l]\ar@{->>}[r]&(R;M)}$ be a reflexive coequaliser. By applying the functorial free resolution given by the free forgetful adjunction we obtain a commutative diagram
\[
\xymatrix@R=15pt{
FUFU(T;P)\ar@<.5ex>[r]\ar@<-.5ex>[r]\ar@<.5ex>[d]\ar@<-.5ex>[d]&FUFU(S;N)\ar@<.5ex>[d]\ar@<-.5ex>[d]\ar[l]\ar@{->>}[r]&FUFU(R;M)\ar@<.5ex>[d]\ar@<-.5ex>[d]
\\
FU(T;P)\ar@<.5ex>[r]\ar@<-.5ex>[r]\ar@{->>}[d]&FU(S;N)\ar[l]\ar@{->>}[r]\ar@{->>}[d]&FU(R;M)\ar@{->>}[d]
\\
(T;P)\ar@<.5ex>[r]\ar@<-.5ex>[r]&(S;N)\ar[l]\ar@{->>}[r]&(R;M)
}
\]
Since reflexive coequalisers of bimodules are computed on underlying sets, $FU$ commutes with reflexive coequalisers, hence the rows of this diagrams are reflexive coequalisers. By applying $W_{n,p}$ we obtain a diagram where all the columns and the upper two rows are coequalisers, by the previous argument. Since colimits commute with each other the bottom row must also be a coequaliser, concluding the proof.
 \end{proof}
 
\begin{example}\label{exrel}\
\begin{enumerate}
\item If $(R;M)$ is free, $W_{n,p}(R;M)$ is the quotient of $\prod_{i=0}^{n-1}\tens{R}{M}{i}$ by the smallest equivalence relation that makes $w$ into an injective map.
 
\item If $R$ is a commutative ring and $M=R$, we have that as sets $W_{n,p}(R;R)=R^{\times n}$.
If $R$ has no $p$-torsion, this is because the ghost map 
 \[w\colon R^{\times n}\cong \prod_{i=0}^{n-1}\tens{R}{R}{i}\longrightarrow \prod_{i=0}^{n-1}(\cyctens{R}{R}{j})^{C_{p^j}}\cong R^{\times n},\]
 which is given by the usual Witt polynomials,
    is already injective. In general, one can resolve $(R;R)$ by the bimodules $(\Z[R];\Z[R])$ and  $(\Z[\Z[R]];\Z[\Z[R]])$  given by the free commutative rings. Thus the isomorphism with the product in the free case induces an isomorphism on the coequaliser $W_{n,p}(R;R)=R^{\times n}$.
\item
The $0$-th ghost map induces a natural bijection $W_{1,p}(R;M)\cong M/[R,M]$. This is clear when $(R;M)$ is free, and in general it is the case since both $W_{1,p}(R;M)$ and $M/[R,M]$ commute with reflexive coequalisers. 
\item 
When $R$ is a commutative ring and $M$ is an $R$-module, we will see in Proposition \ref{2truncationcomm} below that there is a natural bijection 
\[W_{2,p}(R;M)\cong M\times (\tens{R}{M}{})_{C_p}.\]
%
\item 
If $R$ is commutative and torsion-free, and $M$ is a free $R$-module, there is a canonical bijection
\[W_{n,p}(R;M)\cong \prod_{i=0}^{n-1}(\tens{R}{M}{i})_{C_{p^i}},\]
but in general the higher truncations $W_{n,p}(R;M)$ are difficult to describe as sets (cf. Proposition \ref{Wasset}).
\end{enumerate}
\end{example}

Before stating the main Theorem of this section, we recall that the category $\bimod$ has a symmetric monoidal structure, defined by the componentwise tensor product over $\Z$. We endow the functors $\cyctens{R}{M}{j}$ from $\bimod$ to $C_{p^j}$-equivariant abelian groups with the canonical monoidal structure
\[
\cyctens{R}{M}{j}\otimes \cyctens{S}{N}{j}\cong \cyctens{R\otimes S}{(M\otimes N)}{j}
\]
defined from the shuffle permutations, and their fixed-points with the canonical induced lax-monoidal structure.

\begin{theorem}\label{existence}
The map $\gamma$ that sends $(m_0,m_1,\dots,m_n)$ to the equivalence class of $\prod_{i=0}^n(1-m_it^{p^i})$ descends to a bijection $\gamma\colon W_{n+1,p}(R;M)\stackrel{\cong}{\longrightarrow} W_{\langle p^n\rangle}(R;M)$ for every integer $n\geq 0$ and any prime $p$. The diagram
\[\xymatrix{
W_{n+1,p}(R;M)\ar[rr]_-{\cong}^-{\gamma} \ar[dr]_-w&&W_{\langle p^n\rangle}(R;M)\ar[dl]^-{\tlog_{\langle p^n\rangle}}
\\
&\prod_{j=0}^{n}(\cyctens{R}{M}{j})^{C_{p^j}}
}
\]
commutes, where $\tlog_{\langle p^n\rangle}$ is the logarithmic derivative of \cite[Propositions 1.18 and 1.41]{DKNP1}.
The abelian group structure and the lax symmetric monoidal structure on $W_{n+1,p}$ which correspond to those of $W_{\langle p^n\rangle}$ are the unique ones such that the ghost maps $w_j\colon W_{n+1,p}(R;M)\to (\cyctens{R}{M}{j})^{C_{p^j}}$ are additive symmetric monoidal transformations for all $0\leq j<n+1$.
\end{theorem}
\begin{proof}
First, we show that $\gamma\colon \prod_{i=0}^{n}\tens{R}{M}{i}\to W_{\langle p^n\rangle}(R;M)$ is surjective. Any power series $1+\sum_{k\geq 1}a_kt^k$ representing an element of $W(R;M)$ can be written uniquely as 
\[
1+\sum_{k\geq 1}a_kt^k=\prod_{i=1}^{\infty}(1-b_it^i)
\]
in the completed tensor algebra, with $b_i\in M^{\otimes_R i}$. Thus any element $x$ of the quotient $W_{\langle p^n\rangle}(R;M)$ can be represented by an element of the form $\prod_{i=0}^{n}(1-c_it^{p^i})$ with $c_i\in M^{\otimes_R p^i}$, that is $\gamma(c_0,\dots,c_n)=x$, which proves that $\gamma$ is surjective.

Now let us show that $\gamma$ descends to a well-defined isomorphism. Let us first assume that $(R;M)$ is free. In this case $W_{n+1,p}(R;M)$ is the quotient of $\prod_{i=0}^{n}\tens{R}{M}{i}$ by the relation that makes the ghost map injective. The ghost map factors as
\[w\colon \prod_{i=0}^{n}\tens{R}{M}{i}\stackrel{\gamma}{\longrightarrow}W_{\langle p^n\rangle}(R;M)\xrightarrow{\tlog_{\langle p^n\rangle}}\prod_{j=0}^{n}(\cyctens{R}{M}{j})^{C_{p^j}},\]
 since $\tlog_{\langle p^n\rangle}(1-m_{i}t^{p^i}) =  \tr_e^{C_{p^i}} m_{i} t^{p^i} + \tr_{C_p}^{C_{p^{i+1}}} m_{i}^p t^{p^{i+1}} + \ldots$, so that the coefficient of $t^{p^j}$ in $\tlog_{\langle p^n\rangle}\gamma(m_0,\dots,m_n)$ is given by $w_j(m_0,\dots,m_{n})$. Thus $\gamma$ descends to a well-defined injection $W_{n+1,p}(R;M)\to W_{\langle p^n\rangle}(R;M)$. It is therefore an isomorphism, and the diagram of \ref{existence} commutes.

In general, we choose a free resolution of $(R;M)$. Since $\gamma$ is an isomorphism for the free resolution, it induces an isomorphism on the coequalisers $\gamma\colon W_{n+1,p}(R;M)\to W_{\langle p^n\rangle}(R;M)$.

The ghost maps $w_j$ are additive and symmetric monoidal because $\tlog_{\langle p^n\rangle}$ is, by \cite{DKNP1}. To see that the additive structure and the symmetric monoidal structure are unique with this property, it is sufficient to see this on the subcategory of free bimodules since $W_{n+1,p}$ commutes with reflexive coequalisers. In this case the uniqueness follows by the injectivity of the ghost map.
\end{proof}

\begin{cor}\label{CompLars}
For every ring $R$, the abelian group $W_{\langle p^n\rangle}(R;R)$ is isomorphic to the $p$-typical $(n+1)$-truncated Witt vectors $W_{n+1}(R)$ of \cite{HesselholtncW}, naturally in $R$. If $R$ is moreover commutative, $W_{\langle p^n\rangle}(R;R)$ is isomorphic to the usual $p$-typical $(n+1)$-truncated Witt vectors of $R$ as a commutative ring.
\end{cor}

\begin{proof}
The claim for commutative rings follows from the characterisation of the ring structure in ghost components of \ref{existence} and the fact that in this case, as a set, $W_{n+1,p}(R;R)$ is $R^{\times n+1}$ by \ref{exrel}. In the non-commutative case we need to make sure that the quotient of $\prod_{i=0}^{n}R^{\otimes_R p^i}=R^{\times n+1}$ defining $W_{n+1,p}(R;R)$ agrees with the quotient defining $W_{n+1}(R)$ from \cite{HesselholtncWcorr}. When $R$ is free the projection $R^{\times n+1}\to W_{n+1}(R)$ descends to an injection
\[
W_{n+1,p}(R;R)\longrightarrow W_{n+1}(R)
\]
since the ghost map of $W_{n+1}(R)$ is injective \cite[1.3.7]{HesselholtncWcorr}. This is therefore an isomorphism for free rings, and it descends to an isomorphism in general since both sides commutes with reflexive coequalisers ($W_{n+1}(R)$ does as a consequence of the identification with $\TR^{n+1}_0(R)$ of \cite{HesselholtncWcorr}).
\end{proof}

\begin{rem} In \cite{Read} Read provides a more general construction of $G$-typical Witt vectors with coefficients for any profinite group $G$. For $G=C_{p^n}$ Read's construction specialises to the $(n+1)$-truncated $p$-typical Witt vectors with coefficients defined above. 
\end{rem}

\begin{prop}\label{2truncationcomm}
Let $R$ be a commutative ring and $M$ an $R$-module. Then the canonical projection $M\times\tens{R}{M}{}\to W_{2,p}(R;M)$ descends to a bijection
\[W_{2,p}(R;M)\cong M\times (\tens{R}{M}{})_{C_p}.\]
\end{prop}

\begin{proof}
Suppose first that $R$ is torsion-free and that $M$ is a free $R$-module. Let us first prove that the transfer map
\[
\tr_e^{C_p}\colon (\tens{R}{M}{})_{C_p}\longrightarrow (\tens{R}{M}{})^{C_p}
\]
is injective. The composition of the transfer with the projection $(\tens{R}{M}{})^{C_p}\to (\tens{R}{M}{})_{C_p}$ is multiplication by $p$.
By writing $M$ as the free $R$-module $R(X)$ on a set $X$, we see that the orbits
\[
R(X^{\times p})_{C_p}\cong R(X^{\times p}/C_p)
\]
also form a free $R$-module, and since $R$ has no torsion multiplication by $p$ is injective. By \cite[Proposition 1.41]{DKNP1} $\tlog_{\langle p\rangle}$ is injective, and by \ref{existence} so is the ghost map of $W_{2,p}(R;M)$. Now consider the commutative square of sets
\[
\xymatrix{
M\times \tens{R}{M}{}\ar[r]\ar[d]& M\times (\tens{R}{M}{})_{C_p}\ar[d]^-w
\\
W_{2,p}(R;M)\ar@{-->}[ur]\ar[r]_-w&M\times (\tens{R}{M}{})^{C_p}
}
\]
where the unlabelled maps are the projections, the lower map $w$ is injective by the argument above, and the vertical $w$ is injective since the transfer is. The dashed arrow exists by the injectivity of the right vertical map and surjectivity of the left vertical map. It is surjective by the surjectivity of the top horizontal map, and injective by the injectivity of the lower horizontal map.

Without freeness assumptions on $R$ and $M$, the dashed map induces a bijection since both $W_{2,p}(R;M)$ and $M\times (\tens{R}{M}{})_{C_p}$ commute with reflexive coequalisers.
\end{proof}

\begin{rem}
In general there is no product decomposition analogous to Proposition \ref{2truncationcomm} for $W_{n,p}(R;M)$ for $n\geq 3$, since in this case the right vertical map $w$ of the diagram of the proof is generally not injective (see however Proposition \ref{Wasset} for the free case).
\end{rem}

In the rest of the section we try to get a feeling of this construction by describing explicitly the lower components of the addition and the symmetric monoidal structure of $W_{n+1,p}$ of Theorem \ref{existence}. As for the classical Witt vectors, there is no closed formulas for the components of the sum of two sequences $a=(a_0,a_1,\dots, a_n)$ and $b=(b_0,b_1,\dots, b_n)$. The characterisation in terms of ghost components of Theorem \ref{existence} gives however an inductive procedure for calculating them.
The examples below in particular identify $W_{2,2}(\Z;A)$ of a commutative ring $A$ with the ``$2$-truncated non-commutative ring of Witt vectors'' $W^{\otimes}_2(A)$ of \cite{THRmodels}.

\begin{example} \label{ex:solid}
Suppose that $R$ is commutative and $M$ is an $R$-module. Under the canonical bijection $W_{2,p}(R;M)\cong M\times (\tens{R}{M}{})_{C_p}$ of Proposition \ref{2truncationcomm},
we find that
\begin{align*}
a+b&=(a_0+b_0,a_1+b_1-\sum_{\{\emptyset\neq U\subsetneq p\}/C_{p}}t^{U}_1(a_0,b_0)\otimes\dots\otimes t^{U}_p(a_0,b_0))
\end{align*}
where the sum runs through the orbits of the standard $C_p$-action on the set of proper non-empty subsets of the $p$-elements set, and $t^{U}_i(a_0,b_0)=a_0$ if $i\in U$, and $t^{U}_i(a_0,b_0)=b_0$ otherwise. For example for the primes $p=2,3$ these are respectively 
\begin{align*}
a+b&=(a_0+b_0,a_1+b_1-a_0\otimes b_0)
\\
a+b&=(a_0+b_0,a_1+b_1-a_0\otimes b_0\otimes b_0-b_0\otimes a_0\otimes a_0).
\end{align*}
This expression is not well-defined in $M\times \tens{R}{M}{}$, as it requires a choice of orbit representatives of $(\tens{R}{M}{})_{C_p}$.
We observe the resemblance with the universal polynomials for the sum of the usual Witt vectors. In fact these are the usual universal polynomials when $M=R$.
\end{example}

\begin{example}\label{ex:multiplication} Let $R$ be a commutative ring and $M$ an $R$-algebra. The first two components of the product of $a=(a_0,a_1)$ and $b=(b_0,b_1)$ in $W_{2,p}(R;M)\cong M\times (\tens{R}{M}{})_{C_p}$
are
\begin{align*}
a\cdot b&=(a_0\cdot b_0,a^{\otimes p}_0\cdot b_1+a_1\cdot b^{\otimes p}_0+\sum_{\sigma\in C_p}a_1\cdot\sigma(b_1)).
\end{align*}
The term $a^{\otimes p}_0\cdot b_1\in (M^{\otimes_R p})_{C_p}$ is clearly independent on the choice of orbit representative for $b_1$, and similarly for $a_1\cdot b^{\otimes p}_0$. The sum over $C_p$ is clearly independent on the choice of representative for $b_1$, and if we choose a different representative $\tau(a_1)$ for some $\tau\in C_p$ we have
\[
\sum_{\sigma\in C_p}\tau(a_1)\cdot\sigma(b_1)=\tau(\sum_{\sigma\in C_p}a_1\cdot\tau^{-1}\sigma(b_1))=\tau(\sum_{\sigma\in C_p}a_1\cdot\sigma(b_1))=\sum_{\sigma\in C_p}a_1\cdot\sigma(b_1)
\] 
in the set of coinvariants $(M^{\otimes_R p})_{C_p}$. 
Thus this expression is well-defined in the coinvariants.  It is moreover not difficult to see that the last sum is symmetric in $\underline{a}$ and $\underline{b}$.
In particular we see directly that $W_{2,p}(R;M)$ is a commutative ring when $M$ is commutative.
\end{example}

\subsection{$p$-typical operators}\label{secop}

In this section we describe the operators on the truncated Witt vectors under the isomorphism of Theorem \ref{existence}, and investigate some of their properties.
Under the isomorphism of Theorem \ref{existence} the Witt vectors operators of \cite[\S 1.5]{DKNP1} take the form
\[
\begin{array}{ll}
F=F_p\colon W_{n+1,p}(R;M)\longrightarrow W_{n,p}(R;\tens{R}{M}{})&\hspace{1cm} V=V_p\colon W_{n,p}(R;\tens{R}{M}{})\longrightarrow W_{n+1,p}(R;M)
\\
\\
R\colon W_{n+1,p}(R;M)\longrightarrow W_{n,p}(R;M) &\hspace{1cm} \tau\colon M\longrightarrow W_{n,p}(R;M)
\\
\\
\sigma\colon W_{n,p}(R;M^{\otimes_Rk})\longrightarrow W_{n,p}(R;M^{\otimes_Rk})
\end{array}
\]
It is not difficult to see that the maps $V,R$ and $\tau$ can be described on representatives by
\begin{align*}
V(x_0,\dots,x_{n-1})&=(0,x_0,\dots,x_{n-1})
\\
R(m_0,\dots,m_{n})&=(m_0,\dots,m_{n-1})
\\
\tau(m)&=(m,0,\dots,0).
\end{align*} 
The Frobenius and the cyclic action do not however admit closed formulas on representatives. Their lowest components are computed in the following examples over a commutative base ring.

\begin{example}
Suppose that $R$ is commutative and that $M$ is an $R$-module. Under the bijection $W_{2,p}(R;M)\cong M\times (\tens{R}{M}{})_{C_p}$,  $F$ and $\sigma$ are described as follows:
\begin{enumerate}
\item
Let $(a_0,a_1,a_2)$ represent an element $a\in W_{3,p}(R;M)$. The element $F(a)\in M^{\otimes_R p}\times (M^{\otimes_R p^2})_{C_p}$ is represented by a pair of the form $(w_{1}(a_0,a_1),x_1)$. For $p=2$ the element $x_1\in M^{\otimes_R 4}$ is given by
\[
x_1=a_2+\sigma_2a_2-a_1\otimes(\sigma_1a_1)-a^{\otimes 2}_0\otimes (\tr^{C_2}_ea_1)-z(a_1),
\]
where $\sigma_i$ generates $C_{2^i}$, and $z(a_1)$ is a certain element of $M^{\otimes_R 4}$ such that 
\[(\sigma_1 a_1)^{\otimes 2}-\sigma_2(a^{\otimes 2}_1)=\tr^{C_2}_e(z(a_1)).\]
When $M$ is free this difference can be uniquely expressed as a transfer, and $z(a_1)$ is well-defined modulo $C_2$-coinvariants. In general, one needs to calculate $z(q_1)$ where $q_1$ is a lift of $a_1$ to a free resolution $\epsilon \colon Q\twoheadrightarrow M$, and $z(a_1)=\epsilon z(q_1)$.
\item For every $x_0\in M^{\otimes_Rk}$ the difference $(\sigma_k x_0)^{\otimes p}-\sigma_{kp}( x_0^{\otimes p})$ is in the image of the transfer $\tr_{e}^{C_p}\colon M^{\otimes_R kp}\to (M^{\otimes_R kp})^{C_p}$, where $\sigma_n$ generates $M^{\otimes_R n}$. This difference is zero on elementary tensors, and on their sums it is a transfer by the binomial formula. When $M$ is free this transferred term is unique, and 
\[
\sigma(x_0,x_1):=(\sigma_kx_0,\sigma_{kp}x_1-(\tr_{e}^{C_p})^{-1}((\sigma_k x_0)^{\otimes p}-\sigma_{kp}( x_0^{\otimes p}))).
\]
When $M$ is not free, one chooses a preimage of the transferred term in a free resolution of $M$, and then uses its image in $M$ in the same formula.
\end{enumerate}
\end{example}

The Verschiebung and Frobenius maps are equivariant with respect to the cyclic action of $C_p$, and in fact their iterations are invariant by the action of the higher order cyclic groups. In particular they define maps
\[
V^k\colon W_{n,p}(R;\tens{R}{M}{k})_{C_{p^k}}\to W_{n+k,p}(R;M) \ \ \ \  \  , \ \ \  \ F^k\colon W_{n+k,p}(R;M)\to W_{n,p}(R;\tens{R}{M}{k})^{C_{p^k}}.
\]

The cokernel of the Verschiebung is in fact an iteration of $R$, since by  \cite[Proposition 1.43]{DKNP1}  there are exact sequences
\[
W_{n,p}(R;\tens{R}{M}{k})_{C_{p^k}}\stackrel{V^k}{\longrightarrow} W_{n+k,p}(R;M)\stackrel{R^n}{\longrightarrow} W_{k,p}(R;M)\to 0
\]
for every $n,k\geq 1$. Moreover $V^k$ is injective when  $(\cyctens{R}{M}{l})_{C_{p^{l}}}$ has no $p$-torsion for all $k\leq l\leq k+n-1$, by \cite[Proposition 1.44]{DKNP1} (for example if $(R;M)$ is free), but not in general (see \cite{HesselholtncWcorr} for a counterexample where $M=R$ with $R$ non-commutative). We now explore some consequences  of the existence of these exact sequences.

Let us denote by
\[
\tau^k\colon M^{\otimes_R p^k}\longrightarrow W_{n,p}(R;M^{\otimes_R p^k})
\]
the Teichm\"uller map for the $R$-bimodule $M^{\otimes_R p^k}$.
When $R=\Z$ and $M=\Z(X)$ is the free abelian group on a set $X$, the ghost map of the Witt vectors $W_{\langle p^n\rangle}(\Z;\Z(X))$ is injective with target a free abelian group, and thus $W_{\langle p^n\rangle}(\Z;\Z(X))$ is also free abelian. In the following we describe a basis, analogous to the basis for the usual Witt vectors of $\Z$ from \cite[Proposition 1.6]{LarsBig}.

\begin{prop}\label{basis}
Let $\Z(X)$ be the free abelian group on a set $X$. Then there is an isomorphism of abelian groups
\[
W_{\langle p^n\rangle}(\Z;\Z(X))\cong \bigoplus_{i=0}^n \  \bigoplus_{x\in X^{\times p^i}/C_{p^i}}\Z\cdot V^i\tau^{n-i}(x_1\otimes\dots\otimes x_{p^i}).
\]
Under this isomorphism, the monoidal structure map $\star\colon W_{\langle p^n\rangle}(\Z;\Z(X))\otimes W_{\langle p^n\rangle}(\Z;\Z(Y))\to W_{\langle p^n\rangle}(\Z;\Z(X)\otimes \Z(Y))$ multiplies two generators by the formula
\[
V^i\tau^{n-i}(x_1\otimes\dots\otimes x_{p^i})\star V^j\tau^{n-j}(y_1\otimes\dots\otimes y_{p^j})=
  \sum_{\sigma\in C_{p^i}} V^j(\tau^{n-j}((\sigma (\otimes_{l=1}^{p^i}x_l))^{\otimes p^{j-i}}\otimes (\otimes_{h=1}^{p^j}y_h)))
\]
if $i\leq j$, and
\[
V^i\tau^{n-i}(x_1\otimes\dots\otimes x_{p^i})\star V^j\tau^{n-j}(y_1\otimes\dots\otimes y_{p^j})=
  \sum_{\sigma\in C_{p^j}} V^i(\tau^{n-i}((\otimes_{l=1}^{p^i}x_l)\otimes (\sigma (\otimes_{h=1}^{p^j}y_h))^{\otimes p^{i-j}})))
\]
if $j\leq i$.
\end{prop}
\begin{proof}
There is a clear map from the free abelian group on the right to the Witt-vectors which takes the sum. For $n=0$, the map is an isomorphism since it is the canonical isomorphism $W_{\langle 1\rangle}(\Z;\Z(X))\cong \Z(X)$. Suppose by induction that the map is an isomorphism for $n-1$, and consider the map of exact sequences
\[
\xymatrix@C=8pt{
0\ar[r]&\Z(X^{\times p^n}/C_{p^n})\ar[r]\ar[d]_-{\cong}&\displaystyle \bigoplus_{i=0}^n \bigoplus_{x\in X^{\times p^i}/C_{p^i}}\Z\cdot V^i\tau^{n-i}(\otimes_{l=1}^{p^i}x_l)\ar[r]\ar[d]& \displaystyle \bigoplus_{i=0}^{n-1} \bigoplus_{x\in X^{\times p^i}/C_{p^i}}\Z\cdot V^i\tau^{n-1-i}(\otimes_{l=1}^{p^i}x_l)\ar[r]\ar[d]^{\cong}& 0
\\
0\ar[r]&(\Z(X)^{\otimes p^n})_{C_{p^n}}\ar[r]^-{V^n}&W_{\langle p^n\rangle}(\Z;\Z(X))\ar[r]^R& W_{\langle p^{n-1}\rangle}(\Z;\Z(X))\ar[r]& 0
}
\]
where the bottom row is exact since $(\Z;\Z(X))$ is free. The maps of the top row are respectively the projection and the inclusion of the summand $i=n$, and the left square commutes by definition since the left vertical map sends the orbit of $(x_1,\dots, x_{p^n})$ to $x_1\otimes\dots\otimes x_{p^n}$. The right square commutes since $R$ is additive, $R\tau^{k}=\tau^{k-1}$ (with the convention $\tau^0=\id$), $RV^i=V^iR$ for $i>0$, and $RV=0$. Thus the middle map is also an isomorphism.

Let us now determine the multiplication of the generators. We show that case $i\leq j$, the other is similar. We have that
\begin{align*}
V^i\tau^{n-i}(\otimes_{l=1}^{p^i}x_l)\star V^j\tau^{n-j}(\otimes_{h=1}^{p^j}y_h)&=V^j((F^jV^i\tau^{n-i}(\otimes_{l=1}^{p^i}x_l))\star(\tau^{n-j}(\otimes_{h=1}^{p^j}y_h)))
\\&=V^j((F^{j-i}\sum_{\sigma\in C_{p^n}/C_{p^{n-i}}}\sigma \tau^{n-i}(\otimes_{l=1}^{p^i}x_l))\star(\tau^{n-j}(\otimes_{h=1}^{p^j}y_h)))
  \\&=\sum_{\sigma\in C_{p^i}} V^j((F^{j-i} \tau^{n-i}(\sigma (\otimes_{l=1}^{p^i}x_l)))\star(\tau^{n-j}(\otimes_{h=1}^{p^j}y_h)))
  \\&=\sum_{\sigma\in C_{p^i}} V^j((\tau^{n-j}((\sigma (\otimes_{l=1}^{p^i}x_l))^{\otimes p^{j-i}}))\star(\tau^{n-j}(\otimes_{h=1}^{p^j}y_h)))
  \\&=\sum_{\sigma\in C_{p^i}} V^j(\tau^{n-j}((\sigma (\otimes_{l=1}^{p^i}x_l))^{\otimes p^{j-i}}\otimes (\otimes_{h=1}^{p^j}y_h)))
\end{align*}
where the first equality holds by Frobenius reciprocity and the last by the monoidality of $\tau$, from \cite[Proposition 1.27]{DKNP1}. The second equality is the double-coset formula from \cite[Proposition 1.32, 5)]{DKNP1}, and the third follows from the equivariance of $\tau$ of \cite[Proposition 1.25]{DKNP1}. Finally, the fourth equality is the fact that for every $k\leq l$, bimodule $(R;M)$, and $m\in M$
\[
F^k\tau^l(m)=\tau^{l-k}(m^{\otimes p^k})
\]
in $W_{\langle p^{l-k}\rangle}(R;M^{\otimes_Rp^k})$, which by the standard resolution argument can be verified in ghost components, where
\[
w_jF^k\tau^l(m)=w_{j+k}(m,0,\dots,0)=m^{\otimes p^{j+k}}=(m^{\otimes p^k})^{\otimes p^j}=w_j(\tau^{l-k}(m^{\otimes p^k})).\qedhere
\] 
\end{proof}

\begin{prop}\label{singleVinj}
A single Verschiebung  $V\colon W_{n,p}(R;\tens{R}{M}{})_{C_{p^{}}}\to W_{n+1,p}(R;M)$ is injective when $R$ is a commutative ring and $M$ is an $R$-module (considered as a bimodule with the same left and right action). 
\end{prop}

\begin{proof} 
By resolving in the subcategory of pairs of commutative rings and modules, we can construct a reflexive coequaliser  $\xymatrix@C=10pt{(\overline{S};\overline{Q})\ar@<.5ex>[r]\ar@<-.5ex>[r]&(S;Q)\ar[l]\ar@{->>}[r]&(R;M)}$ were $S$ and $\overline{S}$ are free commutative rings, $Q$ is a free $S$-module and $\overline{Q}$ is a free $\overline{S}$-module. Since $W_{n,p}$ commutes with reflexive coequalisers we obtain a commutative diagram
\[
\xymatrix{
W_{n,p}(\overline{S};\tens{\overline{S}}{\overline{Q}}{})_{C_p}\ar@<.5ex>[r]^-{f_*}\ar@{>->}[d]^V\ar@<-.5ex>[r]_-{g_*}&W_{n,p}(S;\tens{S}{Q}{})_{C_p}\ar@{>->}[d]^V\ar[l]\ar@{->>}[r]&W_{n,p}(R;\tens{R}{M}{})_{C_p}\ar[d]^V
\\
W_{n+1,p}(\overline{S};\overline{Q})\ar@<.5ex>[r]^{f_*}\ar@<-.5ex>[r]_{g_*}&W_{n+1,p}(S;Q)\ar[l]\ar@{->>}[r]&W_{n+1,p}(R;M)
}
\]
where the rows are reflexive coequalisers. The middle vertical map is injective by \cite[Proposition 1.44]{DKNP1}, since for $Q=S(X)=\oplus_XS$ the free module on a set $X$ we have that 
\[(\cyctens{S}{(S(X))}{i})_{C_{p^i}}\cong (\oplus_{X^{\times p^i}}S)_{C_{p^i}}=S(X^{\times p^i}/C_{p^i})\]
is a free $S$-module, which is torsion-free since $S$ is torsion-free. The same argument applies to the left vertical map.
Let $x=(x_0,\dots,x_{n-1})$ represent an element of $W_{n,p}(S;\tens{S}{Q}{})_{C_p}$ such that $V(x)$ is zero in the coequaliser $W_{n+1,p}(R;M)$. Then there is an element $q=(q_0,\dots,q_{n})$ in $\prod_{i=0}^{n}\tens{\overline{S}}{\overline{Q}}{i}$ such that $wf_* q=w V(x)$ and $wg_* q=0$ in $\prod_{j=0}^{n}(\tens{S}{Q}{j})^{C_{p^j}}$. In particular since $Q$ is an $S$-module 
\[g(q_0)=w_0g_* q=0=w_0f_* q=f(q_0)\]
as elements of $Q/[S,Q]=Q$, and therefore $g(q_0)^{\otimes p^{i}}=f(q_0)^{\otimes p^{i}}=0$ for all $i\geq 0$. It follows that
\[
wf_* (0,q_1,\dots,q_n)=w V(x) \ \ \ \ \ \ \mbox{and} \  \  \ \ \ \ \ \  wg_* (0,q_1,\dots,q_n)=0,
\]
or in other words that
\[
wV f_* (q_1,\dots,q_n)=w V(x) \ \ \ \ \ \ \mbox{and} \  \  \ \ \ \ \ \  wVg_* (q_1,\dots,q_n)=0.
\]
Since $V$ and $w$ on the Witt vectors of $Q$ are injective we have that $ f_* (q_1,\dots,q_n)=x$ and $g_* (q_1,\dots,q_n)=0$ in $W_{n,p}(S;\tens{S}{Q}{})_{C_p}$, that is that $x$ and zero define the same class in the coequaliser $W_{n,p}(R;\tens{R}{M}{})_{C_p}$.
\end{proof}

\begin{prop}\label{Wasset}
A choice of sections of the quotients $\tens{R}{M}{i}\to(\cyctens{R}{M}{i})_{C_{p^i}}/\ker V^i$ determines a bijection
\[
W_{n+1,p}(R;M)\cong \prod_{i=0}^{n}\big(((\cyctens{R}{M}{i})_{C_{p^i}})/\ker V^i\big),
\]
where $V^i$ is the iterated Verschiebung $V^i\colon (\cyctens{R}{M}{i})_{C_{p^i}}\to W_{i+1,p}(R;M)$. In particular when $M=R$, a choice of section for $R\to R/[R,R]$ determines a bijection $W_{n+1,p}(R)\cong \prod_{i=0}^nR/[R,R]$ when $R/[R,R]$ has no $p$-power torsion, as in \cite{HesselholtncWcorr}. If $R$ is commutative with no $p$-power torsion, and $M$ is a free $R$-module, there is a canonical bijection
\[
W_{n+1,p}(R;M)\cong \prod_{i=0}^{n}(\tens{R}{M}{i})_{C_{p^i}}.
\]
\end{prop}
\begin{proof}
A choice of section of $M\to\cyctens{R}{M}{0}=W_{1,p}(R;M)$, followed by the map $\tau$, determines a splitting (as a map of sets) of the right-hang map of the exact sequence
\[
W_{n,p}(R;\tens{R}{M}{})_{C_{p}}\stackrel{V}{\longrightarrow} W_{n+1,p}(R;M)\stackrel{R^n}{\longrightarrow} \cyctens{R}{M}{0}\to 0
\]
from  \cite[Proposition 1.43]{DKNP1}, and thus a bijection
\[
W_{n+1,p}(R;M)\cong \cyctens{R}{M}{0}\times (W_{n,p}(R;\tens{R}{M}{})_{C_{p}})/\ker V.
\]
A straightforward diagram chase on the diagram with exact rows
\[
\xymatrix{
W_{n-i,p}(R;\tens{R}{M}{i+1})_{C_{p^{i+1}}}\ar@{=}[d]\ar[r]^-V& W_{n+1-i,p}(R;\tens{R}{M}{i})_{C_{p^i}}\ar[r]^-{R^{n-i}}\ar[d]^{V^i}& (\cyctens{R}{M}{i})_{C_{p^i}}
\ar[d]^{V^i}\ar[r]&0
\\
W_{n-i,p}(R;\tens{R}{M}{i+1})_{C_{p^{i+1}}}\ar[r]^-{V^{i+1}}& W_{n+1,p}(R;M)\ar[r]^-{R^{n-i}}&W_{i+1,p}(R;M)\ar[r]&0
}
\]
shows that the top sequence stays exact after quotienting by the kernel of $V^i$, for every $i=0,\dots,n$. 
A section of $\tens{R}{M}{i}\to (\cyctens{R}{M}{i})_{C_{p^i}}/\ker V^i$, followed by the map $\tau$, determines a splitting (as a map of sets) of the right-hang map of the top sequence quotiented by the kernel of $V^i$. 
Moreover the kernel of 
\[
V\colon W_{n-i,p}(R;\tens{R}{M}{i+1})_{C_{p^{i+1}}}\longrightarrow W_{n+1-i,p}(R;\tens{R}{M}{i})_{C_{p^i}}/\ker V^i
\]
is equal to the kernel of $V^{i+1}$.
Thus by downward induction on $i$, we obtain a sequence of bijections
\begin{align*}
W_{n+1,p}(R;M)&\cong \cyctens{R}{M}{0}\times (W_{n,p}(R;\tens{R}{M}{}_{C_{p}})/\ker V)
\\&\cong \cyctens{R}{M}{0}\times ((\cyctens{R}{M}{})_{C_p}/\ker V)\times (W_{n-1,p}(R;\tens{R}{M}{2})_{C_{p^2}}/\ker V^2)
\\&\cong\dots
\\&\cong \cyctens{R}{M}{0}\times ((\cyctens{R}{M}{})_{C_p}/\ker V)\times ((\cyctens{R}{M}{2})_{C_{p^2}}/\ker V^2)\times\dots \times  ((\cyctens{R}{M}{n})_{C_{p^n}}/\ker V^n).
\end{align*}

If $R$ is commutative without $p$-power torsion and $M=R(X)$ is a free $R$-module, $V$ is injective by \cite[Proposition 1.44]{DKNP1} and the  projection maps
\[
\tens{R}{M}{k}\cong R(X^{\times p^k})\to R((X^{\times p^k})_{C_{p^k}})\cong(\cyctens{R}{M}{k})_{C_{p^k}}
\]
have canonical sections, induced by the maps $(X^{\times p^k})_{C_{p^k}}\cong \hom(p^k/C_{p^k},X)\to \hom(p^k,X)=X^{\times p^k}$ which precompose with the quotient map $p^k\to p^k/C_{p^k}$ (here we denoted $p^k$ the set with $p^k$-elements).
\end{proof}

Using the maps $R$ we can also define
\[W_{\infty,p}(R;M):=\lim(W_{1,p}(R;M)\stackrel{R}{\longleftarrow} W_{2,p}(R;M)\stackrel{R}{\longleftarrow}\dots).\]
It follows from Theorem \ref{existence} and \cite[Lemma 1.37]{DKNP1} that $W_{\infty,p}(R;M) \cong W_{\langle p^{\infty} \rangle}(R;M)$, where $\langle p^{\infty} \rangle$ is the truncation set of all powers of $p$.

\begin{prop}
Let $R$ be a commutative $\F_p$-algebra and $M$ an $R$-algebra. Then $V^n(1)=p^n$ in $W_{n+m+1,p}(R;M)$ for all $n,m\geq 0$, and 
\[W_{\infty,p}(R;-)=\lim(W_{1,p}(R;-)\stackrel{R}{\longleftarrow} W_{2,p}(R;-)\stackrel{R}{\longleftarrow}\dots)\]
 takes commutative $R$-algebras to commutative rings of characteristic zero.
\end{prop}

\begin{proof}
Since $R$ is an $\F_p$-algebra there is a map of bimodules $\iota\colon (\F_p;\F_p)\to (R;M)$. By naturality of $V$ we have that
\[
V^n(1)=V^n(\iota(1))=\iota V^n(1)
\]
where the last Verschiebung $V^n$ is for the ring $\F_p$, and it sends $1$ to $p^n$.


The ring homomorphism $\eta\colon R\to M$ defines a map of bimodules $(R;M)\to (M;M)$, and the induced map sends $p^n$ in  $W_{\infty,p}(R;M)$ to $p^n$ in $W_{\infty,p}(M)$. When $M$ is commutative the Verschiebung of $W_{\infty,p}(M)$ is injective and therefore $V^n(1)=p^n\neq 0\in W_{\infty,p}(M)$. Thus it must be non-zero already in $W_{\infty,p}(R;M)$.
\end{proof}

\begin{rem} 
Classically $W_{\infty,p}(A)$ is $p$-complete when $A$ is a commutative semi-perfect $\F_p$-algebra. With coefficients, we do always have a natural isomorphism
\[
W_{\infty,p}(R;M)= \lim_{n\geq 1}W_{n,p}(R;M)\cong \lim_{n\geq 1} W_{\infty,p}(R;M)/V^{n}.
\]
In other words the $\Z$-module $W_{\infty,p}(R;M)$ is complete with respect to the sequence of submodules 
\[\im V\supset \im(V^2)\supset\dots\supset  \im(V^n)\supset\dots.\]

When $A$ is a semi-perfect commutative $\F_p$-algebra the image of $V^n$ on the classical Witt vectors is $p^nW_{\infty,p}(A)$, and the $V$-completeness statement  above shows that $W_{\infty,p}(A)$ is $p$-complete. This characterisation of $V$ follows from Frobenius reciprocity, the identity $V(1)=p$, and the fact that $F$ is surjective since for $\F_p$-algebras it is of the form 
\[
F(a_0,a_1,\dots)=(a_{0}^p,a_{1}^p,\dots).
\]
In pursuing a similar argument for the Frobenius of $W_{\infty,p}(\Z;A)$ for a ring $A$, one would need the $p$-th power map $(-)^{\otimes p}\colon A\to A^{\otimes p}$ to be additive. This is the case modulo the image of the transfer map $\tr^{C_p}_e\colon A^{\otimes p}\to A^{\otimes p}$, and therefore one would require that $\tr^{C_p}_e=0$. This is however the case only when $A=\F_p$ and this just recovers the standard completeness of $W(\F_p)$. 
Indeed since $\mu_p \tr^{C_p}_e 1=p\in A$ the condition that $\tr^{C_p}_e=0$ forces $A$ to be an $\F_p$-algebra. By choosing a basis of $A$ as an $\F_p$-vector space, we further see that  $A$ is of rank $1$. Thus $W_{\infty,p}(\Z;A)=W_{\infty,p}(\F_p;A)$ does not seem to be $p$-complete for any ring except for $A=\F_p$.
\end{rem}

When $M$ is a commutative $R$-algebra, we show below that the Teichm\"{u}ller map $\tau\colon M\to W_{n+1,p}(R;M)$ factors through multiplicative norm maps
\[
N\colon W_{n,p}(R;\tens{R}{M}{})\longrightarrow W_{n+1,p}(R;M).
\]
This map extends the norm map of the Witt vectors introduced by Angeltveit in \cite{Angeltveit} in the case where $R=M$.
Let us consider the multiplicative map $N_w\colon \prod_{j=0}^{n-1}(\tens{R}{M}{j+1})^{C_{p^j}}\to \prod_{j=0}^{n}(\tens{R}{M}{j})^{C_{p^j}}$
defined by
\[
N_w(y_0,y_1,\dots,y_{n-1}):=(\mu_py_0,\prod_{\sigma\in C_p}\sigma y_0,\prod_{\sigma\in C_{p^2}/C_p}\sigma y_1,\dots,\prod_{\sigma\in C_{p^n}/C_{p^{n-1}}}\sigma y_{n-1})
\]
where $\mu_p\colon \tens{R}{M}{}\to M$ is the multiplication map. Here we are using that, since $R$ is commutative and $M$ is an $R$-module, $\tens{R}{M}{j}=\cyctens{R}{M}{j}$ so that the product of cyclic tensors is well-defined, as well as the commutativity of $M$ so that this product is independent of the cyclic ordering.

\begin{prop}\label{N}
Let $M$ be a commutative $R$-algebra. There is a unique natural map of sets $N\colon W_{n,p}(R;\tens{R}{M}{})\to W_{n+1,p}(R;M)$, which we call the norm, such that $wN=N_ww$. It is multiplicative, unital, and it satisfies the identities
\[\begin{array}{ll}
RN&=NR
\\F N(x)&=\prod_{\sigma\in C_p}\sigma x
\end{array}\]
for every $x\in W_{n,p}(R;\tens{R}{M}{})$, where $\sigma$ is the Weyl action.
\end{prop}

\begin{proof}
It is sufficient to show that such a unique natural transformation exists on pairs $(R;M)$ where $R$ is a free commutative ring and $M$ is a free commutative $R$-algebra. The map $N$ for a general pair will then be defined as the reflexive coequaliser of the map induced on the Witt vectors of a free resolution (the fact that $N$ is not a ring homomorphism is not an issue since reflexive coequalisers of commutative rings are computed in sets). The map will be independent on the choice of resolution since different free resolutions can be compared by a map as in the proof of \ref{Wwelldef}.

Let us then suppose that $(R;M)$ is such a free pair and denote by $\hat{w}$ the ghost map of $W_{n,p}(R;\tens{R}{M}{})$. We remark that since the ghost map $w$ of $W_{n+1,p}(R;M)$ is injective the existence and uniqueness of $N$ follow if we can prove that  the image of $N_w\hat{w}$ is included in the image of $w$. In order to verify this we use a version of the Dwork lemma with coefficients which characterises the image of $w$. The proof of the lemma is technical and is deferred to the appendix \ref{Dwork}. It states that when $(R;M)$ is free (or more generally if it has an ``external Frobenius''), there are additive maps
\[
\phi_{j}\colon (\cyctens{R}{M}{j})^{C_{p^{j}}}\longrightarrow(\cyctens{R}{M}{j+1})^{C_{p^{j+1}}}
\]
such that a sequence $(b_0,b_1,\dots,b_{n})$ of $ \prod_{j=0}^{n}(\cyctens{R}{M}{j})^{C_{p^j}}$ lies in the image of the ghost map $w$ if and only if
\[
\phi_j(b_j)\equiv b_{j+1}\ \mbox{mod}\ \tr_{e}^{C_{p^{j+1}}}
\]
for every $0\leq j<n$,
where the congruence is modulo the image of the additive transfer map $\tr_{e}^{C_{p^{j+1}}}\colon \cyctens{R}{M}{j+1}\to (\cyctens{R}{M}{j+1})^{C_{p^{j+1}}}$.
Thus we need to verify that the sequence
\[(\mu_p\hat{w}_0,\prod_{\sigma\in C_p}\sigma \hat{w}_0,\prod_{\sigma\in C_{p^2}/C_p}\sigma \hat{w}_1,\dots,\prod_{\sigma\in C_{p^n}/C_{p^{n-1}}}\sigma \hat{w}_{n-1})\]
satisfies these congruences.  We start by separately analyzing the first congruence
\begin{align*}
\phi \mu_p \hat{w}_0x_0=\phi \mu_p x_0=\phi w_0\mu_p x_0\equiv w_1(\mu_p x_0)\equiv (\mu_p x_0)^{\otimes p}= \prod_{\sigma\in C_p}\sigma x_0\ \ \ \ \ \mbox{mod}\ \tr^{C_p}_e,
\end{align*}
where the third equality is from the Dwork Lemma, and the last one can be easily verified on elementary tensors. Now let us take $j\geq 1$. When $M$ is a free commutative $R$-algebra the maps $\phi_j$ are moreover multiplicative (see \ref{Frobpoly}). They are defined before taking invariants, and they satisfy $\phi_{j-1}\sigma_{j-1}=\sigma_j\phi_{j-1}$ where $\sigma_k$ denotes the action of a generator of $C_{p^k}$ (Lemma \ref{higherFrob} ). Moreover the maps
\[
\Phi_{j}:=\phi_{j+1}\colon \cyctens{R}{(M^{\otimes_R p})}{j}= \cyctens{R}{M}{j+1}\longrightarrow \cyctens{R}{M}{j+2}= \cyctens{R}{(M^{\otimes_R p})}{j+1}
\]
satisfy the Dwork Lemma for the free $R$-module $M^{\otimes_R p}$ (Lemma \ref{iterFrob}). We can now verify that
\begin{align*}
\phi_j \prod_{\sigma\in C_{p^j}/C_{p^{j-1}}}\sigma \hat{w}_{j-1}&= \prod_{l=1}^p\phi_j\sigma^{l}_j \hat{w}_{j-1}=  \prod_{l=1}^p\sigma^{l}_{j+1} \phi_j\hat{w}_{j-1}
=\prod_{\sigma\in C_{p^{j+1}}/C_{p^{j}}}\sigma \phi_j \hat{w}_{j-1}
\\&=\prod_{\sigma\in C_{p^{j+1}}/C_{p^{j}}}\sigma \Phi_{j-1} \hat{w}_{j-1}=\prod_{\sigma\in C_{p^{j+1}}/C_{p^{j}}}\sigma (\hat{w}_{j}+\tr^{C_{p^j}}_e)
\\&=(\prod_{\sigma\in C_{p^{j+1}}/C_{p^{j}}}\sigma \hat{w}_{j})+\tr^{C_{p^{j+1}}}_e
,
\end{align*}
and therefore that $N_w\hat{w}$ lands in the image of $w$. The  last equality follows from the Tambara reciprocity relations of the $C_{p^{j+1}}$-Tambara functor $\tens{R}{M}{j+1}$ (see e.g. \cite[Corollaries 2.6 and 2.9]{HillMazur}). It can also be verified directly as follows. For every subset $V\subset\{1,\dots,p\}$ and $1\leq l\leq p$ let us define $t^{V}_l:=\sigma^{l}_{j+1} \hat{w}_{j}$ if $l\notin V$, and  $t^{V}_l:=\sigma^{l}_{j+1}\tr^{C_{p^j}}_e$ if $l\in V$. Then
\begin{align*}
\prod_{\sigma\in C_{p^{j+1}}/C_{p^{j}}}\sigma (\hat{w}_{j}+\tr^{C_{p^j}}_e)&=\sum_{V\subset \{1,\dots,p\}}t^{V}_1t^{V}_2\dots t^{V}_p=(\prod_{\sigma\in C_{p^{j+1}}/C_{p^{j}}}\sigma \hat{w}_{j})+\sum_{\emptyset\neq V\subset \{1,\dots,p\}}t^{V}_1t^{V}_2\dots t^{V}_p
\\
&=(\prod_{\sigma\in C_{p^{j+1}}/C_{p^{j}}}\sigma \hat{w}_{j})+\sum_{l=1}^p\sum_{l\in V\subset \{1,\dots,p\}}t^{V}_1\dots t^{V}_{l-1}(\sigma^{l}_{j+1}\tr^{C_{p^j}}_e) t^{V}_{l+1}\dots t^{V}_p
\\&=(\prod_{\sigma\in C_{p^{j+1}}/C_{p^{j}}}\sigma\hat{w}_{j})+\sum_{l=1}^p\sum_{l\in V\subset \{1,\dots,p\}}\sigma^{l}_{j+1}\tr^{C_{p^j}}_e
=\prod_{\sigma\in C_{p^{j+1}}/C_{p^{j}}}\sigma (\hat{w}_{j})+\tr^{C_{p^{j+1}}}_e
\end{align*}
where the third equality holds by Frobenius reciprocity.

The relations between $N$, $R$, $F$ are easily verified in ghost components and are natural, and therefore they hold in Witt coordinates by the usual resolution argument. 
\end{proof}

\begin{rem}\label{TambaraStructure}
There are relations between $N$ and $V$ and $N$ and the sum which are difficult to express without the help of exponential diagrams. These are the Tambara reciprocity conditions for norms and transfers and norms and sums of \cite[(2.1)(v)]{Tambara} (See also \cite[Corollaries 2.6 and 2.9]{HillMazur}). They can be directly verified in ghost coordinates using Tambara reciprocity for the $C_{p^j}$-Tambara functors $(\tens{R}{M}{j})^{(-)}$. 
\end{rem}

\begin{prop}\label{tauandN}
When $M$ is a commutative $R$-algebra, the map $\tau\colon M\to W_{n+1,p}(R;M)$ agrees with the composite
\[
M\stackrel{u}{\longrightarrow} \tens{R}{M}{n}=W_{1,p}(R;\tens{R}{M}{n})\stackrel{N}{\longrightarrow}W_{2,p}(R;\tens{R}{M}{n-1})\stackrel{N}{\longrightarrow}\dots \stackrel{N}{\longrightarrow}W_{n+1,p}(R;M),
\]
where $u(m)=m\otimes 1\otimes\dots\otimes 1$.
\end{prop}

\begin{proof}
In ghost components the iterated norm $N$ sends $x\in \tens{R}{M}{n}$ to
\[
wN^n(x)=(\mu_{p}^n(x),\prod_{\sigma\in C_{p}}\sigma\mu_{p}^{n-1}(x),\prod_{\sigma\in C_{p^2}}\sigma\mu_{p}^{n-2}(x),\dots, \prod_{\sigma \in C_{p^n}}\sigma(x)),
\]
where $\mu^{n-i}_p$ is the composite of the multiplication maps
\[
M^{\otimes_Rp^n}\stackrel{\mu_p}{\longrightarrow} M^{\otimes_Rp^{n-1}}\stackrel{\mu_p}{\longrightarrow} M^{\otimes_Rp^{n-2}}\stackrel{\mu_p}{\longrightarrow} \dots \stackrel{\mu_p}{\longrightarrow} M^{\otimes_Rp^{i}}.
\]
After precomposing with $u\colon M\to \tens{R}{M}{n}$ this sends $a\in M$ to
\[
wN^n(a\otimes 1\otimes\dots\otimes 1)=(a,a^{\otimes p},a^{\otimes p^2},\dots, a^{\otimes p^n})=w(a,0,\dots,0),
\]
showing that the maps agree in ghost components. By the usual resolution argument they agree in Witt coordinates.
\end{proof}

 \subsection{The comparison with Kaledin's polynomial Witt vectors}\label{secKaledin}

Kaledin defines in \cite{KaledinWpoly} and \cite{KaledinncW} a functor $\overline{W}_n$ (denoted by $\widetilde{W}_n$ in \cite{KaledinWpoly}) of ``polynomial Witt vectors'' from the category of vector spaces over a perfect field $k$ of characteristic $p$ to the category of abelian groups. 
We show that on $k$-vector spaces our functor $W_{n,p}$ can be described as the cokernel of a transfer map, and use this to identify $W_{n,p}$ with $\overline{W}_n$.

For any $R$-bimodule $M$, we let $Q_{p^n}(R;M)$ denote the cokernel of the transfer map
\[
(\cyctens{R}{M}{n})_{C_{p^n}}\xrightarrow{\tr^{C_{p^n}}_e} (\cyctens{R}{M}{n})^{C_{p^n}} \twoheadrightarrow Q_{p^n}(R;M).
\]
The canonical lax symmetric monoidal structure of $(\cyctens{R}{M}{n})^{C_{p^n}}$ (as defined in \cite{DKNP1}) descends to a lax symmetric monoidal structure on the functor $Q_{p^n}$. We let $R[m]$ denote the $m$-torsion subgroup of $R$ for every integer $m$.

\begin{theorem}\label{Kaledin} Let $M$ be an $R$-bimodule. For every integer $n\geq 1$ there is a surjective natural lax symmetric monoidal transformation
\[
w_n\colon W_{n,p}(R;M/p)=W_{n,p}(R/p;M/p)\twoheadrightarrow Q_{p^n}(R;M).
\]
It is an isomorphism when $R$ is commutative, $R/p$ is perfect, $M$ is a free $R$-module, and the multiplication by $p^{l}$-map $p^{l}\colon R[p^{l+1}]\to R[p]$ is surjective for every $1\leq l\leq n-1$.
\end{theorem}

\begin{rem}
When $M=R$, this in particular states that for commutative  rings without $p$-power torsion and with perfect $R/p$, there is a ring isomorphism $W_{n,p}(R/p)\cong R/p^n$. For example  $W_{n,p}(\F_p)=\Z/p^n$.
\end{rem}

\begin{proof}[Proof of \ref{Kaledin}]
We start by observing that the top ghost component $w_n$ of $W_{n+1,p}(R;M)$ modulo transfer descends along the restriction map
\[
\xymatrix@C=60pt{
W_{n+1,p}(R;M)\ar[r]^-{w_n}\ar@{->>}[d]_R&(\cyctens{R}{M}{n})^{C_{p^n}}\ar@{->>}[d]
\\
W_{n,p}(R;M)\ar@{-->}[r]_-{w_n}&Q_{p^n}(R;M)\rlap{\ .}
}
\]
This is because the summand of $w_n(a_0,\dots, a_n)$ which depends on $a_n$ is $\tr^{C_{p^n}}_e(a_n)$ and it therefore vanishes in $Q_{p^n}(R;M)$. We claim that this map further descends along the map $W_{n,p}(R;M)\to W_{n,p}(R;M/p)$ induced by the modulo $p$ reduction $M\to M/p$. We start by computing that for every $0\leq i\leq n-1$, the $i$-th summand of $w_n(a_0+px_0,\dots, a_n+px_n)$ is
\begin{align*}
&\tr^{C_{p^n}}_{C_{p^{n-i}}}(a_i+px_i)^{\otimes p^{n-i}}=\tr^{C_{p^n}}_{C_{p^{n-i}}}\sum_{f\colon p^{n-i}\to 2}s_{f(1)}\otimes \dots \otimes s_{f(p^{n-i})}
  \\&=\tr^{C_{p^n}}_{C_{p^{n-i}}}({a_i}^{\otimes p^{n-i}}+(px_i)^{\otimes p^{n-i}}+\sum_{l=0}^{p^{n-i-1}}\sum_{\{\substack{f\colon p^{n-i}\to 2\\\mathrm{stab}(f)=C_{p^{l}}}\}/C_{p^{n-i}}}\tr^{C_{p^{n-i}}}_{C_{p^l}}(s_{f(1)}\otimes \dots \otimes s_{f(p^{n-i})}))
,
\end{align*}
where $s_1=a_i$ and $s_2=px_i$, and the sum is the decomposed according to the orbits of the $C_{p^{n-i}}$-action on the set of maps $p^{n-i}\to 2$ by precomposition. Every non-constant $f\colon p^{n-i}\to 2$ with stabiliser $C_{p^l}$ needs to have value $2$ on at least $p^l$ points. Thus each tensor product above is divisible by $p^l$, and 
\begin{align*}
&\tr^{C_{p^n}}_{C_{p^{n-i}}}(a_i+px_i)^{\otimes p^{n-i}}=\tr^{C_{p^n}}_{C_{p^{n-i}}}(a_i^{\otimes p^{n-i}})+\sum_{l=0}^{p^{n-i}}\tr^{C_{p^{n}}}_{C_{p^l}}p^ly_l
,
\end{align*}
for some $C_{p^l}$-invariant $y_l\in (\cyctens{R}{M}{n})^{C_{p^l}}$. But then $p^ly_l=\tr_{e}^{C_{p^l}}y_l$, and the right hand side is congruent to $\tr^{C_{p^n}}_{C_{p^{n-i}}}(a_i^{\otimes p^{n-i}})$ modulo the image of $\tr^{C_{p^n}}_{e}$. Note that this is not quite enough to conclude that the map $w_n$ factors over $W_{n,p}(R;M/p)$ as one needs to handle the kernel of $W_{n,p}(R;M) \to W_{n,p}(R;M/p)$. 
Let $M'$ denote the sub-$R$-bimodule of $M \times M$ of those pairs $(x,y) \in M$ with $x-y \in pM$. Then we get a reflexive coequaliser
\[
\xymatrix{M'\ar@<.5ex>[r]^a \ar@<-.5ex>[r]_b & M \ar[l] \ar@{->>}[r] &  M/p, }
\]
where the section is the diagonal and $a$ and $b$ are the projections. Consider the diagram
\[\xymatrix{ \prod_{i=0}^{n-1}\tens{R}{M'}{i} \ar@<.5ex>[r]^{a_\ast} \ar@<-.5ex>[r]_{b_\ast} \ar@{->>}[d]^{\pi} & \prod_{i=0}^{n-1}\tens{R}{M}{i}  \ar@{->>}[d]^{\pi}  &   \\ W_{n,p}(R;M')  \ar@<.5ex>[r]^{a_\ast} \ar@<-.5ex>[r]_{b_\ast} & W_{n,p}(R;M) \ar[r]^{w_n} & Q_{p^n}(R;M). }\]
From the above congruences we know that $w_n \pi a_\ast=w_n \pi b_\ast$ and hence, since the left hand diagram commutes, we get that $w_n a_\ast\pi = w_n  b_\ast \pi$. But $\pi$ is surjective and therefore $w_n a_\ast=w_n b_\ast$. Since 
\[
\xymatrix{W_{n,p}(R;M')\ar@<.5ex>[r]^{a_\ast} \ar@<-.5ex>[r]_{b_\ast} & W_{n,p}(R;M) \ar@{->>}[r] &  W_{n,p}(R;M/p) }
\]
is a coequaliser, this gives a well-defined additive natural transformation
\[
w_n\colon W_{n,p}(R;M/p)\longrightarrow Q_{p^{n}}(R;M)
\]
for every $R$-bimodule $M$, factoring $w_n\colon W_{n,p}(R;M)\longrightarrow Q_{p^{n}}(R;M)$. It is lax symmetric monoidal because $w_n$ is. It is surjective because the fixed-points $(\cyctens{R}{M}{n})^{C_{p^n}}$ are generated additively by elements of the form 
\[\tr^{C_{p^n}}_{C_{p^{n-i}}}a_i^{\otimes p^{n-i}}=w_n(0,\dots,a_i,\dots,0)\]
for $0\leq i\leq n$, where $a_i\in \cyctens{R}{M}{i}$ (this can be verified first for free bimodules, and then by resolving $(R;M)$ by a free bimodule).

Let us now suppose that $R$ is commutative with perfect $R/p$, and that $M=R(X):=\oplus_XR$ is a free $R$-module. Under the isomorphism $\tens{R}{R(X)}{n}\cong R(X^{\times p^n})$, we write the transfer as the map
\[
\tr^{C_{p^n}}_e\colon R(X^{\times p^n})\longrightarrow R(X^{\times p^n})^{C_{p^n}}
\]
which sends a basis element $(x_1,\dots ,x_{p^n})$ to $\sum_{\sigma\in C_{p^n}}\sigma(x_1,\dots ,x_{p^n})$.
For $n=1$ the invariants $R(X^{\times p})^{C_{p}}$ decompose as
\[
R(X^{\times p})^{C_{p}}\cong R(X^{\times p}/{C_{p}})\cong R(X)\oplus R((X^{\times p}\backslash \Delta)/C_p)
\]
where the first summand is generated by the diagonal elements $(x,\dots,x)$, and the second summand by the transferred elements. Thus the second summand is quotiented off in $Q_{p}(R;R(X))$, and the first summand is hit by the multiplication by $p$ map. This provides the identification
\[
w_1\colon W_{1,p}(R;R(X)/p)=R/p(X) \stackrel{\cong}{\longrightarrow} Q_p(R;R(X))
\]
that sends $rx$ to $r^p(x,\dots,x)$ for every $r\in R$ and $x\in X$. 

Now suppose inductively that $w_{n}$ is an isomorphism and that $p^{n}\colon R[p^{n+1}]\to R[p]$ is surjective, and let us show that $w_{n+1}$ is also an isomorphism.
The abelian group of invariants $R(X^{\times p^n})^{C_{p^n}}$ decomposes as
\[
R(X^{\times p^n})^{C_{p^n}}\cong R(X^{\times p^n}/C_{p^n})\cong R(X)\oplus R((X^{\times p}\backslash X)/C_p)\oplus 
\dots\oplus R((X^{\times p^n}\backslash X^{\times p^{n-1}})/C_{p^n})
\]
where $X^{\times p^{i-1}}\subset X^{p^{i}}$ via the $p^{i-1}$-power of the diagonal map $\Delta\colon X\to X^{\times p}$, and the isomorphism sends a basis element $(x_1,\dots,x_{p^{i}})$ in $X^{\times p^{i}}\backslash X^{\times p^{i-1}}$ to
\[
\tr^{C_{p^{n}}}_{C_{p^{n-i}}}(x_1,\dots,x_{p^{i}},x_1,\dots,x_{p^{i}},\dots,x_1,\dots,x_{p^{i}}).
\]
The transfer map hits the $X^{\times p^{i}}\backslash X^{\times p^{i-1}}$ summand with the multiplication by $p^{n-i}$ map, thus inducing an isomorphism
\[
Q_{p^n}(R;R(X))\cong R/p^n(X)\oplus R/p^{n-1}((X^{\times p}\backslash X)/C_p)\oplus 
 \dots\oplus R/p((X^{\times p^{n-1}}\backslash X^{\times p^{n-2}})/C_{p^{n-1}}).
\]
We can therefore define a map $R\colon Q_{p^{n+1}}(R;R(X))\to Q_{p^n}(R;R(X))$ which under this decomposition collapses the last summand, and which on the summand $X^{\times p^{i}}\backslash X^{\times p^{i-1}}$ is the sum of the modulo $p^{n-i}$ reductions $R/p^{n+1-i}\to R/p^{n-i}$. 
We claim that there is a short exact sequence
\[
0\to Q_{p^n}(R;\tens{R}{R(X)}{})_{C_p}\longrightarrow Q_{p^{n+1}}(R;R(X))\stackrel{R^n}{\longrightarrow}Q_{p}(R;R(X))\to 0
\]
where $\tens{R}{R(X)}{}\cong R(X^{\times p})$ is also a free $R$-module, and the $C_p$-action on $Q_{p^n}(R;\tens{R}{R(X)}{})$ is induced by the Weyl action on $(\tens{R}{(\tens{R}{R(X)}{})}{n})^{C_{p^n}}\cong (\tens{R}{R(X)}{n+1})^{C_{p^n}}$. Indeed under the decomposition above the kernel $K_n$ of $R^n$ is
\begin{align*}
K_n&\cong R/p^{n}(X)\oplus R/p^{n}((X^{\times p}\backslash X)/C_p)\oplus  R/p^{n-1}((X^{\times p^2}\backslash X^{\times p})/C_{p^2})\oplus
 \dots\oplus R/p((X^{\times p^{n}}\backslash X^{\times p^{n-1}})/C_{p^{n}})
 \\&\cong R/p^{n}(X^{\times p}/C_p)\oplus  R/p^{n-1}((X^{\times p^2}\backslash X^{\times p})/C_{p^2})\oplus
 \dots\oplus R/p((X^{\times p^{n}}\backslash X^{\times p^{n-1}})/C_{p^{n}})
\\&\cong (R/p^{n}(X^{\times p})\oplus  R/p^{n-1}((X^{\times p^2}\backslash X^{\times p})/C_{p})\oplus
 \dots\oplus R/p((X^{\times p^{n}}\backslash X^{\times p^{n-1}})/C_{p^{n-1}}))_{C_p}
 \\&\cong Q_{p^n}(R;R(X^{\times p}))_{C_p},
\end{align*}
where the first isomorphism is due to the fact that since $p^{n}\colon R[p^{n+1}]\to R[p]$ is surjective, the kernel of the $p$-reduction map $R/p^{n+1}\to R/p$ is $R/p^n$. The second isomorphism collects the first two summands. The third isomorphism commutes the quotient $S(Y)_{C_p}\cong S(Y/C_p)$ for the free $S$-module on a $C_{p}$-set $Y$. The last one is again by the decomposition of $Q_{p^n}$ above using that $R(X^{\times p})$ is free.
Thus the bottom row of the commutative diagram
\[
\xymatrix{
0\ar[r]&W_{n,p}(R;R(X^{\times p})/p)_{C_p}\ar[r]^-V\ar[d]^{w_{n}}_{\cong}&W_{n+1,p}(R;R(X)/p)\ar[r]^{R^n}\ar[d]^{w_{n+1}}&W_{1,p}(R;R(X)/p)\ar[d]^{(-)^{p^{n}}}\ar[r]&0
\\
0\ar[r]&Q_{p^n}(R;R(X^{\times p}))_{C_p}\ar[r]^-V&Q_{p^{n+1}}(R;R(X))\ar[r]^{R^n}&Q_{p}(R;R(X))\ar[r]&0
}
\]
is exact, where the bottom map $V$ is induced by $\tr_{C_{p^n}}^{C_{p^{n+1}}}\colon R(X^{\times p^{n+1}})^{C_{p^n}}_{C_p}\to  R(X^{\times p^{n+1}})^{C_{p^{n+1}}}$.
The top row is exact by Proposition \ref{singleVinj} since $R$ is commutative.
Moreover $w_{n}$ is an isomorphism by the inductive assumption.  The right vertical map is the map $R/p(X)\to R/p(X)$ which sends $r\cdot x$ to  $r^{p^n}\cdot x$, which is an isomorphism since $R/p$ is assumed to be perfect.
 It follows that $w_{n+1}$ is an isomorphism.
\end{proof}

%

In  \cite[Cor. 2.5]{KaledinncW}  Kaledin shows that there is  a unique functor $\overline{W}_n$ from $\F_p$-vector spaces to abelian groups such that $\overline{W}_n(A/p)=Q_{p^n}(\Z;A)$ for every free abelian group $A$. Thus Theorem \ref{Kaledin} immediately gives the following.
\begin{cor} There is a natural isomorphism of abelian groups $\overline{W}_n\cong W_{n,p}(\F_p;-)$.\qed
\end{cor}
This construction is lifted in \cite[Prop 2.3]{KaledinWpoly}  to a functor $\overline{W}_n$ from $k$-modules to $W_{n,p}(k)$-modules for every perfect field $k$ of characteristic $p$ (and in fact further to a category of Mackey functors). It is determined by a similar formula 
\[\overline{W}_n(E/p)=Q_{p^n}(W_{m,p}(k);E)\]
for every free $W_{m,p}(k)$-module $E$ and every $m\geq n$.

\begin{cor} Let $k$ be a perfect field of characteristic $p$, $m\geq n\geq 1$ integers and $V$ a $k$-vector space. There is a natural isomorphism of $W_{n,p}(k)$-modules
\[W_{n,p}(k;V)\cong \overline{W}_n(V)\] with the polynomial Witt vectors $\overline{W}_n(V)$ of \cite{KaledinWpoly} and \cite{KaledinncW}.
\end{cor}

\begin{proof} By Kaledin's characterisation of the functor $\overline{W}_n$ it is sufficient to show that $W_{n,p}(k;E/p)\cong Q_{p^n}(W_{m,p}(k);E)$ for every free $W_{m,p}(k)$-module $E$. Since $k$ is perfect of characteristic $p$ and $m\geq n$, the commutative ring $W_{m,p}(k)$ satisfies the conditions of Theorem \ref{Kaledin} and
\[
W_{n,p}(k;E/p)\cong W_{n,p}(W_{m,p}(k)/p;E/p)\stackrel{w_n}{\longrightarrow} Q_{p^n}(W_{m,p}(k);E)
\]
is an isomorphism. Since $w_n$ is symmetric monoidal this is an isomorphism of  $W_{n,p}(k;k)=W_{n,p}(k)$-modules. 
\end{proof}
 
 \begin{rem}
 The first Corollary can be deduced by the second one as follows. Let $A$ be a free abelian group. Since $A/p\cong (A\otimes\Z^{\wedge}_p)/p$, the corollary for the perfect field $k=\F_p$ gives an isomorphism
 \[
 W_{n,p}(\F_p;A/p)\cong  W_{n,p}(\F_p;(A\otimes\Z^{\wedge}_p)/p)\cong Q_{p^n}(\Z^{\wedge}_p;A\otimes\Z^{\wedge}_p)\cong Q_{p^n}(\Z;A)\otimes\Z^{\wedge}_p\cong Q_{p^n}(\Z;A)
 \]
 where the last isomorphism holds because $p^n=\tr^{C_{p^n}}_e(1)$ acts as zero on $Q_{p^n}(\Z;\Z)$.
 \end{rem}

\section{Witt vectors with coefficients in homotopy theory}\label{sec:TR}

The topological restriction homology (TR) of a ring spectrum $R$ with coefficients in an $R$-bimodule $M$ was introduced by Lindenstrauss and McCarthy in \cite{LMcC}, as a version with coefficients of the cyclic bar construction. It is defined for every integer $n\geq 0$ as the fixed-points of a genuine $C_n$-spectrum $M^{\varowedge_Rn}$, a homotopical analogue of the algebraic cyclic tensor powers appearing earlier in the present paper, constructed as the geometric realisation of a simplicial object with $k$-simplicies
\[
(M^{\varowedge_Rn})_k:=(M\wedge R^{\wedge k})^{\wedge n},
\]
and with a simplicial structure analogous to that of the $n$-fold subdivision of the cyclic bar construction of $R$. The underlying spectrum is in fact equivalent to $\THH(R;M^{\wedge_Rn})$. In order to derive this construction appropriately and obtain a genuine equivariant spectrum the authors employed B\"okstedt's model for the smash product, and in turn defined $\TR_{\langle n\rangle}$ as the fixed-points
\[
\TR_{\langle n\rangle}(R;M)=(M^{\varowedge_Rn})^{C_n}.
\]
The foundations of this theory have been reworked in \cite{polygonic} by McCandless and the second and third authors, and we now review the bases of this construction.

\begin{defn}\cite{polygonic} 
A \emph{polygonic spectrum} $X$ consists of a spectrum $X_d$ with an action of the cyclic group $C_d$ for every integer $d \geq 1$, together with $C_d$-equivariant Frobenius maps
\[
\phi_{p,d}\colon X_d\longrightarrow (X_{dp})^{tC_p}
\]
for every prime $p$ and $d \geq 1$, where $(-)^{tC_p}$ denotes the Tate construction of the $C_p$-action.

Given a truncation set $T$ (a subset of $\mathbb{N}_{>0}$ such that if $xy \in T$, then $x \in T$ and $y \in T$), a $T$-\emph{typical polygonic spectrum} is a polygonic spectrum with $X_d=0$, whenever $d \notin T$. We will mainly focus on the case where $T=\langle n\rangle$ is the truncation set of divisors of an integer $n\geq 1$, especially when $n$ is a power of a prime $p$ .
\end{defn}

In \cite{polygonic} the authors introduce a stable and presentable  $\infty$-category $\PgcSp$ of polygonic spectra, and the full $\infty$-subcategory  $\PgcSp_{T}$ of $T$-typical polygonic spectra. Given an inclusion of truncation sets  $T' \subset T$, the corresponding inclusion $\PgcSp_{T'} \hookrightarrow \PgcSp_T$ has a left adjoint $\Res^{T}_{T'} \colon \PgcSp_T \to  \PgcSp_{T'}$ which is a localisation, by  \cite[Section 2 and Construction 2.11]{polygonic}. On objects $\Res^{T}_{T'}$ sends a $T$-typical polygonic spectrum $X=\{X_d\}_{d \geq 1}$ to the $T'$-typical polygonic spectrum $\Res^{T}_{T'}(X)=\{Y_d\}_{d \geq 1}$ with $Y_d=X_d$ whenever $d \in T'$ and $Y_d=0$ otherwise.
 In particular, we have the localisation
\[\xymatrix{\PgcSp_{\langle m\rangle} \ar@<0.5ex>[r]^{\Res^{\langle m\rangle}_{\langle n\rangle}} & \PgcSp_{\langle n\rangle}, \ar@<0.5ex>[l] }\]
whenever $n$ divides $m$. 

The $\infty$-category $\PgcSp$ is symmetric monoidal by \cite[Construction 2.14]{polygonic}. The symmetric monoidal structure comes from levelwise tensoring the entries and using the lax symmetric monoidality of the Tate construction. 

\begin{example} \label{example: polygonic}
\begin{enumerate}
\item Any cyclotomic spectrum $X$ as in \cite{NikSchol} defines a polygonic spectrum, with $X_{d}=X$ for all $d \geq 1$, and $\phi_{p,d}\colon X\to X^{tC_p}$ the cyclotomic Frobenius of $X$ for every $d$ and any prime $p$. In particular any spectrum $X$ with trivial action defines a polygonic spectrum $X^{\mathrm{triv}}$, for example the sphere spectrum $\mathbb{S}^{\mathrm{triv}}$. By applying the truncation we obtain a $T$-typical polygonic spectrum $X^{\mathrm{triv}}_T=\Res_{T} X^{\mathrm{triv}}$, for any truncation set $T$.

Further, any $p$-typical cyclotomic spectrum $X$ as in \cite{NikSchol} defines a $\langle p^{\infty}\rangle$-typical polygonic spectrum, where $\langle p^{\infty}\rangle$ is the truncation set consisting of all the powers of $p$. It consists of the spectra $X_{p^k}=X$ for all $k \geq 0$ with maps $\phi_{p,p^k}\colon X\to X^{tC_p}$ the cyclotomic Frobenius of $X$ for every $k \geq 0$, and $X_d=0$ if $d$ is not a power of $p$. After applying $\Res^{\langle p^{\infty}\rangle}_{\langle p^n\rangle}$, we get a $\langle p^n\rangle$-typical polygonic spectrum with $X_{p^k}=X$ for all $0 \leq k \leq n$.

\item Every space $M$ defines a polygonic spectrum $\Sigma^{\infty}_+M^{\times}$, with $(\Sigma^{\infty}_+ M^{\times })_{n}=\Sigma^{\infty}_+ M^{\times n}$ and with the Frobenius maps defined by the composite of the canonical maps
\[
\Sigma^{\infty}_+M^{\times n} \to \Sigma^{\infty}_+(M^{\times pn})^{hC_p}\to (\Sigma^{\infty}_+ M^{\times pn})^{hC_p}\to (\Sigma^{\infty}_+ M^{\times pn})^{tC_p},
\]
where the first map is the diagonal, the second is the unique map into the limit and the last map is the canonical map from the homotopy fixed points to the Tate construction. By applying the appropriate truncation functors we also obtain its $T$-typical versions. In particular, we get a $\langle p^n\rangle$-typical polygonic spectrum $\Res_{\langle p^n\rangle}\Sigma^{\infty}_+M^{\times}$ with $(\Res_{\langle p^n\rangle}\Sigma^{\infty}_+M^{\times})_{p^k}=\Sigma^{\infty}_+ M^{\times p^k}$ for $0 \leq k \leq n$.

\item For every ring spectrum $R$ and $R$-bimodule $M$, there is a polygonic spectrum $\underline{\THH}(R;M)$ with $\underline{\THH}(R;M)_{n}=\THH(R;M^{\wedge_R n})$ and Frobenius maps
\[
\phi_{p,n}\THH(R;M^{\wedge_R n})\longrightarrow \THH(R;M^{\wedge_R pn})^{tC_p}
\]
defined on the cyclic bar construction from the Tate diagonals, see \cite[Construction 6.31]{polygonic}.
\end{enumerate}
\end{example}

\begin{defn}\cite{polygonic} 
For any $T$-typical polygonic spectrum $X$, one defines $\TR_T(X)$ to be the mapping spectrum out of the sphere spectrum 
\[\TR_T(X)=\map_{\PgcSp_T}(\mathbb{S}^{\mathrm{triv}}_T, X).\]
For a ring spectrum $R$ and $R$-bimodule $M$, we let $\TR_T(R;M):=\TR_T(\Res_{T}\underline{\THH}(R;M))$.
\end{defn}

Since $\mathbb{S}^{\mathrm{triv}}_T$ is the unit of the monoidal structure on $T$-typical polygonic spectra, we have that $\TR_T$ is a lax monoidal functor.
Let $\bimod_{\mathbb{S}}$ denote the $\infty$-category of spectral bimodules. Then the functor 
\[\underline{\THH} \colon \bimod_{\mathbb{S}} \to \PgcSp \]
sending $(R;M)$ to $\underline{\THH}(R;M)$ is a lax symmetric monoidal functor, by \cite[Section 6]{polygonic}. Hence all in all we conclude that $\TR_T(R;M)$ is a lax symmetric monoidal functor from $\bimod_{\mathbb{S}}$ to spectra. 

In \cite[Proposition 2.10]{polygonic}, the authors provide an equaliser formula for $\TR_T(X)$ analogous to the description of topological cyclic homology of \cite{NikSchol} and \cite{BluMan}. In the case of $T={\langle p^n\rangle}$, this is the equaliser
\[
\TR_{\langle p^n\rangle}(R;M)=eq\left(\xymatrix{
\prod_{i=0}^n\THH(R;M^{\wedge_R p^i})^{hC_{p^i}}\ar@<.5ex>[r]\ar@<-.5ex>[r]
&
\prod_{i=0}^{n-1}(\THH(R;M^{\wedge_R p^{i+1}})^{tC_p})^{hC_{p^i}}
}
\right).
\]
One of the maps of the equaliser is the composite
\[
\prod_{i=0}^n\THH(R;M^{\wedge_R p^i})^{hC_{p^i}}\to \prod_{i=0}^{n-1}\THH(R;M^{\wedge_R p^i})^{hC_{p^i}}\longrightarrow \prod_{i=0}^{n-1}(\THH(R;M^{\wedge_R p^{i+1}})^{tC_p})^{hC_{p^i}}
\]
of the projection and the product of the homotopy fixed-points of the Frobenius maps. The other map in the equaliser is the composite
\[
\prod_{i=0}^n\THH(R;M^{\wedge_R p^i})^{hC_{p^i}}\to \prod_{i=1}^{n}\THH(R;M^{\wedge_R p^i})^{hC_{p^i}}\longrightarrow \prod_{i=1}^{n}(\THH(R;M^{\wedge_R p^{i}})^{tC_p})^{hC_{p^{i-1}}}
\]
of the other projection followed by the product of the homotopy $C_{p^{i-1}}$-fixed points of canonical maps $\THH(R;M^{\wedge_R p^i})^{hC_{p}}\to \THH(R;M^{\wedge_R p^i})^{tC_{p}}$.

\begin{example}\label{knowncalculations}
Let $(R;M)$ be a bimodule spectrum, where $R$ and $M$ are connective. 
\begin{enumerate}
\item By definition $\TR_{\langle 1\rangle}(R;M)=\THH(R;M)$ is the topological Hochschild homology of $R$ with coefficients in $M$. Thus $\pi_0\TR_{\langle 1\rangle}(R;M)\cong \pi_0M/[\pi_0R,\pi_0M]\cong W_{\langle 1\rangle}(\pi_0R;\pi_0M)$.
\item When $R=M$, we have that $\TR_{\langle p^n\rangle}(R;R)=\TR^{n+1}(R)$ is the classical $p$-typical TR of \cite{BHM}. Thus by the calculations of \cite{Wittvect} and \cite{HesselholtncW} there is a natural isomorphism
\[
W_{\langle p^n\rangle}(\pi_0R;\pi_0R)\cong \pi_0\TR_{\langle p^n\rangle}(R;R)
\]
which is multiplicative when $R$ is commutative.
\item When $R=\mathbb{S}$ is the sphere spectrum and $M=A$ is a connective spectrum, there is a natural equivalence
\[
\TR_{\langle p^n\rangle}(\mathbb{S};A)=(N_{e}^{C_{p^n}}A)^{C_{p^n}},
\]
where the right-hand side is the genuine fixed-points of the Hill-Hopkins-Ravenel norm of $A$ of \cite{HHR}. Indeed $\THH(\mathbb{S};A)=A$, and the Frobenius maps above
\[
\phi\colon A^{\wedge p^i}\to (A^{\wedge p^{i+1}})^{tC_{p}}
\]
are by construction the Tate diagonals of the spectra $A^{\wedge p^i}$. By identifying the geometric fixed-points spectrum $(N_{e}^{C_{p^n}}A)^{\Phi C_{p^i}}$ with $A^{\wedge p^{n-i}}$ we find that the equaliser formula for $\TR_{\langle p^n\rangle}(\mathbb{S};A)$ above is equivalent to the iterated pullback of \cite[Corollary II.4.7]{NikSchol} which describes the genuine fixed points spectrum $(N_{e}^{C_{p^n}}A)^{C_{p^n}}$. When $A=H\F_p$, the components of this norm is the ring of $(n+1)$-truncated $p$-typical Witt vectors
\[\pi_0(N_{e}^{C_{p^n}}H\F_p)^{C_{p^n}}\cong W_{\langle p^n\rangle}(\F_p),\]
by work of Mazur (see \cite[Proposition 5.23]{BGHL}). For the prime $p=2$ on the other hand, there is a natural ring isomorphism
\[\pi_0(N_{e}^{C_{2}}HA)^{C_{2}}\cong W_{\langle 2\rangle}(\Z;A),\]
for every ring $A$ by \cite[Proposition 5.5]{THRmodels} (compare with \ref{ex:solid}).
\end{enumerate}
\end{example}

\subsection{Witt vectors with coefficients and TR with coefficients}\label{secnorm}

We now state the main result of this section, which extends the calculations of Example \ref{knowncalculations} to all bimodules. 

\begin{theorem}\label{pi0norm}
Let $R$ be a connective ring spectrum and $M$ a connective $R$-bimodule. There is an isomorphism
\[
W_{\langle p^n\rangle}(\pi_0R;\pi_0M)\cong\pi_0\TR_{\langle p^n\rangle}(R;M),
\]
which is moreover natural in $(R;M)$ and monoidal.
In particular for every connective spectrum $A$, this gives an isomorphism $W_{\langle p^n\rangle}(\Z;\pi_0 A)\cong \pi_0(N_{e}^{C_{p^n}}A)^{C_{p^n}}$ with the Hill-Hopkins-Ravenel norm construction, which is a ring isomorphism when $A$ is a ring spectrum.
\end{theorem}

\begin{rem}
When $A=\Sigma_{+}^{\infty}X$ is the suspension spectrum of a space $X$, the norm $N^{C_{p^n}}_e(\Sigma^{\infty}_+X)\cong \Sigma^{\infty}_+X^{\times p^n}$ is again a suspension spectrum, and the tom Dieck splitting provides a canonical isomorphism of abelian groups
\[
\pi^{C_{p^n}}_0N^{C_{p^n}}_e(\Sigma^{\infty}_+X)\cong \bigoplus_{i=0}^n\pi_0(\Sigma^{\infty}_+((X^{\times p^n})^{C_{p^{n-i}}}))_{hC_{p^i}} \cong \bigoplus_{i=0}^n(\mathbb{Z}(\pi_0X)^{\otimes p^i})_{C_{p^i}}
\]
where $\mathbb{Z}(\pi_0X)$ is the free abelian group generated by $\pi_0 X$.
 Thus by Theorem \ref{pi0norm} the group of Witt vectors $W_{\langle p^n\rangle}(\Z;\Z(\pi_0 X))$ is  isomorphic to the direct sum of the $(\mathbb{Z}(\pi_0X)^{\otimes p^i})_{C_{p^i}}\cong \Z((\pi_0X^{\times p^i})_{C_{p^i}})$, and it is in particular free abelian.
 This product decomposition matches with the algebraic one of Proposition \ref{basis}.
\end{rem}

The proof of Theorem \ref{pi0norm} will occupy the rest of the section. To set up our proof, and in particular to construct the map giving the isomorphism of the statement,  we need to discuss the spectrum level analogue of the Witt vectors operators of the previous section. For describing these operators we will use the stable $\infty$-category $\Sp^{\Z}_{\qfgen}$ of \emph{quasifinitely genuine $\Z$-spectra} which was constructed in \cite{polygonic} following ideas of Kaledin. We briefly recall the setup, and refer to \cite[Sections 4-6]{polygonic} for further details. 

The $\infty$-category $\Sp^{\Z}_{\qfgen}$ is defined as the  $\infty$-category of spectral Mackey-functors on spans of quasifinite $\Z$-sets. In particular, any quasifinitely genuine $\Z$-spectrum $Y$ has genuine and geometric fixed-points spectra $Y^{d\Z}$ and $Y^{\Phi d\Z}$ for any integer $d\geq 1$.
By \cite[Proposition 5.2, Theorem 5.4]{polygonic} there is an adjunction
\[\xymatrix{\Sp^{\Z}_{\qfgen}  \ar@<0.5ex>[r]^L &  \PgcSp \ar@<0.5ex>[l]^{\underline{TR}}}\]
which restricts to an equivalence of $\infty$-categories on uniformly bounded below objects. The left adjoint $L$ is defined by sending a quasifinitely genuine $\Z$-spectrum $Y$ to the polygonic spectrum $LY=\{X_d\}_{d \geq 1}$ consisting of the  the geometric fixed points spectra $X_d=Y^{\Phi d\Z}$ with their residual $C_d \cong \Z/d\Z$-action. The Frobenius maps come from the canonical map from the $C_p$-geometric fixed points to the $C_p$-Tate construction. 
The right adjoint $\underline{\TR}$ sends a polygonic spectrum $X$ to a quasifnitely genuine $\Z$-spectrum with genuine fixed points $\underline{\TR}(X)^{n\Z}=\TR(\sh_nX)$, where by definition $(\sh_nX)_d=X_{nd}$  for $n \geq 1$, with the obvious actions and Frobenius maps \cite[Example 2.28]{polygonic}. This equivalence is symmetric monoidal since the geometric fixed points are symmetric monoidal. In particular, for uniformly bounded below polygonic spectra $\underline{\TR}$ is symmetric monoidal. However, since genuine fixed points are only lax symmetric monoidal this shows that the functor $\TR(X)$ is only lax symmetric monoidal which agrees with our observation above. 

For a connective pair $(R;M)$, we have formulas for the genuine and geometric fixed-points of $\underline{\TR}(R;M)$, given respectively by
\[\underline{\TR}(R;M)^{n\Z}\simeq \TR(R;M^{\wedge_R n}) \ \ \ \ \  \mbox{and}\ \ \ \ \ \ \underline{\TR}(R;M)^{\Phi n\Z} \simeq \THH(R;M^{\wedge_R n}).\]
By construction, any 
quasifinitely genuine $\Z$-spectrum comes equipped with transfer and restriction maps, and Weyl actions. Thus $\underline{\TR}(R;M)$ admits a Weyl $C_n$-action, whose generator we denote by
\[ \sigma_{\langle n\rangle} \colon \TR(R;M^{\wedge_R n}) \to \TR(R;M^{\wedge_R n}),\]
a restriction map
\[F \colon \TR(R;M) \to \TR(R;M^{\wedge_R n})^{hC_n}\]
which we call Frobenius, and a transfer
\[V \colon \TR(R;M^{\wedge_R n})_{hC_n} \to \TR(R;M)\]
which we call Verschiebung, for every integer $n\geq 1$.
 Let us now produce truncated versions of these structure maps.

As mentioned above, the $\infty$-category of ${\langle p^n\rangle}$-typical polygonic spectra $\PgcSp_{\langle p^n\rangle}$ is a localisation of $\PgcSp$, where the localisation functor annihilates the values on integers which are not of the form $p^k$, for $0 \leq k \leq n$. Under the equivalence above for uniformly bounded below spectra, this corresponds to the full subcategory of $\Sp^{\Z}_{\qfgen}$ of those $\Z$-spectra $X$ with $X^{\Phi d\Z}=0$ unless $d=1,p,\dots p^n$. The geometric fixed points functor $(-)^{\Phi p^n\Z} \colon \Sp^{\Z}_{\qfgen} \to \Sp^{C_{p^n}}$ to genuine $C_{p^n}$-spectra is a localisation whose right adjoint inflation $\infl_{p^n\Z} \colon  \Sp^{C_{p^n}} \to \Sp^{\Z}_{\qfgen}$ \cite[Construction 4.20]{polygonic} is fully faithful with the essential image given by the latter full subcategory of $\Sp^{\Z}_{\qfgen}$. This shows (see also \cite[Example 2.9]{polygonic}) that the functor
\[
\underline{\TR}(-)^{\Phi p^n \Z}\colon \PgcSp_{\langle p^n\rangle}\longrightarrow \Sp^{C_{p^n}}
\]
is an equivalence on the full subcategories of bounded below objects.

\begin{prop} \label{prop: truncated TR} For every $n\geq 0$, there is a symmetric monoidal functor $\underline{\TR}_{\langle p^n\rangle} \colon \bimod_{\mathbb{S}}^{\geq 0} \to \Sp^{C_{p^n}}$ from connective spectral bimodules to genuine $C_{p^n}$-spectra, and a natural equivalence
\[
\underline{\TR}_{\langle p^n\rangle}(R;M)^{C_{p^k}}\simeq \TR_{\langle p^{k}\rangle}(R;M^{\wedge_Rp^{n-k}})
\]
 for every $0\leq k\leq n$. \end{prop}

\begin{proof} Consider the ${\langle p^n\rangle}$-typical polygonic spectrum $\Res_{\langle p^n\rangle} (\underline{\THH}(R;M))$ with corresponding quasifinitely genuine $\Z$-spectrum  $Y=\underline{\TR}(\Res_{\langle p^n\rangle} (\underline{\THH}(R,M)))$ (with vanishing $d\Z$-genuine and geometric fixed points unless $d=1,p,\dots p^n$). By the equaliser formula of \cite[Proposition 2.10]{polygonic}, the genuine fixed points spectrum $Y^{p^{n-k}\Z}$ is equivalent to $\TR_{\langle p^{k}\rangle}(R;M^{\wedge_Rp^{n-k}})$ for all $0\leq k \leq n$. Define
\[\underline{\TR}_{\langle p^n\rangle}(R;M)= Y^{\Phi p^n\Z} \in \Sp^{C_{p^n}}.\]
Since $Y$ is in the essential image of the inflation $\infl_{p^n\Z}$, we have a good control on its genuine fixed-points. Indeed, by definition of the inflation
\begin{align*}(\underline{\TR}_{\langle p^n\rangle}(R;M))^{C_{p^k}}&=(Y^{\Phi p^n \Z})^{C_{p^k}} ={(Y^{\Phi p^n \Z})}^{p^{n-k}\Z/p^n\Z}\\&=Y^{p^{n-k}\Z}\simeq \TR_{\langle p^{k}\rangle}(R;M^{\wedge_Rp^{n-k}}). \end{align*}
The symmetric monoidality follows from the definition since the functors $(-)^{\Phi p^n\Z}$ and $\underline{\TR}$ and $\underline{\THH}$ are all symmetric monoidal. 

\end{proof}

As a consequence of the latter theorem, we obtain a $C_{p^k}$-action on $\TR_{\langle p^{n-k}\rangle}(R;M^{\wedge_R p^k})$, with generator
\[\sigma_{k} \colon \TR_{\langle p^{n-k}\rangle}(R;M^{\wedge_R p^k}) \to\TR_{\langle p^{n-k}\rangle}(R;M^{\wedge_R p^k}),\] 
given by the action of the Weyl group $C_{p^n}/C_{p^{n-k}}\cong C_{p^k}$, and maps
\[
\begin{array}{ll}
F\colon \TR_{\langle p^n\rangle}(R;M)\longrightarrow \TR_{\langle p^{n-1}\rangle}(R;M^{\wedge_Rp})^{hC_p}&\hspace{.5cm} V\colon \TR_{\langle p^{n-1}\rangle}(R;M^{\wedge_Rp})_{hC_p}\longrightarrow \TR_{\langle p^n\rangle}(R;M)
\end{array}
\]
corresponding respectively to the restriction and transfer of a spectral Mackey functor, thus satisfying the homotopy coherent analogue of the double-coset formula of usual Mackey functors. Analogous maps were also defined in \cite[Lemma 4.10 and Corollary 5.7]{LMcC}. We also observe that $\TR_{\langle p^n\rangle}(R;M)$ is naturally lax symmetric monoidal since $\underline{\TR}$ is symmetric monoidal and the genuine fixed points functor is lax symmetric monoidal.

There are also maps $R\colon \TR_{\langle p^n\rangle}(R;M)\to \TR_{\langle p^{n-1}\rangle}(R;M)$, which under the equaliser formula for $\TR_{\langle p^n\rangle}$ above correspond to projections of product factors, and therefore fit into fibre sequences
\[
\THH(R;M^{\wedge_Rp^{n}})_{hC_{p^{n}}}\stackrel{V}{\longrightarrow} \TR_{\langle p^n\rangle}(R;M)\stackrel{R}{\longrightarrow} \TR_{\langle p^{n-1}\rangle}(R;M)
\]
for every $n\geq 1$ (see also \cite[Corollary 5.7]{LMcC}).

The final structure that we need on $\TR_{\langle p^n\rangle}$ is a topological analogue of the Teichm\"uller character map. We recall its construction, from \cite[Construction 6.33]{polygonic}. Given a connective bimodule $(R,M)$, there is a $C_n$-equivariant map
\[M^{\wedge n} \to \THH(R;M^{\wedge_R n}) \]
for any $n \geq 1$. If we identify $M^{\wedge n}$ with $\THH(\mathbb{S};M^{\wedge_\mathbb{S} n})$, this is the map induced by the map of bimodules $(\mathbb{S},M) \to (R,M)$. In particular, these maps for all $n \geq 1$ assemble into a natural morphism of polygonic spectra. We then consider the composite
\[\Sigma^{\infty}_+ (\Omega^\infty M)^{\times n} \to \Sigma^{\infty}_+ (\Omega^\infty (M^{\wedge n})) \to M^{\wedge n} \to \THH(R;M^{\wedge_R n})\]
where the first map is induced by the canonical map from the product to the smash product of spaces and the  lax monoidal structure of $\Omega^{\infty}$, and the second map is the counit of the $(\Sigma^{\infty}_+,\Omega^\infty)$-adjunction.
The source of this map admits a polygonic structure by Example \ref{example: polygonic}, where the Frobenius maps are induced by the diagonal. Similarly, the second term in this sequence can also be assembled into a polygonic spectrum via the polygonic structure on $M^{\wedge (-)}$ and the diagonals, and these maps form a sequence of morphisms of polygonic spectra when $n$ runs through natural numbers. 
The spectrum $\Sigma^{\infty}_+ (\Omega^\infty M)^{\times n}$ is the geometric fixed points ${(\Sigma^{\infty}_+\underline{M})}^{\Phi n\Z}$ of the suspension spectrum $\Sigma^{\infty}_+\underline{M}$, where $\underline{M}$ is the quasifinitely genuine $\Z$-space assigning to any finite orbit $S$ the space $(\Omega^\infty M)^{\times S}$, in particular $\underline{M}^{n\Z}={(\Omega^\infty M)}^{\times n}$.  
Hence we can interpret the latter composite as a morphism of polygonic spectra
\[L(\Sigma^{\infty}_+\underline{M}) \to \underline{\THH}(R;M) \]
which by adjunction gives the map of quasifinitely genuine $\Z$-spectra
\[\tau \colon \Sigma^{\infty}_+\underline{M} \to \underline{\TR}(R;M).\]
This is the topological analogue of the Teichm\"uller map. Let us now produce a truncated version of this map. By composing with the unit of the localisation we get a map of quasifinitely genuine $\Z$-spectra 
\[\Sigma^{\infty}_+\underline{M} \to \underline{\TR}(R;M) \to\infl_{p^n\Z} (\underline{\TR}_{\langle p^n\rangle}(R;M)),\]
and by adjoining a map of genuine $C_{p^n}$-spectra
\[\tau \colon \Sigma^{\infty}_+(\Omega^\infty M)^{\times p^n}\simeq (\Sigma^{\infty}_+\underline{M})^{\Phi p^n\Z}  \to \underline{\TR}_{\langle p^n\rangle}(R;M) \]
also denoted by $\tau$. By adjoining and taking $C_{p^k}$-fixed-points spaces this gives, using the equivalence of Proposition \ref{prop: truncated TR}, a map of spaces
\[\tau \colon (\Omega^{\infty}M)^{\times p^{n-k}} \to \Omega^{\infty} \TR_{\langle p^{k}\rangle}(R;M^{\wedge_Rp^{n-k}})\]
 for any $0 \leq k \leq n$.

Next, we give a conceptual description of the maps $R\colon \TR_{\langle p^n\rangle}(R;M)\to \TR_{\langle p^{n-1}\rangle}(R;M)$ which will explain why they are compatible with Frobenius, Verschiebung and Weyl actions. 

\begin{prop} \label{R maps} For any $n \geq 0$, there is a natural morphism of genuine $C_{p^{n+1}}$-spectra
\[R \colon \underline{\TR}_{\langle p^{n+1} \rangle}(R;M) \to \infl_{C_{p^n}}(\underline{\TR}_{\langle p^{n} \rangle}((R;M))),\]
where $\infl_{C_{p^n}} \colon \Sp^{C_{p^n}} \to \Sp^{C_{p^{n+1}}}$ is the right adjoint of the geometric fixed points $(-)^{\Phi C_p}$. For $0 \leq k \leq n+1$, the genuine $C_{p^k}$-fixed points of this map is equivalent to
\[R \colon \TR_{\langle p^{k}\rangle}(R;M^{\wedge_Rp^{n+1-k}}) \to \TR_{\langle p^{k-1}\rangle}(R;M^{\wedge_Rp^{n+1-k}}),\]
where the target is interpreted as zero when $k=0$. 
\end{prop}

\begin{proof}
By adjunction, we have the truncation map of polygonic spectra
\[R \colon \Res_{\langle p^{n+1} \rangle} (\underline{\THH}(R;M)) \to \Res_{\langle p^{n} \rangle} (\underline{\THH}(R;M)).\]
After passing to $\underline{\TR}$, we get the a map of quasifinitely genuine $\Z$-spectra
\[R \colon \underline{\TR}(\Res_{\langle p^{n+1} \rangle} (\underline{\THH}(R;M))) \to\underline{\TR}(\Res_{\langle p^{n} \rangle} (\underline{\THH}(R;M))).\]
We claim that by taking $p^{n+1}\Z$-geometric fixed points we obtain the map of the statement. By definition the $p^{n+1}\Z$-geometric fixed points of the source are  $\underline{\TR}_{\langle p^{n+1} \rangle}(R;M)$. For the target, we use that $\underline{\TR}$ is an equivalence of categories on bounded below objects, and that under this equivalence the localisation onto the subcategory of $\langle p^n\rangle$-typical polygonic spectra corresponds to that onto the genuine $C_{p^n}$-spectra. Thus
\begin{align*}
\Phi^{p^{n+1}\Z}\underline{\TR}(\Res_{\langle p^{n} \rangle} (\underline{\THH}(R;M)))&\simeq\underline{\TR}_{\langle p^{n+1}\rangle}(\Res_{\langle p^{n+1} \rangle}\Res_{\langle p^{n} \rangle} (\underline{\THH}(R;M)))
\\&\simeq \underline{\TR}_{\langle p^{n+1}\rangle}(\Res_{\langle p^{n} \rangle} (\underline{\THH}(R;M)))
\\&
\simeq  \infl_{C_{p^n}}(\underline{\TR}_{\langle p^{n} \rangle}((R;M))).
\end{align*}
\end{proof}

\begin{lemma} \label{one-connectivity} Let $f \colon (R;M) \to (S;N)$ be a morphism of connective spectral bimodules. Suppose that $f$ is $1$-connected, i.e., it induces isomorphisms on $\pi_0$ and surjections on $\pi_1$ between the underlying rings and the underlying bimodules. Then the induced map
\[\TR_{\langle p^n\rangle}(R;M) \longrightarrow \TR_{\langle p^n\rangle}(S;N)\]
is $1$-connected for any $n \geq 0$ and any prime $p$.
\end{lemma}

\begin{proof} For any $n \geq 0$, the induced map 
\[\THH(R; M^{\wedge_R p^{n}})  \longrightarrow \THH(S; N^{\wedge_S p^{n}}) \]
is $1$-connected. This follows by looking at the homotopy fibers and noticing that the property of being $1$-connected is preserved under smash products and geometric realisations. In particular the claim holds for $n=0$. The general case is proved inductively, by considering the commutative diagram 
\[\xymatrix{ 
 \THH(R; M^{\wedge_R p^{n}})_{h C_{p^{n}}}
 \ar[d]^{f_\ast}  \ar[r] & \TR_{\langle p^n\rangle}(R;M) \ar[d]^{f_\ast}   \ar[r]^R & \TR_{\langle p^{n-1}\rangle}(R;M) \ar[d]^{f_\ast}  
 \\
  \THH(S; N^{\wedge_R p^{n}})_{h C_{p^{n}}} \ar[r] & \TR_{\langle p^n\rangle}(S;N)  \ar[r]^R & \TR_{\langle p^{n-1}\rangle}(S;N),}\]
where the left hand map is $1$-connected by the previous paragraph, and the right hand map by the inductive assumption. Hence so is the middle map by the five lemma. 
\end{proof}

\begin{cor} \label{pi0dependence} Let $(R;M)$ be a connective bimodule spectrum. For any $n\geq 0$ and prime $p$, the canonical map induces a natural isomorphism
\[ \pi_0 \TR_{\langle p^n\rangle}(R;M) \cong \pi_0 \TR_{\langle p^n\rangle}(H\pi_0R; H\pi_0 M). \] \qed
\end{cor}

With this lemma at hand and by using Theorem \ref{existence}, in order to prove Theorem \ref{pi0norm} we can equivalently show that for every (discrete) bimodule $(R;M)$ there is an isomorphism
\[
W_{n+1,p}(R;M)\cong\pi_0\TR_{\langle p^n\rangle}(HR;HM),
\]
natural in $(R;M)$ and lax symmetric monoidal. In what follows unless necessary, we will often suppress $H$ and just write $(R;M)$ instead of $(HR;HM)$, keeping in mind that $(R;M)$ is discrete. 

We want to single out the formal properties  and structure of $\TR$ which makes it possible to construct the desired isomorphism. Let us define a functor ${\mathcal F}_{n+1} \colon \bimod \to \Ab$ for every $n\geq 0$, by 
\[
{\mathcal F}_{n+1}(R;M):=\pi_0\TR_{\langle p^n\rangle}(R;M),
\]
where we make the prime implicit to lighten up the notation. This functor inherits a lax symmetric monoidal structure from the one of $\TR_{\langle p^n\rangle}$, as well as operators
\[
\begin{array}{ll}
F \colon \FF_{n+1}(R;M)\longrightarrow \FF_{n}(R;\tens{R}{M}{})&\hspace{1cm} V \colon \FF_{n}(R;\tens{R}{M}{})\longrightarrow \FF_{n+1}(R;M)
\\
\\
R\colon \FF_{n+1}(R;M)\longrightarrow \FF_{n}(R;M) &\hspace{1cm} \tau\colon M\longrightarrow \FF_{n}(R;M)
\\
\\
\sigma_i \colon \FF_{n}(R;\tens{R}{M}{i})\longrightarrow \FF_{n}(R;\tens{R}{M}{i})
\end{array}
\]
defined by taking $\pi_0$ of the corresponding maps of $\TR_{\langle p^n\rangle}$.
These operators enjoy the following properties, and these are all we need for proving Theorem \ref{pi0norm}.

\begin{prop}\label{Axiomatic}\
\begin{enumerate}[label={(\roman*)}]
\item $R$, $F$ and $V$ and $\sigma_i$ are natural group homomorphisms, with $R$ and $F$ monoidal transformations, and $\tau$ is a natural set valued map.

\item For all $n \geq 1$, the diagrams
\[\xymatrix{ \FF_{n+1}(R;M) \ar[d]^R \ar[r]^F & \FF_{n}(R;\tens{R}{M}{}) \ar[d]^R & & \FF_{n}(R;\tens{R}{M}{}) \ar[d]^R \ar[r]^V & \FF_{n+1}(R;M) \ar[d]^R  \\ \FF_{n}(R;M) \ar[r]^F & \FF_{n-1}(R;\tens{R}{M}{}) & & \FF_{n-1}(R;\tens{R}{M}{}) \ar[r]^V & \FF_{n}(R;M) }\]
commute. Here we use the convention that $\FF_{0}=0$. 
\item For any $i, k \geq 0$, we have $\sigma_i^{p^i}=\id$, $R \sigma_i =\sigma_i R$, and $\sigma_k= \sigma_{i+k}^{p^i}$ as maps
\[ \FF_{n}(R;\tens{R}{(M^{\otimes_R p^i})}{k}) = \FF_{n}(R;\tens{R}{M}{i+k}) \to \FF_{n}(R;\tens{R}{(M^{\otimes_R p^i})}{k}) = \FF_{n}(R;\tens{R}{M}{i+k}). \]
Moreover the maps
\[V \colon \FF_{n}(R;\tens{R}{M}{i+1}) = \FF_{n}(R;\tens{R}{(M^{\otimes_R p^i})}{}) \to \FF_{n+1}(R;\tens{R}{M}{i}) \]
and 
\[F \colon \FF_{n+1}(R;\tens{R}{M}{i}) \to \FF_{n}(R;\tens{R}{(M^{\otimes_R p^i})}{}) =\FF_{n}(R;\tens{R}{M}{i+1}) \]
are equivariant with respect to the projection $C_{p^{i+1}} \to C_{p^{i+1}}/C_p \cong C_{p^i}$ which sends  $\sigma_{i+1}$ to $\sigma_i$. 
\item The following identity holds: $FV=\sum_{k=0}^{p-1} \sigma_1^k \colon \FF_{n}(R;\tens{R}{M}{}) \to \FF_{n}(R;\tens{R}{M}{}).$
\item For any $n \geq 1$, the diagrams 
\[\xymatrix{ M \ar[dr]_-{\tau} \ar[r]^-{\tau} & \FF_{n+1}(R; M) \ar[d]^{R} \\ &  \FF_{n}(R; M)  }
\ \ \ \ \ \ \ \ \ \ \ \ \ \ \ \ \ \xymatrix{ M \ar[d]_{(-)^{\otimes p}} \ar[r]^-{\tau} \ar[r]^-{\tau} & \FF_{n+1}(R; M) \ar[d]^{F} \\ \tens{R}{M}{} \ar[r]^-{\tau} &  \FF_{n}(R; \tens{R}{M}{})}
\]
commute. The map $\tau \colon M \to \FF_{1}(R, M)$ is additive, it sends $[R,M]$ to zero and induces an isomorphism
\[ M/ [R,M] \cong  \FF_{1}(R; M). \]
In particular we have $\cyctens{R}{M}{n} \cong \FF_{1}(R, \tens{R}{M}{n})$, and under this isomorphism the action of $\sigma_n$ is given by permuting the cyclic tensor factors. 
\item For any $n \geq 1$, the sequence
\[\xymatrix{\FF_{1}(R;\tens{R}{M}{n}) \ar[r]^{V^n} &  \FF_{n+1}(R; M) \ar[r]^R &  \FF_{n}(R; M) \ar[r] & 0}\]
is exact. 
\item For any $n \geq 1$, the functor $\FF_n$ commutes with reflexive coequalisers. 
\end{enumerate}
\end{prop}

\begin{proof}
Part (i) follows from the naturality of the operators on $\TR$ and the fact that $F$, $V$, and the action of the cyclic generator are maps of spectra, and the monoidality of $R$ and $F$ follows from the lax monoidality of $\TR$ and $\TR_{\langle p^n \rangle}$. 
Part (ii) follows from Proposition \ref{R maps}. Parts (iii) and (iv) follow from Propositions \ref{prop: truncated TR} and \ref{R maps} and the fact that $F$ and $V$ are respectively restriction and transfer of the $C_{p^{n}}$-Mackey functor $\underline{\pi}_0$ of $\underline{\TR}_{\langle p^{n} \rangle}(R;M)$, and $\sigma_i$ its Weyl action.

The first diagram in Part (v) commutes by definition of $\tau$ for  $\underline{\TR}_{\langle p^{n} \rangle}$ via the truncation 
of $\tau \colon \Sigma^{\infty}_+\underline{M} \to \underline{\TR}(R;M)$ and by Proposition \ref{R maps}. Let us show that the second diagram commutes. Since $\tau$ is defined from a map of $C_{p^{n+1}}$-spaces, it is compatible with the Frobenius. Thus the diagram
\[
\xymatrix{M \ar[d]_{\Delta} \ar[r]^-{\tau} \ar[r]^-{\tau} & \FF_{n+1}(R; M) \ar[d]^{F} \\ M^{\times p} \ar[r]^-{\tau_p} &  \FF_{n}(R; \tens{R}{M}{})}
\]
commutes, where $\tau_p$ is the value at $C_{p^{n+1}}/C_{p^n}$ of the map on $\pi_0$-coefficient systems induced by the morphism of $C_{p^{n+1}}$-spaces $\tau$. 
It is then sufficient to show that $\tau_p$ factors as the composite $M^{\times p}\to \tens{R}{M}{}\to  \FF_{n}(R; \tens{R}{M}{})$ of the canonical map and the map $\tau$ for the bimodule $ \tens{R}{M}{}$ (recall that we are assuming that $R$ and $M$ are discrete). By adjoining the infinite loop space and $\underline{\TR}$, this is the case if the diagram of polygonic spectra
\[
\xymatrix{
\sh_pL(\Sigma^{\infty}_+\underline{M})\ar[d] \ar[r] & \sh_p\underline{\THH}(R;M)\ar@{=}[d]
\\
L(\Sigma^{\infty}_+\underline{ \tens{R}{M}{}})\ar[r] & \underline{\THH}(R; \tens{R}{M}{})
}
\]
commutes, where the horizontal maps are the maps defining $\tau$ under the adjunction, and the vertical map is induced by the canonical map $M^{\times p}\to  \tens{R}{M}{}$. This holds by definition of the horizontal maps.

The properties of $\FF_{1}(R; M)$ follow from the fact that $\TR_{\langle 1\rangle}(R;M)=\THH(R;M)$ (Example \ref{knowncalculations}) and the definition of the cyclic action on $\THH(R;M^{\wedge_R p})$.
Part (vi) follows from the fibre sequence involving the map $R$ for $\TR_{\langle p^n\rangle}(R;M)$.
Let us show Part (vii).
Given a reflexive coequaliser
\[
\xymatrix{(R_1;M_1)\ar@<.5ex>[r]\ar@<-.5ex>[r]&(R_0;M_0)\ar[l]\ar@{->>}[r]&(R;M)}
\]
in the category $\bimod$ of (discrete) bimodules, we need to show that 
\[
\xymatrix{\FF_{n}(R_1;M_1)\ar@<.5ex>[r]\ar@<-.5ex>[r]&\FF_{n}(R_0;M_0)\ar@{->>}[r]&\FF_{n}(R;M)}
\]
is a coequaliser of abelian groups. Using the right Kan extension, the diagram 
\[\xymatrix{(R_1;M_1)\ar@<.5ex>[r]\ar@<-.5ex>[r]&(R_0;M_0)\ar[l]}\]
can be extended to a simplicial object $(R_{\bullet}; M_{\bullet})$ in $\bimod$ (explicitly this can be done by taking iterated pullbacks which are computed underlying, as limits in $\bimod$ are). Upon taking Eilenberg-MacLane spectra we obtain a simplicial object  $(HR_{\bullet}; HM_{\bullet})$ in the category of bimodule spectra.
Since smash powers and smash products of spectra commute with sifted colimits, $\THH(R;M^{\wedge_Rp^{k}})$ commutes with sifted colimits of bimodule spectra for every $k\geq 0$. Thus by induction on the fibre sequences
\[
\THH(R;M^{\wedge_Rp^{k}})_{hC_{p^{k}}}\stackrel{V}{\longrightarrow} \TR_{\langle p^k\rangle}(R;M)\stackrel{R}{\longrightarrow} \TR_{\langle p^{k-1}\rangle}(R;M)
\]
so does $ \TR_{\langle p^{k}\rangle}$ for all $k\geq 0$. It follows that
\[ \vert \TR_{\langle p^{n-1}\rangle}(HR_{\bullet}; HM_{\bullet}) \vert \simeq \TR_{\langle p^{n-1}\rangle}( \vert HR_{\bullet} \vert; \vert HM_{\bullet} \vert).\]
After applying $\pi_0$ on both sides and using Corollary \ref{pi0dependence} we get an isomorphism
\[ \pi_0 \vert \TR_{\langle p^{n-1}\rangle}(HR_{\bullet}; HM_{\bullet}) \vert \cong \pi_0 \TR_{\langle p^{n-1}\rangle}( H\pi_0\vert HR_{\bullet} \vert; H\pi_0 \vert HM_{\bullet} \vert)  \cong \pi_0 \TR_{\langle p^{n-1}\rangle}(R;M)=\FF_{n}(R;M).\]
On the other hand $ \pi_0 \vert \TR_{\langle p^{n-1}\rangle}(HR_{\bullet}; HM_{\bullet})\vert$ fits into the coequaliser diagram
\[ \xymatrix{ \pi_0   \TR_{\langle p^{n-1}\rangle}(HR_{1}; HM_{1})\ar@<.5ex>[r]\ar@<-.5ex>[r]&\pi_0  \TR_{\langle p^{n-1}\rangle}(HR_{0}; HM_{0}) \ar@{->>}[r]& \pi_0 \vert \TR_{\langle p^{n-1}\rangle}(HR_{\bullet}; HM_{\bullet})\vert
}\]  
which proves the desired result.
\end{proof}

\begin{rem}\label{rem:axiomatic}
The proof of Theorem \ref{pi0norm} that we give just below works for any collection of functors $\mathcal{F}_n$ with operators which satisfy the conditions of Proposition \ref{Axiomatic}.
\end{rem}

\begin{proof}[Proof of Theorem \ref{pi0norm}]
Let $(R;M)$ be a bimodule. We start by defining a map
\[I_{n+1} \colon W_{n+1,p}(R;M)\stackrel{}{\longrightarrow} \mathcal{F}_{n+1}(R;M)
\]
by taking a representative $(m_0,m_1,\dots,m_n)$ of a class in $W_{n+1,p}(R;M)$ to
\[
I_{n+1}(m_0,m_1,\dots,m_n):=\sum_{i=0}^n V^i\tau^{n-i}(m_i), 
\]
where $\tau^{n-i}\colon M^{\otimes_Rp^i}\to \mathcal{F}_{n+1-i}(R;M^{\otimes_R p^i})$ is the map $\tau$ for the bimodule $M^{\otimes_R p^i}$, and we wrote 
$V^i\colon  \mathcal{F}_{n+1-i}(R;M^{\otimes_R p^i})\to \mathcal{F}_{n+1}(R;M)$ for the iteration of the map $V$. In order to show that $I_{n+1}$ is well-defined, and ultimately an isomorphism, we need to define an analogue of the ghost maps for $\FF_{n+1}(R; M)$. Define for $0 \leq j < n+1$, 
\[\overline{w}_j := F^j R^{n-j} \colon \FF_{n+1}(R; M) \to \FF_{1}(R; \tens{R}{M}{j})^{C_{p^j}} \cong (\tens{R}{M}{j}/[R,\tens{R}{M}{j}])^{C_{p^j}} =  (\cyctens{R}{M}{j})^{C_{p^j}}\]
where for the latter identification we used (v).
We want to verify that under the map $I_{n+1}$ this map corresponds to the usual ghost map, i.e. that the diagram
\[\xymatrix{\prod_{i=0}^n  \tens{R}{M}{i} \ar[r]^{I_{n+1}} \ar[dr]_{w_j} & \FF_{n+1}(R; M) \ar[d]^{\overline{w}_j } \\ &  (\cyctens{R}{M}{j})^{C_{p^j}} } \]
commutes. To see this we first observe that using (iii)-(iv) the following identities hold:
\begin{align*} F^j V^j&=F^{j-1} FV V^{j-1}= F^{j-1} (\sum_{k=0}^{p-1} \sigma_1^k) V^{j-1}=(\sum_{k=0}^{p-1} \sigma_{j}^{k}) F^{j-1} V^{j-1}
\\
&=(\sum_{k=0}^{p-1} \sigma_{j}^{k}) (\sum_{k=0}^{p-1} \sigma_{j}^{k p}) F^{j-2} V^{j-2}= \dots =(\sum_{k=0}^{p-1} \sigma_{j}^{k}) (\sum_{k=0}^{p-1} \sigma_{j}^{k p}) \cdots (\sum_{k=0}^{p-1} \sigma_{j}^{k p^{j-1}})=\sum_{k=0}^{p^j-1} \sigma_{j}^k=\sum_{\sigma \in C_{p^j}} \sigma, \end{align*}
as endomorphisms of $\FF_{n+1}(R; \tens{R}{M}{j})$. 
It then follows that 
\begin{align*}
\overline{w}_jI_{{n+1}}(m_0,\dots,m_n)&=F^j R^{n-j}\sum_{i=0}^nV^i\tau^{n-i}(m_i)
\\
&=F^j\sum_{i=0}^jV^iR^{n-j}\tau^{n-i}(m_i),
\end{align*}
where we used (ii) and in particular that $R^{n-j}V^i=0$ on $\FF_{n-i+1}(R;\tens{R}{M}{i})$ if $i >j$. Further using (v), (iii) and the previous paragraph, we get
\begin{align*} 
F^j\sum_{i=0}^jV^iR^{n-j}\tau^{n-i}(m_i) = \sum_{i=0}^j\sum_{\sigma\in C_{p^{j}}/C_{p^{j-i}}}\sigma F^{j-i} \tau^{j-i}(m_i)
\\
=\sum_{i=0}^j\sum_{\sigma \in C_{p^{j}}/C_{p^{j-i}}}\sigma m^{\otimes p^{j-i}}_i=w_j(m_0,\dots,m_n),\end{align*}
which shows that the triangle above commutes. 

Before showing that the map $I_{n+1}$ descends to Witt vectors we check that the morphism
\[\overline{w}=(\overline{w}_0,\dots,\overline{w}_n) \colon \FF_{n+1}(R, M) \to \prod_{j=0}^n (\cyctens{R}{M}{j})^{C_{p^j}}  \]
is injective when $(R;M)$ is free (or more generally when $\tr \colon  (\cyctens{R}{M}{j})_{C_{p^j}} \to (\cyctens{R}{M}{j})^{C_{p^j}}$ is injective, see \cite[Lemma 1.4 and Proposition 1.18]{DKNP1}). For $n=0$, the map $\overline{w}$ is just the isomorphism
\[\FF_{1}(R; M) \cong M/[R, M]. \]
Now inductively assume that 
\[\overline{w}=(\overline{w}_0,\dots,\overline{w}_{n-1}) \colon \FF_{n}(R; M) \longrightarrow \prod_{j=0}^{n-1} (\cyctens{R}{M}{j})^{C_{p^j}}  \]
is injective. Consider the diagram with exact rows
\[\xymatrix{& (\cyctens{R}{M}{n})_{C_{p^n}} \ar[d]^{\tr} \ar[r]^{V^n} &  \FF_{n+1}(R; M) \ar[d]^{\overline{w}} \ar[r]^R &  \FF_{n}(R; M) \ar[d]^{\overline{w}}  \ar[r] & 0 \\ 0 \ar[r] & (\cyctens{R}{M}{n})^{C_{p^n}} \ar[r] &  \prod_{j=0}^{n} (\cyctens{R}{M}{j})^{C_{p^j}}  \ar[r]^{\proj} & \prod_{j=0}^{n-1} (\cyctens{R}{M}{j})^{C_{p^j}} \ar[r] & 0 }\]
where the left hand diagram commutes because of (ii), (iv) and (v). The commutativity of the right hand diagram is clear. The left hand vertical map is injective since $(R;M)$ is free. Assuming inductively that the right-hand map is also injective,  the standard diagram chase shows that the middle map is injective as well. 

For free $(R;M)$, since $\overline{w}$ is injective, the definition of the relation defining $W_{n+1,p}(R;M)$ shows that $I_{n+1}$ descends to a well-defined group homomorphism 
\[I_{n+1} \colon W_{n+1,p}(R;M) \longrightarrow \FF_{n+1}(R;M)\]
which is moreover injective.
To see that it is also surjective, we consider the diagram
\[\xymatrix{ (\cyctens{R}{M}{n})_{C_{p^n}} \ar@{=}[d] \ar[r]^{V^n} &  W_{n+1,p}(R; M) \ar[d]^{I_{n+1}} \ar[r]^R &  W_{n,p}(R; M) \ar[d]^{I_{n}}  \ar[r] & 0 \\(\cyctens{R}{M}{n})_{C_{p^n}}  \ar[r]^{V^n} &  \FF_{n+1}(R; M) \ar[r]^R &  \FF_{n}(R; M) \ar[r] & 0. }\]
The right hand square commutes by the construction of $I_n$ and from the fact that $R$ commutes with $V$ and $\tau$. The left hand square commutes by the description of $V^n \colon (\cyctens{R}{M}{n})_{C_{p^n}} \to  W_{n+1,p}(R; M)$ of Section \ref{secop}. The exactness of the top row follows from the results of \cite{DKNP1} (see Section \ref{secop}). Now $I_{0}$ is surjective by (v), and by induction so is the middle map by the standard diagram chase. 

We have thus constructed a natural isomorphism 
\[I_{n+1} \colon W_{n+1,p}(R;M) \cong \FF_{n+1}(R;M)\]
for any free bimodule $(R;M)$. Using the fact that both the source and target commute with reflexive coequalisers, this uniquely extends to a natural isomorphism on the whole category $\bimod$. Also using the naturality one can see that the formula for the general $I_n$ is given as claimed above. 

Finally, let us show that $I_{n+1}$ is monoidal. Again since $\FF_{n+1}$ and $W_{{n+1},p}$ commute with reflexive coequalisers we may show this on the full subcategory of free bimodules. It is then sufficient to see that
\[ \overline{w}_j (I_{n+1}(a) \ast I_{n+1}(b)) =  \overline{w}_j  I_{n+1} ( a \ast b)\]
holds for all $0 \leq j \leq n$. By definition $\overline{w}_j=F^j R^{n-1-j}$ and therefore it is monoidal by $i)$, and
\begin{align*}\overline{w}_j (I_{n+1}(a) \ast I_{n+1}(b))&=\overline{w}_j (I_{n+1}(a)) \ast \overline{w}_j (I_{n+1}(a)) = w_j(a) \ast w_j(b)
\\&
=w_j(a \ast b)=  \overline{w}_j  I_{n-1} ( a \ast b).
\end{align*}
%
%
\end{proof}

\subsection{The Mackey structure on the components of TR} 


We now identify the operators of $\TR$ and the Witt vectors under the isomorphism of Theorem \ref{pi0norm}. By Proposition \ref{prop: truncated TR} there is a genuine $C_{p^n}$-spectrum $\underline{\TR}_{\langle p^n\rangle}(R;M)$ whose genuine $C_{p^k}$-fixed-points are the spectra $\TR_{\langle p^k\rangle}(R;M^{\wedge_R p^{n-k}})$, and in particular the next result identifies the Mackey structure of its components.

\begin{prop} \label{structure compatible} Let $R$ be a connective ring spectrum and $M$ a connective $R$-bimodule. For any $n \geq 1$, the following diagrams commute:
\[
\xymatrix@R=20pt@C=22pt{
W_{\langle p^{n-1}\rangle}(\pi_0R;\tens{\pi_0R}{\pi_0M}{})\ar[d]^-{V}\ar[r]^-{I_{n}}_-{\cong}&\pi_0\TR_{\langle p^{n-1}\rangle}(R; M^{\wedge_R p}) \ar[d]^-{V}
\\
W_{\langle p^n\rangle}(\pi_0R;\pi_0M)\ar[r]^-{I_{n+1}}_-{\cong}& \pi_0\TR_{\langle p^{n}\rangle}(R;M) ,
}\ \ \ \ \ \ \ \ 
\xymatrix@R=20pt@C=30pt{
W_{\langle p^n\rangle}(\pi_0R;\pi_0M) \ar[d]^-{R}\ar[r]^-{I_{n+1}}_-{\cong}&\pi_0\TR_{\langle p^{n}\rangle}(R;M) \ar[d]^{R}
\\
W_{\langle p^{n-1}\rangle}(\pi_0R;\pi_0M) \ar[r]^-{I_{n}}_-{\cong}&
\pi_0\TR_{\langle p^{n-1}\rangle}(R;M) }
\]

\[
\xymatrix@R=20pt@C=22pt{
W_{\langle p^{n}\rangle}(\pi_0R;\pi_0M) \ar[d]^-{F}\ar[r]^-{I_{n+1}}_-{\cong}&\pi_0\TR_{\langle p^{n}\rangle}(R;M)  \ar[d]^{F}
\\
W_{\langle p^{n-1}\rangle}(\pi_0R;\tens{\pi_0R}{\pi_0M}{})\ar[r]^-{I_{n}}_-{\cong}&\pi_0\TR_{\langle p^{n-1}\rangle}(R;M^{\wedge_R p}),
}
\ \ \ \ \ \ \ 
\xymatrix@R=20pt@C=22pt{
W_{\langle p^{n-1}\rangle}(\pi_0R;\tens{\pi_0R}{\pi_0M}{i})\ar[d]^-{\sigma_i}\ar[r]^-{I_{n}}_-{\cong}&\pi_0\TR_{\langle p^{n-1}\rangle}(R;M^{\wedge_R p^i})\ar[d]^{\sigma_i}
\\
W_{\langle p^{n-1}\rangle}(\pi_0R;\tens{\pi_0R}{\pi_0M}{i})\ar[r]^-{I_{n}}_-{\cong}& \pi_0\TR_{\langle p^{n-1}\rangle}(R;M^{\wedge_R p^i})
}
\]
\[
\xymatrix@R=20pt@C=22pt{
  M \ar[d]^-{\tau}\ar[r]^-{\mathrm{id}}& M  \ar[d]^{\tau}
\\
W_{\langle p^{n}\rangle}(\pi_0R;\pi_0M)\ar[r]^-{I_{n+1}}_-{\cong}&\pi_0\TR_{\langle p^{n}\rangle}(R;M),
}
\]
In particular the maps $I_{k}$, for $1\leq k\leq n+1$, determine a monoidal isomorphism of $C_{p^n}$-Mackey functors between $\pi_0\underline{\TR}_{\langle p^{n}\rangle}(R;M)$ and the Mackey functor $C_{p^k}\mapsto W_{\langle p^{k}\rangle}(\pi_0R;\tens{\pi_0R}{\pi_0M}{n-k})$ equipped with the restriction maps $F$ and the transfers $V$ of \S\ref{secop}.
\end{prop}

\begin{proof} By Corollary \ref{pi0dependence} we can assume that $(R;M)$ is discrete, and we prove this theorem for $W_{n,p}(R;M)$. Our argument will moreover work for any collection of functors $\FF_{n}$ satisfying Proposition \ref{Axiomatic}.

The commutativity of the first diagram follows from the construction of the map $I_n$. Indeed, it suffices to show that the diagram commutes after precomposing with the projection 
\[ \prod_{i=0}^{n-1}  \tens{R}{(\tens{R}{M}{})}{i} \to  W_{n,p}(R;\tens{R}{M}{}).\]
We then compute:
\begin{align*} &V  I_{n}(m_0, \dots m_{n-1})= V \sum_{i=0}^{n-1} V^i\tau^{n-1-i}(m_i) = \sum_{i=0}^{n-1} V^{i+1}\tau^{n-1-i}(m_i) 
\\ 
&= \sum_{i=1}^{n} V^{i}\tau^{n-1-{(i-1)}}(m_{i-1})= \sum_{i=1}^{n} V^{i}\tau^{n-i}(m_{i-1})= I_{n+1}(0, m_0, \dots, m_{n-1})=I_{n+1} V(m_0, \dots, m_{n-1}). \end{align*} 
Here we use the description of $V$ on representatives as in Section \ref{secop} (see also  \cite[Proposition 1.23]{DKNP1}), the identification 
\[\tens{R}{(\tens{R}{M}{})}{i}  \cong \tens{R}{M}{i+1}\]
and that under these identifications
\[\tau^{n-1-i}=\tau^{n-(i+1)} \colon \tens{R}{(\tens{R}{M}{})}{i}=\tens{R}{M}{i+1} \to \FF_{n-1-i+1}(R;\tens{R}{(\tens{R}{M}{})}{i}) =  \FF_{n-i}(R;\tens{R}{M}{i+1}),\]
which follows from the naturality of $\tau$. 

Next we check $RI_{n+1}=I_{n}R$. Again it suffices to check that the identity holds after precomposing with the projection
 \[ \prod_{i=0}^n  \tens{R}{M}{i} \to  W_{n+1,p}(R;M),\]
 and we get that
 \begin{align*} 
 RI_{n+1}(m_0, \dots, m_n) = R (\sum_{i=0}^n V^i\tau^{n-i} (m_i) )= \sum_{i=0}^n R V^i\tau^{n-i}(m_i) \\= \sum_{i=0}^{n-1} V^i \tau^{n-1-i}(m_i) = I_{n} R(m_0, \dots m_n). 
 \end{align*}
 Here we have used the description of $R \colon W_{n+1,p}(R;M) \to W_{n,p}(R;M)$ as in Section \ref{secop}, as well as (ii) and (v). In particular, we have used $RV^n(m_n)=0$.
 
 For checking the commutativity of the next two diagrams we employ the ghost components. We first check that
 \[\overline{w} F I_{n+1}=  \overline{w} I_{n} F \]
 in $\prod_{j=0}^{n-1} \cyctens{R}{M}{j}$. This implies the result when $(R;M)$ is free since in this case  $\overline{w}$ is injective as observed above. For general $(R;M)$ the result then follows by naturality. We verify the latter identity for every component $0 \leq j \leq n-1$:
 \[ \overline{w}_j I_{n} F = w_j F =  w_{j+1}  = \overline{w}_{j+1}  I_{n+1} = F^{j+1} R^{n-j-1}  I_{n+1} =F^j R^{n-j-1} F I_{n+1} = \overline{w}_j F I_{n+1},\]
 where we have used (ii) and \cite[Proposition 1.25]{DKNP1}. 
 
For the fourth diagram, we check that the identity $\sigma_i I_{n}=I_{n} \sigma_i$ holds. We can again assume that $(R;M)$ is free and compute componentwise that for $0 \leq j \leq n-1$: 
 \begin{align*} \overline{w}_j I_{n} \sigma_i = w_j \sigma_i = \sigma_{i+j} w_j = \sigma_{i+j} \overline{w}_j I_{n}=\overline{w}_j \sigma_i I_{n}, \end{align*}
 where we used (iii). 
 
 The compatibility with $\tau$ is immediate from the definition of $I_n$.

 In order to identify the full Mackey structure simply notice that for every $0\leq k\leq n$ there is an equivalence of $C_{p^k}$-equivariant spectra
 \[
 \Res^{C_{p^n}}_{C_{p^k}}(\underline{\TR}_{\langle p^n\rangle}(R;M)) \simeq \underline{\TR}_{\langle p^k\rangle}(R;M^{\wedge_R p^{n-k}}).
 \]
 It follows that the transfer of $\underline{\TR}_{\langle p^n\rangle}(R;M)$ from $C_{p^{k-1}}$ to $C_{p^{k}}$ agrees with the same transfer for $\underline{\TR}_{\langle p^k\rangle}(R;M^{\wedge_R p^{n-k}})$, which by the first part of Proposition \ref{structure compatible} applied to the bimodule $M^{\wedge_R p^{n-k}}$ agrees with the Witt vectors Verschiebung. A similar argument identifies the lower restrictions and the Weyl actions.
\end{proof}

\begin{rem}
As in Remark\ref{rem:axiomatic}, the proof of Proposition \ref{structure compatible} works for any collection of functors $\mathcal{F}_n$ with operators which satisfy the conditions of Proposition \ref{Axiomatic}. In fact, the isomorphisms $I_n$ are the unique isomorphisms which satisfy \ref{structure compatible}.
\end{rem}

\begin{rem} \label{tambnorm} Suppose that $R$ is a connective $E_\infty$-ring spectrum and $M$ a connective $E_\infty$-$R$-algebra. In this case we expect that analogously to the map $\tau$ (Section \ref{secnorm}), one can construct multiplicative maps
\[
N\colon \pi_0\TR_{\langle p^{n-1}\rangle}(R;M^{\wedge_Rp})\longrightarrow \pi_0\TR_{\langle p^{n}\rangle}(R;M)
\]
for every $n\geq 1$ which endow $\pi_0\underline{\TR}_{\langle p^{n}\rangle}(R;M)$ with the structure of a $C_{p^n}$-Tambara functor.
%
%
%
%
%
%
These maps should be characterised algebraically by the commutative diagram
\[
\xymatrix@R=20pt@C=60pt{
W_{\langle p^{n-1}\rangle}(\pi_0R;\tens{\pi_0R}{\pi_0M}{})\ar[d]_-{N}\ar[r]^-{I_{n}}_-{\cong}&\pi_0\TR_{\langle p^{n-1}\rangle}(R;M^{\wedge_Rp})\ar[d]^{N}
\\
W_{\langle p^{n}\rangle}(\pi_0R;\pi_0M)\ar[r]^-{I_{n+1}}_-{\cong}&\pi_0\TR_{\langle p^{n}\rangle}(R;M)
}
\]
where the map $N$ on the left is the norm operator of \ref{N}. In particular the maps $I_{k}$, for $1\leq k\leq n+1$, should determine an isomorphism of $C_{p^n}$-Tambara functors between $\pi_0\underline{\TR}_{\langle p^{n}\rangle}(R;M)$ and the Tambara functor $C_{p^k}\mapsto W_{\langle p^{k}\rangle}(\pi_0R;\tens{\pi_0R}{\pi_0M}{n-k})$ equipped with the restriction maps $F$ and the transfers $V$ and the norm $N$ of \S\ref{secop}. 

We leave these observations about the norms on $\TR$ as open questions and encourage interested readers to work out the details. 
 \end{rem}

\subsection{Free Tambara functors and Witt vectors}\label{secbrun}

In this section we describe the free $C_{p^n}$-Tambara functor in terms of Witt vectors with coefficients. This is a result analogous to Brun's \cite[Theorem B]{BrunTamb} which establishes a relationship between these Tambara functors  and the usual Witt vectors of a commutative ring.

We recall that that a $C_{p^n}$-Tambara functor $T$ consists of a commutative ring $T(C_{p^k})$ with a $C_{p^{n-k}}$-action for every $0\leq k\leq n$, together with equivariant maps
\[
\xymatrix{
T(C_{p^{k+1}})\ar[r]|-F & T(C_{p^k})\ar@<1ex>[l]^-{V}\ar@<-1ex>[l]_-N
}
\]
for all $0\leq k\leq n-1$, where $F$ is a ring homomorphism, $V$ is additive, and $R$ is multiplicative. These satisfy certain relations, which can be encoded by declaring $T$ to be a finite products preserving functor on a certain category of double-spans (see \cite{Tambara}). 
We already saw in Remark \ref{TambaraStructure} that for every commutative ring $A$, the functor $\underline{W}_{\langle p^n\rangle}(\Z;A)$ that sends $C_{p^k}$ to the commutative ring
\[
\underline{W}_{\langle p^n\rangle}(\Z;A)(C_{p^k}):=W_{\langle p^k\rangle}(\Z;A^{\otimes p^{n-k}})
\]
and equipped with the operators of \S\ref{secop} is a $C_{p^n}$-Tambara functor.
Let $U$ be the forgetful functor that takes a $C_{p^n}$-Tambara functor $T$ to the underlying commutative ring $T(1)$.

\begin{prop}\label{freetamb} The functor that takes a commutative ring $A$ to $\underline{W}_{\langle p^n\rangle}(\Z;A)$ is left adjoint to $U$.
\end{prop}

\begin{proof}
Let us start by defining a natural transformation
\[
  \hom_{C_{p^n}\mbox{-}\mathrm{Tamb}}(\underline{W}_{\langle p^n\rangle}(\Z;A),T)\longrightarrow  \hom_{\mathrm{Ring}}(A,T(1))
\]
from the morphism set in the category of Tambara functors to that of the category of commutative rings.
We send a morphism of Tambara functors $\alpha\colon \underline{W}_{\langle p^n\rangle}(\Z;A)\to T$ to the ring homomorphism
\[
f\colon A\xrightarrow{\id\otimes 1\otimes\dots\otimes 1}A^{\otimes p^{n}}\xrightarrow{\alpha_{0}} T(1)
\]
where $\alpha_0$ is the value at the trivial group of the natural transformation $\alpha$. We notice that by assumption $\alpha_0$ is $C_{p^n}$-equivariant, and therefore $f$ does not depend on the choice of ordering made in the first map. Moreover $\alpha_0$ can be recovered by $f$, since by equivariancy
\[\alpha_0(a_1\otimes\dots\otimes a_{p^n})=\alpha_0(a_1\otimes 1\otimes \dots\otimes 1)\cdot\dots\cdot \alpha_0(1\otimes\dots\otimes 1\otimes a_{p^n})=\prod_{l=1}^{p^n}\sigma_{p^n}^lf(a_l)\]
where $\sigma_{p^n}\in C_{p^n}$ is a generator
(in other words, $A^{\otimes p^{n}}$ is the free commutative $C_{p^n}$-ring on the commutative ring $A$).

In order to see that the map above defined on hom sets is a bijection, we first assume that $A=\Z[X]$ is the polynomial ring on a set $X$. We can then define an inverse as follows. Given a ring homomorphism $f\colon \Z[X]\to T(1)$, we define $\alpha_0\colon \Z[X]^{\otimes p^n}\to T(1)$ by the formula above. In order to define $\alpha_{C_{p^k}}$ for $0<k\leq n$ we recall from Proposition \ref{basis} that the group $W_{\langle p^k\rangle}(\Z;\Z[X]^{\otimes p^{n-k}})$ is free abelian, generated by the elements $V^i\tau^{k-i}(m_1\otimes \dots\otimes m_{p^{i}})$, where $0\leq i\leq k$, and $(m_1, \dots, m_{p^{i}})$ ranges through the orbits of the $C_{p^i}$-action on the $p^i$-fold product of additive generators of the free abelian group $\Z[X]^{\otimes p^{n-k}}$, that is on the monomials $m_l$ in the set $X^{\amalg p^{n-k}}$. Since $\tau$ relates to $N$ by Proposition \ref{tauandN} and $\alpha$ needs to be compatible with the Tambara structure, we must define $\alpha_{C_{p^k}}$ by
\[
\alpha_{C_{p^k}}(V^i\tau^{k-i}(m_1\otimes\dots\otimes m_{p^i}))=V^iN^{k-i}\alpha_0((m_1\otimes\dots\otimes m_{p^i})\otimes 1^{\otimes (p^{k-i}-1)})\] 
where $V$ and $N$ on the right hand side are the transfer and norm maps of $T$, and  $1$ is the unit of $\Z[X]^{\otimes p^{n-k+i}}$. Since these are free generators, this gives a well-defined additive map
\[
\alpha_{C_{p^k}}\colon W_{\langle p^k\rangle}(\Z;\Z[X]^{\otimes p^{n-k}})\longrightarrow T(C_{p^k}),
\]
for every $0\leq k\leq n$. This map is moreover multiplicative, since by the formula of Proposition \ref{basis}, $\alpha_{C_{p^k}}$ sends the product of two generators, with $i\leq j$, to (abusing notation below, we will denote the unit of some tensor powers of $A$ just by $1$)
\begin{align*}
&\alpha_{C_{p^k}}(V^i\tau^{k-i}(\otimes_{l=1}^{p^i}u_l)\cdot V^j\tau^{k-j}(\otimes_{h=1}^{p^j}v_h))=
  \alpha_{C_{p^k}}(\sum_{\sigma\in C_{p^i}} V^j(\tau^{k-j}((\sigma (\otimes_{l=1}^{p^i}u_l))^{\otimes p^{j-i}}\cdot (\otimes_{h=1}^{p^j}v_h))))
  \\&=\sum_{\sigma\in C_{p^i}}V^jN^{k-j}\alpha_0\big(((\sigma (\otimes_{l=1}^{p^i}u_l))^{\otimes p^{j-i}}\cdot (\otimes_{h=1}^{p^j}v_h))\otimes 1^{\otimes (p^{k-j}-1)}\big)
\\&=\sum_{\sigma\in C_{p^k}/C_{p^{k-i}}}V^jN^{k-j}\alpha_0\big(\prod_{\omega\in C_{p^{k-i}}/C_{p^{k-j}}}\omega\sigma( (u_1\otimes\dots\otimes u_{p^i})\otimes 1^{\otimes (p^{k-i}-1)})\cdot ( (v_1\otimes\dots\otimes v_{p^j})\otimes 1^{\otimes (p^{k-j}-1)})\big)
\\&=V^iN^{k-i}\alpha_0((u_1\otimes\dots\otimes u_{p^i})\otimes 1^{\otimes (p^{k-i}-1)})\cdot V^jN^{k-j}\alpha_0((v_1\otimes\dots\otimes v_{p^j})\otimes 1^{\otimes (p^{k-j}-1)})
\\&=\alpha_{C_{p^k}}(V^i\tau^{k-i}(\otimes_{l=1}^{p^i}u_l))\cdot \alpha_{C_{p^k}}(V^i\tau^{k-j}(\otimes_{h=1}^{p^j}v_h)).
\end{align*}
The fourth equality holds by an argument  analogous to the calculation of the multiplicative structure of Proposition \ref{basis}, by using the Mackey and Tambara identities of $T$.

Let us now show that the collection of maps $\alpha$ is compatible with the Tambara structures. The map $\alpha$ commutes with the transfer maps, since these are additive and  on generators
\begin{align*}
V\alpha_{C_{p^k}}(V^i\tau^{k-i}(m_1\otimes\dots\otimes m_{p^i}))&=VV^iN^{k-i}\alpha_0((m_1\otimes\dots\otimes m_{p^i})\otimes 1^{\otimes (p^{k-i}-1)})
\\&=V^{i+1}N^{k+1-(i+1)}\alpha_0((m_1\otimes\dots\otimes m_{p^i}\otimes 1^{\otimes (p^{k+1-(i+1)}-1)})
\\&=\alpha_{C_{p^{k+1}}}(V^{i+1}\tau^{k+1-(i+1)}(m_1\otimes\dots\otimes m_{p^i}))
\\&=\alpha_{C_{p^{k+1}}}(VV^i\tau^{k-i}(m_1\otimes\dots\otimes m_{p^i}))
\end{align*}
A similar argument shows that $\alpha$ commutes with the restriction maps $F$ on generators, and therefore on all elements since $F$ is additive. Finally let us show that $\alpha$ is compatible with the norms. The norm of a sum of elements is the sum of the norms of those elements plus a sum of transfer terms, see e.g.  \cite[Lemma 5.2]{Angeltveit} for a precise formula. Since We already showed that $\alpha$ is compatible with transfers and multiplication, it is sufficient to show that it commutes with the norms on additive generators. By \cite[Corollary 2.9]{HillMazur}, for abelian $p$-groups norm of a transfer can be described as a transfer applied to a specific polynomial only depending on the group under the consideration. Since we know that $\alpha$ is multiplicative and commutes with transfers and Weyl actions, this shows that $\alpha$ also commutes with norms. 

This shows that $\alpha$ is a well-defined map of Tambara functors. Using again that the Witt vectors above are free as abelian groups, one can easily see that the map that sends $f$ to $\alpha$ is an inverse for the map above, showing that the Witt vectors are a left adjoint on the subcategory of free commutative rings.

Now assume that $A$ is any commutative ring, and consider the functorial reflexive coequaliser diagram of bimodules
\[
  \xymatrix{(\Z;\Z[\Z[A]])\ar@<.5ex>[r]\ar@<-.5ex>[r]&(\Z;\Z[A])\ar@{->>}[r]&(\Z;A)}.
\]
Since the category of Tambara functors is a category of product-preserving functors, and reflexive coequalisers of product-preserving functors are computed pointwise, the diagram
\[
  \xymatrix{\underline{W}_{\langle p^n\rangle}(\Z;\Z[\Z[A]])\ar@<.5ex>[r]\ar@<-.5ex>[r]&\underline{W}_{\langle p^n\rangle}(\Z;\Z[A])\ar@{->>}[r]&\underline{W}_{\langle p^n\rangle}(\Z;A)}
\]
is a coequaliser in the category of $C_{p^n}$-Tambara functors.  Thus for every $C_{p^n}$-Tambara functor $T$, there is a bijection
\[
  \hom_{C_{p^n}\mbox{-}\mathrm{Tamb}}(\underline{W}_{\langle p^n\rangle}(\Z;A),T)\cong \hom_{C_{p^n}\mbox{-}\mathrm{Tamb}}\big(\colim(
\xymatrix@C=10pt{\underline{W}_{\langle p^n\rangle}(\Z;\Z[\Z[A]])\ar@<.5ex>[r]\ar@<-.5ex>[r]&\underline{W}_{\langle p^n\rangle}(\Z;\Z[A]))},T\big)
\]
\[
\cong
\lim\big(\xymatrix@C=10pt{ \hom_{C_{p^n}\mbox{-}\mathrm{Tamb}}(\underline{W}_{\langle p^n\rangle}(\Z;\Z[\Z[A]]),T)\ar@<.5ex>[r]\ar@<-.5ex>[r]& \hom_{C_{p^n}\mbox{-}\mathrm{Tamb}}(\underline{W}_{\langle p^n\rangle}(\Z;\Z[A]),T)}\big)
\]
\[
\cong
\lim\big(\xymatrix@C=10pt{ \hom_{\mathrm{Ring}}(\Z[\Z[A]],T(0))\ar@<.5ex>[r]\ar@<-.5ex>[r]& \hom_{\mathrm{Ring}}(\Z[A],T(0))}\big)
\]
\[
\cong
\hom_{\mathrm{Ring}}\big(\colim(\xymatrix@C=10pt{\Z[\Z[A]]\ar@<.5ex>[r]\ar@<-.5ex>[r]& \Z[A]}),T(0)\big)\cong \hom_{\mathrm{Ring}}(A,T(0))\]
where the third isomorphism follows from the free case above. It is moreover not difficult to see that this bijection coincides with the natural transformation defined above.
\end{proof}

\begin{rem}
It is also possible to prove Proposition \ref{freetamb} from the results of \cite{Ullman}, by identifying $\underline{W}_{\langle p^n\rangle}(\Z;A)$ with the components of the norm via \ref{tambnorm}, and use that $N^{C_{p^n}}_e$ is the left adjoint of the forgetful functor from genuine $C_{p^n}$-commutative equivariant ring spectra to $E_\infty$-ring spectra.
\end{rem}

\begin{example}
Let us describe explicitly the free $C_p$-Tambara functor on a commutative ring $A$. This is the diagram
\[
L(A)=\big(\xymatrix@C=70pt{\ar@(ul,dl)[]_{C_p} \ \ A^{\otimes p}
&W_{2,p}(\Z;A)\cong A\times (A^{\otimes p})_{C_p} \ar[l]|-{F} \ar@{<-}@<-1ex>[l]_-{V}\ar@{<-}@<1ex>[l]^-{N}}
\  \big),
\]
where the $C_p$-action on $A^{\otimes p}$ is the standard one, and
\begin{align*}
F(a,[x])&=a^{\otimes p}+\sum_{\sigma\in C_p}\sigma x
\\
V(x)&=(0,[x])
\\
N(x)&=(\mu_p(x),0)
\end{align*}
where $\mu_p\colon A^{\otimes p}\to A$ is the multiplication map. The ring structure on the right-hand term is that described explicitly in Examples \ref{ex:solid} and \ref{ex:multiplication}.
\end{example}

We now explain the relationship between this result and \cite[Theorem B]{BrunTamb}. Let $U'$ be the forgetful functor from $C_{p^n}$-Tambara functors to commutative rings with $C_{p^n}$-action, which sends $T$ to $T(0)$ with its $C_{p^n}$-action. Let $L'$ denote its left adjoint.

\begin{theorem}[\cite{BrunTamb}]
Let $A$ be a commutative ring, and let us regard $A$ as a commutative ring with the trivial  $C_{p^n}$-action.
There is a natural ring isomorphism
\[
L'(A)(C_{p^n}/C_{p^n})\cong W_{\langle p^n\rangle}(A).
\]
\end{theorem}

We remark that since the restriction  of the $C_{p^n}$-Tambara functor $\pi_0\underline{\TR}_{\langle p^n\rangle}(A)$ to the subgroup $C_{p^{n-1}}$ is the $C_{p^{n-1}}$-Tambara functor $\pi_0\underline{\TR}_{\langle p^n\rangle}(A)$, Brun's Theorem in fact provides a natural isomorphism of Tambara functors
\[
L'(A)(C_{p^n}/C_{p^i})\cong W_{\langle p^i\rangle}(A),
\]
where the Tambara structure on $W_{\langle p^{(-)}\rangle}(A)$ is defined by the Frobenius maps $F$, the Verschiebung $V$, and the norm maps $N$ of \cite{Angeltveit}. 
The unit and the multiplication map of $A$ define a morphism of commutative monoids $(\Z;A^{\otimes p^{n-i}})\to (A;A)$ in the category of bimodules, and thus a ring homomorphism
\[
\mathfrak{m}\colon W_{\langle p^i\rangle}(\Z;A^{\otimes p^{n-i}})\longrightarrow W_{\langle p^i\rangle}(A;A)=W_{\langle p^i\rangle}(A),
\]
which by naturality of the operators is a natural morphism of Tambara functors.

\begin{prop}\label{compbrun}
For every commutative ring $A$, there is a commutative diagram of $C_{p^n}$-Tambara functors
\[
\xymatrix@C=50pt{
W_{\langle p^{i}\rangle}(\Z;A^{\otimes p^{n-i}})\ar[d]_{\mathfrak{m}}\ar[r]^-{\cong}&L(A)(C_{p^n}/C_{p^i})=L'(A^{\otimes p^n})(C_{p^n}/C_{p^i})\ar@<8ex>[d]^{L'(\mu_{p^n})}
\\
W_{\langle p^i\rangle}(A)\ar[r]^{\cong}&\rlap{$L'(A)(C_{p^n}/C_{p^i})$}
}
\]
where the upper isomorphism is from \ref{freetamb} and the lower isomorphism is from Brun's Theorem.
\end{prop}

\begin{proof}
For convenience we denote by $\mathfrak{m}$ also the map $L(A)\to L'(A)$ obtained by transporting  $\mathfrak{m}$ through the horizontal isomorphisms.
Since $L'$ is a left adjoint, the maps $\mathfrak{m}$ and $L'(\mu_{p^n})$ agree if and only if 
\[(U'L'(\mu_{p^n}))\circ \eta=U'(\mathfrak{m})\circ \eta,\]
where $\eta\colon \id\to U'L'$ is the unit of the adjunction. These agree since
\[
U'L'(\mu_{p^n})=\mu_{p^n}=U'(\mathfrak{m})
\]
are both the $p^n$-fold multiplication map.
\end{proof}

\appendix

\section{The Dwork Lemma}

\subsection{Congruences of tensor $p$-powers}

Once and for all, we fix a prime number $p$ and we regard the cyclic groups $C_{p^j}$ as subgroups of the complex circle, so that we have chosen generators $\sigma_j=e^{2i\pi/p^j}$ of $C_{p^j}$ with the property that 
\[\sigma^{p^k}_{j+k}=\sigma_j\]
 for every $j,k\geq 0$. For every abelian group $A$ we equip $A^{\otimes p^j}$ with the $C_{p^j}$-action on which the generator $\sigma_j$ acts by
\[
\sigma_j(a_1\otimes\dots\otimes a_{p^j})=a_{p^j}\otimes a_1\otimes \dots\otimes a_{p^{j}-1}.
\]
We let $(-)^{\otimes p^j}\colon A\to (A^{\otimes p^j})^{C_{p^j}}$
be the (non-additive) map that sends $a$ to the  $p^j$-fold tensor product $a\otimes\dots\otimes a$. 

The following is analogous to the fact that if two elements $a,b\in A$ of a commutative ring $A$ are congruent modulo $p$, then $a^{p^k}$ and $b^{p^k}$ are congruent modulo $p^{k+1}$. It plays a fundamental role in the proof of the Dwork Lemma \ref{Dwork}.

\begin{prop}\label{congruences} Let $A$ be an abelian group, and let $a$ and $b$ be elements of $(A^{\otimes p})^{C_p}$ which are congruent modulo $\tr^{C_p}_e$. Then for every $k\geq 0$
\[
\omega_{k+1}(a^{\otimes p^k})\equiv \omega_{k+1}(b^{\otimes p^k}) \ \mbox{mod}\ \tr_{e}^{C_{p^{k+1}}},
\]
where $\omega_{k+1}$ is the automorphism of $\{1,\dots,p\}^{\times k+1}$ that reverses the order of the product factors. 
\end{prop}

Let $\tau_n$ be the element of the symmetric group $\Sigma_{p^{n}}$ of automorphisms of the set $\{1,\dots,p\}^{\times n}$ that cyclically permutes the $n$-coordinates
\[
\tau_{n}(i_1,\dots,i_n)=(i_n,i_1,i_2,\dots,i_{n-1})
\]
for all $1\leq i_1,\dots,i_n\leq p$.
The key combinatorial ingredient for Proposition \ref{congruences} is the interaction between $\tau_n$ and the cyclic permutations, which is summarised in the following lemma. We write $\{1,\dots,p\}^{\times n}$ in lexicographical order, and think of it as the disjoint union of $p^{n-1}$ blocks of size $p$, or of $p$-blocks of size $p^{n-1}$.

\begin{lemma}\label{tst}
The permutation $\tau_n\in \Sigma_{p^n}$ satisfies
\begin{align*}
\tau_{n}^{-1} \sigma_n \tau_n&=(\sigma_{1}\amalg \id_{ p^{n}-p})\circ (\sigma_{n-1}\times\id_p)
\\
\tau_{n}\sigma_{n}\tau^{-1}_{n}&=(\sigma_{n-1}\amalg\id_{p^n-p^{n-1}}) (\sigma_1\times \id_{p^{n-1}})
\end{align*}
for every $n\geq 1$, where $(\sigma_{n-1}\times\id_p)$ permutes the $p^{n-1}$ blocks of size $p$ by the generator of  $C_{p^{n-1}}$, and $(\sigma_{1}\amalg \id_{ p^{n}-p})$ applies the generator $\sigma_{1}$ of $C_p$ to the first block, and similarly for the second equation.
\end{lemma}

\begin{proof}
The equations can be directly verified using the description of the cyclic permutations
\[
 \sigma_n(i_1,\dots,i_n)=\left\{
 \begin{array}{ll}
 (i_1,\dots,i_n+1)&, i_{n}<p
 \\
  (i_1,\dots,i_{n-1}+1,1)&, i_n=p, i_{n-1}<p
  \\
 \ \ \ \ \ \ \ \ \ \  \vdots
  \\
  (i_1+1,1,\dots,,1)&, i_n=\dots,i_2=p, i_{1}<p
  \\
(1,\dots,1)&, i_n=\dots,i_1=p
 \end{array}
 \right.
\]
\[
(\sigma_{n-1}\times\id_p)(i_1,\dots,i_n)=\left\{
 \begin{array}{ll}
 (i_1,\dots,i_{n-1}+1,i_n)&, i_{n-1}<p
 \\
  (i_1,\dots,i_{n-2}+1,1,i_n)&, i_{n-1}=p, i_{n-2}<p
  \\
 \ \ \ \ \ \ \ \ \ \  \vdots
  \\
  (i_1+1,1,\dots,1,i_n)&, i_{n-1}=\dots,i_2=p, i_{1}<p
  \\
(1,\dots,1,i_n)&, i_{n-1}=\dots,i_1=p
 \end{array}
 \right.
\]
\[
(\sigma_{1}\amalg \id_{ p^{n}-p})(i_1,\dots,i_n)=\left\{
 \begin{array}{ll}
 (i_1,\dots,i_n)&, (i_1,\dots,i_{n-1})\neq (1,\dots,1)
 \\
  (1,\dots,1,i_n+1)&, i_n<p, (i_1,\dots,i_{n-1})=(1,\dots,1)
  \\
  (1,\dots,1,1)&, i_n=p, (i_1,\dots,i_{n-1})=(1,\dots,1) .
 \end{array}
 \right.
\]
and the analogous formulas for $(\sigma_{n-1}\amalg\id_{p^n-p^{n-1}})$ and $(\sigma_1\times \id_{p^{n-1}})$.
\end{proof}

\begin{proof}[Proof of \ref{congruences}]
We start by observing that $\omega_{k+1}\sigma_{k+1}\omega_{k+1}$ is the composition of a block sum of cyclic permutations of $C_p$, and  of a permutation of blocks of size $p$. Thus  since $a$ is in $(A^{\otimes p})^{C_p}$
\[
\sigma_{k+1}\omega_{k+1} (a^{\otimes p^k})=\omega_{k+1} (a^{\otimes p^k}),
\]
and similarly for $\omega_{k+1}(b^{\otimes p^k})$. Thus both $\omega_{k+1} (a^{\otimes p^k})$ and $\omega_{k+1} (b^{\otimes p^k})$ belong to $(A^{\otimes p^{k+1}})^{C_{p^{k+1}}}$, as well as the elements of the image of $ \tr_{e}^{C_{p^{k+1}}}$.

We prove the congruence by induction on $k$. For $k=0$ the claim holds by assumption. For the induction step, we recall the relative binomial formula
\[
(r+s)^{\otimes p}=r^{\otimes p}+s^{\otimes p}+\sum_{\{\emptyset\neq V\subsetneq p\}/C_p}\tr^{C_p}_e(t^{V}_1\otimes\dots\otimes t^{V}_p),
\]
where the sum runs through the orbits of the action of $C_p$ on the subsets of the $p$-elements set,
and $t^{V}_j=r$ if $j\in V$, and $t^{V}_j=s$ otherwise. By supposing that the lemma holds for $k-1$ we see that
\begin{align*}
\omega_{k+1}(a^{\otimes p^{k}})&=\omega_{k+1}((a^{\otimes p^{k-1}})^{\otimes p})=\omega_{k+1}((b^{\otimes p^{k-1}} +\omega_k\tr_{e}^{C_{p^{k}}}(c))^{\otimes p})
\\&=\omega_{k+1}(\id_p\times \omega_k)((\omega_k(b^{\otimes p^{k-1}}) +\tr_{e}^{C_{p^{k}}}(c))^{\otimes p})=\tau_{k+1}^{-1}((\omega_k(b^{\otimes p^{k-1}})  +\tr_{e}^{C_{p^{k}}}(c))^{\otimes p})
\\&=\omega_{k+1}(b^{\otimes p^{k}})+\tau^{-1}_{k+1}((\tr_{e}^{C_{p^{k}}}(c))^{\otimes p})+\sum_{\{\emptyset\neq U\subsetneq p\}/C_p}\tau_{k+1}^{-1}\tr^{C_p}_es^{U}_1\otimes\dots\otimes s^{U}_p,
\end{align*}
where $s^{U}_j=\omega_kb^{\otimes p^{k-1}}$ if $j\in U$, and  otherwise it lies in the image of $\tr_{e}^{C_{p^{k}}}$. Let us analyze the terms of the last sum.
Let us suppose without loss of generality that  a transferred term lies in the first tensor factor . Then the term of the sum corresponding to a subset $U$ is
\begin{align*}
\tau^{-1}_{k+1}\tr^{C_p}_e(\tr_{e}^{C_{p^{k}}}(x)\otimes s^{U}_2\otimes\dots\otimes  s^{U}_p)&=\sum_{j=1}^{p}\sum_{l=1}^{p^k}\tau^{-1}_{k+1}(\sigma^{j}_1\times\id_{p^k})(\sigma_{k}^l\amalg \id_{p^{k+1}-p^k})(x\otimes s^{U}_2\otimes\dots\otimes  s^{U}_p)
\end{align*}
Let us write any $n\in \{1,\dots,p^{k+1}\}$ as $pl+j$ for unique $l\in \{0,1,\dots,p^{k}-1\}$ and $j\in \{1,\dots, p\}$. It is not hard to verify that 
\begin{align*}
\tau_{k+1}\sigma_{k+1}^{pl+j}\tau^{-1}_{k+1}&=(\underbrace{\sigma^{l+1}_k\amalg\dots\amalg \sigma^{l+1}_k}_j\amalg \underbrace{\sigma^{l}_k\amalg\dots\amalg \sigma^{l}_k}_{p-j}) (\sigma^{j}_1\times \id_{p^k})
\\
&=(\sigma^{j}_1\times \id_{p^k}) (\underbrace{\sigma^{l}_k\amalg\dots\amalg \sigma^{l}_k}_{p-j}\amalg \underbrace{\sigma^{l+1}_k\amalg\dots\amalg \sigma^{l+1}_k}_{j}),
\end{align*}
by induction on $n$ from the identity of Lemma \ref{tst}. Since all the terms $s^{U}_j$ for $j=2,\dots,p$ are invariant under the $C_{p^k}$ action, we can rewrite the expression above as
\begin{align*}
&\tau^{-1}_{k+1}\tr^{C_p}_e(\tr_{e}^{C_{p^{k}}}(x)\otimes s^{U}_2\otimes\dots\otimes  s^{U}_p)
\\
&=\sum_{j=1}^{p}\sum_{l=0}^{p^k-1}\tau^{-1}_{k+1}(\sigma^{j}_1\times \id_{p^k}) (\underbrace{\sigma^{l}_k\amalg\dots\amalg \sigma^{l}_k}_{p-j}\amalg \underbrace{\sigma^{l+1}_k\amalg\dots\amalg \sigma^{l+1}_k}_{j})(x\otimes s^{U}_2\otimes\dots\otimes  s^{U}_p)
\\
&=\sum_{j=1}^{p}\sum_{l=0}^{p^k-1}\sigma_{k+1}^{pl+j}\tau^{-1}_{k+1}(x\otimes s^{U}_2\otimes\dots\otimes  s^{U}_p)
\\
&=\sum_{n=1}^{p^{k+1}}\sigma_{k+1}^{n}\tau^{-1}_{k+1}(x\otimes s^{U}_2\otimes\dots\otimes  s^{U}_p)
\\
&=\tr^{C_{p^{k+1}}}_e\tau^{-1}_{k+1}(x\otimes s^{U}_2\otimes\dots\otimes  s^{U}_p).
\end{align*}
This shows that the last term of the expression above lies in the image of $\tr^{C_{p^{k+1}}}_e$, and we are left with verifying that the same holds for $\tau^{-1}_{k+1}((\tr_{e}^{C_{p^{k}}}(c))^{\otimes p})$ for every $k\geq 1$. We recall that the relative multinomial formula for a sequence of $n$-elements $b_1,\dots,b_n$ of an abelian group $B$ is
\[
(b_1+\dots+b_{n})^{\otimes p}=\sum_{f\colon p\to n}b_{f(1)}\otimes\dots\otimes b_{f(p)},
\]
where the sum runs through the set of maps  $f\colon p\to n$. Now let us consider the case where $n=p^k$ for some $k\geq 1$. The group $C_{p^{k+1}}$ acts freely on the set of maps $\{f\colon p\to p^k\}$ by
\[
(\sigma_{k+1}f)(i)=\left\{
\begin{array}{ll}
f(\sigma^{-1}_1i)&\mbox{if }\ i\neq 1
\\
\sigma_kf(\sigma^{-1}_1i)&\mbox{if }\ i= 1.
\end{array}
\right.
\]
The powers of the generator act as
\[
(\sigma_{k+1}^{pl+j}f)(i)=\left\{
\begin{array}{ll}
\sigma_{k}^{l}f(\sigma^{-j}_1i)&\mbox{if }\ j+1\leq i\leq p
\\
\sigma_{k}^{l+1}f(\sigma^{-j}_1i)&\mbox{if }\ 1\leq i\leq j,
\end{array}\right.
\]
where $l\in \{0,1,\dots,p^{k}-1\}$ and $j\in \{1,\dots, p\}$. Indeed, the generator acts freely since if $\sigma_{k+1}f=f$, then mod $p^k$ we must have
\[
f(p)+1=f(1)=f(2)=\dots=f(p),
\]
which is a contradiction. Similarly, if $f$ has non-trivial proper stabilisers we must have that $f=\sigma_{k+1}^{p^n}f$ for some $n\in\{1,\dots,k\}$. By writing $p^n=p(p^{n-1}-1)+p$ we see that we must have
\[
f(i)=(\sigma_{k+1}^{p^n}f)(i)=\sigma_{k}^{p^{n-1}}f(\sigma^{-p}_1i)\equiv f(i)+p^{n-1}\ \mbox{mod}\ p^{k},
\]
which is a contradiction. 
We can therefore decompose the multinomial formula as
\begin{align*}
(b_1+\dots+b_{p^k})^{\otimes p}&=
\sum_{\{f\colon p\to p^k\}/C_{p^{k+1}}}\sum_{n=1}^{p^{k+1}}b_{(\sigma_{k+1}^{n} f)(1)}\otimes\dots\otimes b_{(\sigma_{k+1}^{n} f)(p)}
\\
&\hspace{-1.1cm}=\sum_{\{f\colon p\to p^k\}/C_{p^{k+1}}}\sum_{j=1}^{p}\sum_{l=0}^{p^{k}-1}(b_{\sigma_{k}^{l+1}f(\sigma^{-j}_11)}\otimes\dots\otimes b_{\sigma_{k}^{l+1}f(\sigma^{-j}_1j)}\otimes b_{\sigma_{k}^{l}f(\sigma^{-j}_1(j+1))}\otimes\dots \otimes b_{\sigma_{k}^{l}f(\sigma^{-j}_1p)}).
\end{align*}
We apply this formula to the sequence $\sigma_{k} c,\sigma_{k}^2 c,\dots, \sigma_{k}^{p^k} c$ of the abelian group $B=A^{\otimes p^k}$, and find the expression
\begin{align*}
&\tau^{-1}_{k+1}((\tr_{e}^{C_{p^{k}}}(c))^{\otimes p})\\
&=
\sum_{\{f\colon p\to p^k\}/C_{p^{k+1}}}\sum_{j=1}^{p}\sum_{l=0}^{p^{k}-1}\tau^{-1}_{k+1}((\sigma_{k}^{\sigma_{k}^{l+1}f(\sigma^{-j}_11)}c)\otimes\dots\otimes (\sigma_{k}^{\sigma_{k}^{l+1}f(\sigma^{-j}_1j)}c)
\\&
\otimes (\sigma_{k}^{\sigma_{k}^{l}f(\sigma^{-j}_1(j+1))}c)\otimes\dots \otimes (\sigma_{k}^{\sigma_{k}^{l}f(\sigma^{-j}_1p)}c))
\\&=
\sum_{\{f\colon p\to p^k\}/C_{p^{k+1}}}\sum_{j=1}^{p}\sum_{l=0}^{p^{k}-1}\tau^{-1}_{k+1}(\underbrace{\sigma^{l+1}_k\amalg\dots\amalg \sigma^{l+1}_k}_j\amalg \underbrace{\sigma^{l}_k\amalg\dots\amalg \sigma^{l}_k}_{p-j})(\sigma^{j}_1\times \id_{p^k})(\sigma_{k}^{f(1)}c\otimes\dots\otimes\sigma^{f(p)}_kc)
\\&=
\sum_{\{f\colon p\to p^k\}/C_{p^{k+1}}}\sum_{j=1}^{p}\sum_{l=0}^{p^{k}-1}\sigma_{k+1}^{pl+j}\tau^{-1}_{k+1}(\sigma_{k}^{f(1)}c\otimes\dots\otimes\sigma^{f(p)}_kc)
\\&=\sum_{\{f\colon p\to p^k\}/C_{p^{k+1}}}\tr^{C_{p^{k+1}}}_e\tau^{-1}_{k+1}(\sigma_{k}^{f(1)}c\otimes\dots\otimes\sigma^{f(p)}_kc),
\end{align*}
which concludes the proof.
 \end{proof}

\subsection{External Frobenius and the Dwork Lemma}

We give a characterisation of the image of the ghost map when the bimodule $(R;M)$ is equipped with an ``external Frobenius lift''. 
In the following we will always denote by $\otimes$ the tensor product over the integers.
\begin{defn}
An external Frobenius on a ring $R$ is a ring homomorphism
\[
\varphi\colon R\longrightarrow (R^{\otimes p})^{C_p}
\]
which is congruent to the $p$-th power map $(-)^{\otimes p}$ modulo the image of the additive transfer map $\tr_{e}^{C_p}\colon (R^{\otimes p})_{C_p}\to (R^{\otimes p})^{C_p}$, which sends $x$ to $\sum_{\sigma\in C_p}\sigma(x)$.
\end{defn}

Now let $M$ be an $R$-bimodule. The abelian group $(M^{\otimes p})^{C_p}$ is a $(R^{\otimes p})^{C_p}$-bimodule, where the bimodule actions are factorwise. If $R$ has an external Frobenius $\varphi$ we may then consider $(M^{\otimes p})^{C_p}$ as an $R$-bimodule by restricting scalars along $\varphi\colon R\to (R^{\otimes p})^{C_p}$.

\begin{defn}\label{FrobM}
Let $(R,\varphi)$ be a ring with an external Frobenius. A Frobenius on an $R$-bimodule $M$ is a morphism of $R$-bimodules
\[
\phi\colon M\longrightarrow (M^{\otimes p})^{C_p}
\]
which is congruent to $(-)^{\otimes p}\colon M\to M^{\otimes p}$ modulo the image of $\tr_{e}^{C_p}\colon ( M^{\otimes p})_{C_p}\to  (M^{\otimes p})^{C_p}$. By saying that an $R$-bimodule $M$ has an external Frobenius we will implicitly imply that $R$ also has an external Frobenius.
\end{defn}

\begin{example}\label{Frobpoly}
\
\begin{enumerate}
\item If a ring $R$ has an external Frobenius $\varphi$, the composite
\[
R \stackrel{\varphi}{\longrightarrow} (R^{\otimes p})^{C_p}\stackrel{\mu_p}{\longrightarrow} R
\]
with the $p$-fold multiplication map is a Frobenius lift on $R$ in the usual sense, since $\mu_p((-)^{\otimes p})=(-)^p$ and $\mu_p \tr_{e}^{C_p} \equiv p\mu_p$ modulo $[R,R]$ . 
When $R$ is commutative this is a ring endomorphism of $R$ which is congruent to the $p$-th power  map modulo $p$. For a non-commutative ring this is in the sense of \cite[\S 1.3]{HesselholtncW}, an additive endomorphism of $R$ which preserves the commutators subgroup $[R,R]$, and which is congruent to the $p$-th power map modulo $pR+[R,R]$.
\item The ring of integers has an external Frobenius, defined by the canonical isomorphism $\Z\cong  (\Z^{\otimes p})^{C_p}$ which sends $n$ to $n(1\otimes 1\otimes\dots\otimes 1)$. 
\item Let us denote by $\Z(X):=\oplus_X\Z$ the free abelian group on a set of generators $X$, which we regard as a $\Z$-bimodule. Then $\Z(X)$ has an external Frobenius 
\[\phi\colon \Z(X)\longrightarrow \Z(X)^{\otimes p}\cong \Z(X^{\times p}),\] 
which is defined by sending a basis element $x$ to the diagonal basis element $x^{\otimes p}=(x,\dots,x)$.
We show that that this is congruent to the map $(-)^{\otimes p}$ modulo additive transfer. Let $f$ be a linear combination in $\Z(X)$. By induction, we can assume that $f$ is the linear combination of two basis elements $f=n x+my$. By the relative binomial formula
\[
f^{\otimes p}=(n x+my)^{\otimes p}=(n x)^{\otimes p}+(m y)^{\otimes p}+\sum_{\{\emptyset\neq U\subsetneq p\}/C_p}\tr^{C_p}_e(t^{U}_1\otimes \dots\otimes t^{U}_p),
\]
where $C_p$ acts on the subsets of the $p$-elements set by the image map,
and $t^{U}_k=n x$ if $k\in U$ and $t^{U}_k=my$ otherwise. Thus
\[
f^{\otimes p}\equiv (n x)^{\otimes p}+(m y)^{\otimes p}\equiv n^p \phi(x)+m^p\phi(y)^{\otimes p}\ \mbox{mod}\ \tr^{C_p}_e.
\]
Finally $n^p$ is congruent to $n$ modulo $p$, and similarly for $m^p$, and therefore there are integers $k$ and $l$ such that
\[
f^{\otimes p}\equiv n\phi(x)+m\phi(y)+p\phi(lx+ky)\equiv \phi(nx+my)+\tr^{C_p}_e(\phi(lx+ky))\equiv \phi(f)\ \mbox{mod}\ \tr^{C_p}_e,
\]
where the second congruence holds since the transfer acts by multiplication by $p$ on the fixed-points of $\Z(X)^{\otimes p}$.
\item A completely analogous argument shows that polynomial rings $\Z[X]$ and non-commutative polynomial rings $\Z\{X\}$ have external Frobenius maps which sends $x$ to $x^{\otimes p}$. These refine the standard Frobenius lift in the sense that $\mu_p\phi$ recover the usual Frobenius lift endomorphisms, and they are moreover multiplicative.
\item Let us denote by $R^e(X):=\oplus_X(R\otimes R)$ the free $R$-bimodule on a set of generators $X$, and suppose that $R$ has an external Frobenius $\varphi$. Then $R^e(X)$ has an external Frobenius 
\[\phi\colon R^e(X)\longrightarrow (R^{e}(X)^{\otimes p})^{C_p}\cong (\bigoplus_{X^{\times p}}(R\otimes R)^{\otimes p})^{C_p}\cong \bigoplus_{(X^{\times p})^{C_p}}((R\otimes R)^{\otimes p})^{C_p}\oplus \bigoplus_{(X^{\times p}\backslash \Delta)/C_{p}}(R\otimes R)^{\otimes p},\] 
which is the unique morphism of $R$-bimodules that sends a basis element $x$ to $(x,\dots,x)$ in the first summand. It sends
an element $r\otimes s$ in the $x$-summand to the element $\chi(\varphi(r)\otimes\varphi(s))$ in the $(x,\dots,x)$-summand, where $\chi$ is the shuffle permutation which acts as
\[\chi(r_1\otimes r_2\otimes \dots\otimes r_p\otimes s_1\otimes s_2\otimes \dots\otimes s_p)=(r_1\otimes s_1\otimes r_2\otimes s_2\otimes\dots\otimes r_{p}\otimes s_p).\]
We show that $\phi$ is congruent to the power map $(-)^{\otimes p}$ modulo transfer. As in the example above it is sufficient to show this on the sum of two elements $(r\otimes s)x+(t\otimes u)y$, and by the binomial formula
\[
((r\otimes s)x+(t\otimes u)y)^{\otimes p}\equiv ((r\otimes s)x)^{\otimes p}+((t\otimes u)y)^{\otimes p}\ \mbox{mod}\ \tr^{C_p}_e.
\]
Therefore it is sufficient to show that $\phi((r\otimes s)x)=\chi(\varphi(r)\otimes\varphi(s))(x,\dots,x)$ is congruent to $(r\otimes s)^{\otimes p}(x,\dots,x)$ modulo transfers for every $r,s\in R$ and $x\in X$. Since $\varphi$ is an external Frobenius on $R$ there are $v,w\in R$ such that
\begin{align*}
\phi((r\otimes s)x)&=\chi((r^{\otimes p}+\tr^{C_p}_e(v))\otimes(s^{\otimes p}+\tr^{C_p}_e(w)))(x,\dots,x)
=\big(\chi(r^{\otimes p}\otimes s^{\otimes p})+\chi(T)\big)(x,\dots,x)
\\&=(r\otimes s)^{\otimes p}(x,\dots,x)+\chi(T)(x,\dots,x)
\end{align*}
where $T\in R^{\otimes p}\otimes  R^{\otimes p}$ is a sum of transfers of $R^{\otimes p}$ tensored with a fixed-point of  $(R^{\otimes p})^{C_p}$ (in either order). Thus we need to show that $\chi(T)\in (R\otimes R)^{\otimes p}$ is a transfer. We show this when $T=a\otimes\tr^{C_p}_e(w)$ for a fixed element $a\in R^{\otimes p}$, and
\[\chi(a\otimes \tr^{C_p}_e(w))=\sum_{\sigma\in C_p}\chi(a\otimes \sigma(w))
=\sum_{\sigma\in C_p}\chi(\sigma(a)\otimes \sigma(w))=\sum_{\sigma\in C_p}\sigma \chi(a\otimes w)=\tr^{C_p}_e \chi(a\otimes w)
\]
where the second equality holds because $a$ is a cyclic invariant.
\item A construction analogous to the previous example shows that if $R$ is commutative,  $M$ is a free or free commutative $R$-algebra, and if $R$ has an external Frobenius, then $M$ has an external Frobenius which is multiplicative.
\end{enumerate}
\end{example}

If $M$ is an $R$-bimodule with an external Frobenius and $n\geq 1$ is an integer, we let
\[
\phi^{\otimes p^{n-1}}\colon M^{\otimes p^{n-1}}\longrightarrow M^{\otimes p^{n}}
\]
be  the composite of the map that sends  $m_1\otimes\dots\otimes m_{p^{n-1}}$ to $\phi(m_1)\otimes\dots\otimes \phi(m_{p^{n-1}})\in (M^{\otimes p})^{\otimes p^{n-1}}$, and the canonical isomorphism $(M^{\otimes p})^{\otimes p^{n-1}}\cong  M^{\otimes p^n}$ that sends a generator $m_1\otimes\dots\otimes m_{p^{n}}$ to 
\[
(m_1\otimes\dots\otimes m_{p})\otimes (m_{p+1}\otimes\dots\otimes m_{2p})\otimes\dots\otimes  (m_{(p^{n-1}-1)p+1}\otimes\dots\otimes m_{p^{n}}).
\]
This map is not well-behaved with respect to the cyclic action. In particular, we want to modify this map in such a way that it restricts to a group homomorphism on cyclic invariants. We recall that $\tau_n\in \Sigma_{p^{n}}$ is defined by
\[
\tau_{n}(i_1,\dots,i_n)=(i_n,i_1,i_2,\dots,i_{n-1})
\]
for all $1\leq i_1,\dots,i_n\leq p$.
\begin{lemma}\label{higherFrob} Let $M$ be an $R$-bimodule with an external Frobenius $\phi$, and let $n\geq 1$ be an integer. The map
\[\phi_{n-1}\colon M^{\otimes p^{n-1}}\xrightarrow{\phi^{\otimes p^{n-1}}}M^{\otimes p^{n}}\xrightarrow{\tau_n}M^{\otimes p^{n}}
 \]
satisfies the following properties:
\begin{enumerate}
\item 
It is congruent to $(-)^{\otimes p}$ modulo the image of the  transfer $\tr^{C_p}_e\colon  M^{\otimes p^{n}}\to (M^{\otimes p^{n}})^{C_{p}}$.
\item It descends to a group homomorphism $\phi_{n-1}\colon \cyctens{R}{M}{n-1}\to\cyctens{R}{M}{n}$which satisfies $
\phi_{n-1}\sigma_{n-1}=\sigma_n\phi_{n-1}$,
where $\sigma_k$ is the chosen generator of $C_{p^k}$. In particular it restricts to a group homomorphism 
\[
\phi_{n-1}\colon (\cyctens{R}{M}{n-1})^{C_{p^{n-1}}}\longrightarrow (\cyctens{R}{M}{n})^{C_{p^{n}}}.
\]
\item
For every $i,k\geq 0$ and element $m_i\in M^{\otimes p^i}$ we have 
$\phi_{k+i}(m_{i}^{\otimes p^k})\equiv m_{i}^{\otimes p^{k+1}}\ \mbox{mod}\ \ \tr_{e}^{C_{p^{k+1}}}$.
\end{enumerate}
\end{lemma}

\begin{example}
Let us consider a free abelian group $\Z(X)$ with the external Frobenius that sends $x$ to $(x,\dots,x)$ in $\Z(X^{\times p})\cong\Z(X)^{\otimes p}$.
Under the isomorphism $\Z(X^{\times p^n})\cong\Z(X)^{\otimes p^n}$, the higher Frobenius $\phi_{n-1}$ sends a generator $(x_1,\dots,x_{p^{n-1}})$ to
\[
\phi_{n-1}(x_1,\dots,x_{p^{n-1}})=(x_1,\dots,x_{p^{n-1}},x_1,\dots,x_{p^{n-1}},\dots,x_1,\dots,x_{p^{n-1}}),
\]
whereas $\phi^{\otimes p^{n-1}}$ sends it to $(x_1,\dots,x_1,x_2,\dots,x_2,\dots,x_{p^{n-1}},\dots,x_{p^{n-1}})$.
\end{example}

%


\begin{proof}[Proof of \ref{higherFrob}] 
We start by showing that $\phi_{n-1}$ descends to a map on the cyclic tensor powers over $R$. Let us consider the commutative diagram 
\[
\xymatrix@C=50pt{
M^{\otimes p^{n-1}}\ar@{->>}[d]\ar[r]^-{\phi^{\otimes p^{n-1}}}&M^{\otimes p^{n}}\ar[rr]^{\tau_n}\ar@{->>}[d]&&M^{\otimes p^{n}}\ar@{->>}[d]
\\
\cyctens{R}{M}{n-1}\ar[r]_-{\cyctens{R}{\phi}{n-1}}&\cyctens{R}{(M^{\otimes p})}{n-1}\ar@{->>}[r]&(M^{\otimes p})^{\circledcirc_{((R^{\otimes p})^{C_p})}p^{n-1}} \ar@{-->}[r]_-{\tau_n}&\cyctens{R}{M}{n}\rlap{\ .}
}
\]
The bottom-left horizontal map is well-defined because $\phi$ is a map of $R$-bimodules, where the bimodule structure on the target is via the external Frobenius $\varphi\colon R\to (R^{\otimes p})^{C_p}$ of $R$. The middle bottom horizontal map is the projection map that regards $M^{\otimes p}$ as an $(R^{\otimes p})^{C_p}$-bimodule. Thus we need to show that the permutation $\tau_n$ gives a well-defined map on the bottom-right. We show that it is well-defined with respect to tensoring over $(R^{\otimes p})^{C_p}$ for the first tensor factor, the others are similar. For every $m\in M^{\otimes p^{n}}$ and $r\in R^{\otimes p^{}}$ we need to show that
\[
\tau_n(m\cdot (r\otimes 1^{\otimes (p^n-p)}))=\tau_n((1^{\otimes p}\otimes r\otimes 1^{\otimes (p^n-2p)})\cdot m).
\]
Since in the target we are tensoring over $R$, we can turn the right action of an element of $R^{\otimes p^{n}}$ into the left action by a cyclic permutation of this element, and by Lemma \ref{tst}
\begin{align*}
\tau_n(m\cdot (r\otimes 1^{\otimes (p^n-p)}))&=\tau_n(m)\cdot \tau_n(r\otimes 1^{\otimes (p^n-p)})=(\sigma_n\tau_n(r\otimes 1^{\otimes (p^n-p)}))\cdot \tau_n(m)
\\&=(\tau_n(\sigma_{1}\amalg \id_{ p^{n}-p})\circ (\sigma_{n-1}\times\id_p)(r\otimes 1^{\otimes (p^n-p)}))\cdot \tau_n(m)
\\&=(\tau_n(1^{\otimes p}\otimes r\otimes 1^{\otimes (p^n-2p)}))\cdot \tau_n(m)
\\&=\tau_n((1^{\otimes p}\otimes r\otimes 1^{\otimes (p^n-2p)})\cdot m).
\end{align*}

Let us now show that $\phi_{n-1}$ and $\cyctens{R}{(-)}{}$ are congruent modulo the transfer map.
Let $x=\sum a_1\otimes\dots\otimes a_{p^{n-1}}$ be an element of $M^{\otimes p^{n-1}}$, and $b_j\in  \cyctens{R}{M}{}$ such that $\phi(a_j)= \cyctens{R}{a_j}{}+\tr^{C_p}_eb_j$.
Then in $\cyctens{R}{M}{n}$ we have that
\begin{align*}
\phi_{n-1}(x)&=\sum \tau_n(\phi(a_1)\otimes\dots\otimes \phi(a_{p^{n-1}}))=\sum \tau_n((a^{\otimes p}_1+\tr^{C_p}_eb_1)\otimes\dots\otimes (a^{\otimes p}_{p^{n-1}}+\tr^{C_p}_eb_{p^{n-1}}))
\\&=\sum \tau_n (a^{\otimes p}_1\otimes\dots\otimes a^{\otimes p}_{p^{n-1}}+\sum_{\emptyset\neq V\subset p^{n-1}}t^{V}_1\otimes\dots\otimes t^{V}_{p^{n-1}})
\\&=\sum ((a_1\otimes\dots\otimes a_{p^{n-1}})^{\otimes p}+\sum_{\emptyset\neq V\subset p^{n-1}}\tau_n(t^{V}_1\otimes\dots\otimes t^{V}_{p^{n-1}})).
\end{align*}
The inner sum runs through the non-empty subsets of the set with $p^{n-1}$ elements, where  $t^{V}_{j}= \tr^{C_p}_eb_{j}$ if $j\in V$, and  $t^{V}_{j}= a^{\otimes p}_j$ otherwise. Each of the terms of this sum contains at least one transferred tensor factor, let us say for simplicity the first one. Then
\begin{align*}
\tau_n (t^{V}_1\otimes\dots\otimes t^{V}_{p^{n-1}})&=\tau_n(\tr^{C_p}_eb_{1}\otimes t^{V}_2\otimes\dots\otimes t^{V}_{p^{n-1}})=\tau_n((\sum_{\sigma\in C_p}\sigma(b_1))\otimes t^{V}_2\otimes\dots\otimes t^{V}_{p^{n-1}})
\\&=\sum_{\sigma\in C_p}\tau_n(\sigma(b_1)\otimes t^{V}_2\otimes\dots\otimes t^{V}_{p^{n-1}})=\sum_{\sigma\in C_p}\tau_n(\sigma(b_1)\otimes \sigma(t^{V}_2)\otimes\dots\otimes \sigma(t^{V}_{p^{n-1}}))
\\
&=\sum_{\sigma\in C_p}(\sigma\times\id_{p^{n-1}} )\tau_n(b_1\otimes t^{V}_2\otimes\dots\otimes t^{V}_{p^{n-1}})
\\
&=\tr_{e}^{C_p} \tau_n(b_1\otimes t^{V}_2\otimes\dots\otimes t^{V}_{p^{n-1}}),
\end{align*}
where the fourth equality holds because $t^{V}_{j}$ is $C_p$-invariant for every $j$. It follows that
\begin{align*}
\phi_{n-1}(x)&\equiv  \sum (a_1\otimes\dots\otimes a_{p^{n-1}})^{\otimes p}\ \mbox{mod}\ \tr^{C_p}_e.
\end{align*}
Let us show that this is congruent to the $p$-th tensor power of $x=\sum a_1\otimes\dots\otimes a_{p^{n-1}}$. By induction we can assume that $x$ is the sum of two elementary tensors $x= a_1\otimes\dots\otimes a_{p^{n-1}}+a'_1\otimes\dots\otimes a'_{p^{n-1}}$, and
\begin{align*}
x^{\otimes p}&=(a_1\otimes\dots\otimes a_{p^{n-1}})^{\otimes p}+(a'_1\otimes\dots\otimes a'_{p^{n-1}})^{\otimes p}+\sum_{\emptyset\neq U\subsetneq p}s^{U}_1\otimes\dots\otimes s^{U}_{p}
\\&=(a_1\otimes\dots\otimes a_{p^{n-1}})^{\otimes p}+(a'_1\otimes\dots\otimes a'_{p^{n-1}})^{\otimes p}+\sum_{\{\emptyset\neq U\subsetneq p\}/C_p}\tr_{e}^{C_p}(s^{U}_1\otimes\dots\otimes s^{U}_{p}),
\end{align*}
where $s^{U}_j=a_1\otimes\dots\otimes a_{p^{n-1}}$ if $j\in U$, and  $s^{U}_j=a'_1\otimes\dots\otimes a'_{p^{n-1}}$ otherwise. The last equality holds because $s^{U}_j$ is constant for different values of $j\in U$, and $C_p$ acts on the proper non-empty subsets of $p$ by cyclically permuting their elements. This concludes the proof that $\phi_{n-1}$ and $(-)^{\otimes p}$ are congruent modulo $\tr_{e}^{C_p}$.

Now let us show that $\phi_{n-1}$ is equivariant. By the relations of Lemma \ref{tst}
\begin{align*}
\sigma_{n} \phi_{n-1}&=\sigma_{n}\tau_n\phi^{\otimes p^{n-1}}=\tau_n (\sigma_{1}\amalg \id_{ p^{n}-p})\circ (\sigma_{n-1}\times\id_p)\phi^{\otimes p^{n-1}}\\
&=\tau_n (\sigma_{1}\amalg \id_{ p^{n}-p})\phi^{\otimes p^{n-1}}\sigma_{n-1}
\\
&=\tau_n \phi^{\otimes p^{n-1}}\sigma_{n-1}=\phi_{n-1}\sigma_{n-1}.
\end{align*}
The fourth equality holds since $\phi$ takes values in the invariants $(M^{\otimes p})^{C_p}$. 

For the last statement, we calculate that on representatives
\begin{align*}
\phi_{k+i}(m_{i}^{\otimes p^{k}})&=\tau_{k+i+1}\phi^{\otimes p^{k+i}}(m_{i}^{\otimes p^{k}})=\tau_{k+i+1}((\phi^{\otimes p^i}m_{i})^{\otimes p^{k}})
\\&=(\tau_{k+1}\times\id_{p^i})((\tau_{i+1}\phi^{\otimes p^i}m_{i})^{\otimes p^{k}})=(\tau_{k+1}\times\id_{p^i})((\phi_im_{i})^{\otimes p^{k}})
\\&=(\omega_{k+1}\times\id_{p^i})((\phi_im_{i})^{\otimes p^{k}})=(\omega_{k+1}\times\id_{p^i})(m_{i}^{\otimes p^{k+1}})+\tr^{C_{p^{k+1}}}_e
\\&=m_{i}^{\otimes p^{k+1}}+\tr^{C_{p^{k+1}}}_e,
\end{align*}
where the fifth equality follows since $\tau_{k+1}=\omega_{k+1}(\omega_k\times \id_p)$, and $(\phi_ia_{i})^{\otimes p^{k}}$ is invariant under the action of $\Sigma_{p^k}$. The sixth equality holds in $M^{\otimes p^{k+i+1}}$ by the congruence of Proposition \ref{congruences}, since $\phi_i$ is congruent to $(-)^{\otimes p}$ modulo $\tr^{C_p}_e$ as elements of $M^{\otimes p^{i+1}}$.
\end{proof}

\begin{theorem}[Dwork Lemma]\label{Dwork}
Let $M$ be an $R$-bimodule with  an external Frobenius $\phi\colon M\to (M^{\otimes p})^{C_p}$. A sequence $(b_0,b_1,\dots,b_{n-1})$ of $ \prod_{j=0}^{n-1}(\cyctens{R}{M}{j})^{C_{p^j}}$ lies in the image of the ghost map $w\colon W_{n,p}(R;M)\to \prod_{j=0}^{n-1}(\cyctens{R}{M}{j})^{C_{p^j}}$ if and only if
\[
\phi_j(b_j)\equiv b_{j+1}\ \mbox{mod}\ \tr_{e}^{C_{p^{j+1}}}
\]
for every $0\leq j<n-1$,
where the congruence is modulo the image of the additive transfer map $\tr_{e}^{C_{p^{j+1}}}\colon \cyctens{R}{M}{j+1}\to (\cyctens{R}{M}{j+1})^{C_{p^{j+1}}}$. 
\end{theorem}

\begin{proof}
Let us start by showing that a sequence in the image of the relative ghost map satisfies these congruences, that is that
\[
\phi_j(\sum_{i=0}^{j}\tr_{C_{p^{j-i}}}^{C_{p^j}}(m_i^{\otimes p^{j-i}}))\equiv \sum_{i=0}^{j+1}\tr_{C_{p^{j+1-i}}}^{C_{p^{j+1}}}(m_i^{\otimes p^{j+1-i}}) \ \mbox{mod}\ \tr_{e}^{C_{p^{j+1}}}.
\]
We observe that the $(j+1)$-st term of the sum on the right is in the image of $ \tr_{e}^{C_{p^{j+1}}}$, and thus it is sufficient to show that for every $0\leq i\leq j$
\[
\phi_j(\tr_{C_{p^{j-i}}}^{C_{p^j}}(m_i^{\otimes p^{j-i}}))\equiv \tr_{C_{p^{j+1-i}}}^{C_{p^{j+1}}}(m_i^{\otimes p^{j+1-i}}) \ \mbox{mod}\ \tr_{e}^{C_{p^{j+1}}}.
\]
We calculate the left-hand side
\begin{align*}
\phi_j(\tr_{C_{p^{j-i}}}^{C_{p^j}}(m_i^{\otimes p^{j-i}}))&=\phi_j(\sum_{\sigma\in C_{p^{j}}/C_{p^{j-i}}}\sigma(m_i^{\otimes p^{j-i}}))=\sum_{l=1}^{p^i}\phi_j(\sigma_{j}^l(m_i^{\otimes p^{j-i}}))=\sum_{l=1}^{p^i}\sigma_{j+1}^l\phi_j(m_i^{\otimes p^{j-i}})
\\&=\tr_{C_{p^{j+1-i}}}^{C_{p^{j+1}}}\phi_j(m_i^{\otimes p^{j-i}})=\tr_{C_{p^{j+1-i}}}^{C_{p^{j+1}}}(m_i^{\otimes p^{j+1-i}}+\tr_{e}^{C_{p^{j-i+1}}}(z_{i,j}))
\\&=\tr_{C_{p^{j+1-i}}}^{C_{p^{j+1}}}(m_i^{\otimes p^{j+1-i}})+\tr_{e}^{C_{p^{j+1}}}(z_{i,j}),
\end{align*}
where the fifth equality is from Lemma \ref{higherFrob}.
Conversely, let $(b_0,b_1,\dots)$ be a sequence which satisfies the congruences of the statement, and suppose that we found $a_0,\dots,a_j$ such that $b_j=\omega_j(a_0,\dots,a_j)$. Then
\begin{align*}
b_{j+1}&=\phi_j(b_j)+\tr^{C_{p^{j+1}}}_e(x)=\phi_j(\omega_j(a_0,\dots,a_j))+\tr^{C_{p^{j+1}}}_e(x)
\\&=\omega_{j+1}(a_0,\dots,a_j,0)+\tr^{C_{p^{j+1}}}_e(y)+\tr^{C_{p^{j+1}}}_e(x)\\
&=\omega_{j+1}(a_0,\dots, a_j,y+x).
\end{align*}
\end{proof}

\begin{cor}\label{algTD}
Let $M$ be an $R$-bimodule with  an external Frobenius $\phi\colon M\to (M^{\otimes p})^{C_p}$, such that the transfer maps $\tr^{C_{p^j}}_e\colon (\cyctens{R}{M}{j})_{C_{p^j}}\to (\cyctens{R}{M}{j})^{C_{p^j}}$ are injective (for example if $(R;M)$ is a free bimodule). Then there is a canonical isomorphism of abelian groups
\[
f_\phi\colon \bigoplus_{i=0}^{n-1}(\cyctens{R}{M}{i})_{C_{p^i}}\stackrel{\cong}{\longrightarrow} W_{n,p}(R;M)
\]
with ghosts $w_jf_\phi(a_0,\dots,a_{n-1})=\sum_{i=0}^j\tr_{C_{p^{j-i}}}^{C_{p^j}}\phi^{j-i}(a_i)$, where $\phi^{j-i}(a_i)=\phi_{i+j}\dots\phi_{i+1}\phi_i(a_i)\in (\cyctens{R}{M}{j})_{C_{p^i}}^{C_{p^{j-i}}}$.
\end{cor}

\begin{proof}
The formula for $w_jf_\phi$ gives an additive map  $f^w\colon \bigoplus_{j=0}^{n-1}(\cyctens{R}{M}{j})_{C_{p^j}}\to  \prod_{j=0}^{n-1}(\cyctens{R}{M}{j})^{C_{p^j}}$. Since $\tr_{C_{p^{j-i}}}^{C_{p^j}}\phi^i(a_i)=\phi^{j-i}\tr_{e}^{C_{p^{j-i}}}(a_i)$ we see that
\[
\phi_{j}f_{j}^w=\phi_j(\sum_{i=0}^j\tr_{C_{p^{j-i}}}^{C_{p^j}}\phi^{j-i}(a_i))=f^{w}_{j+1}-\tr^{C_{p^{j+1}}}_ea_j.
\]
Thus by the Dwork lemma $f^w$ hits precisely the image of $w$. Since the transfers are injective the ghost $w$ is injective and $f^w$ lifts to a surjection $f_\phi$ onto $W_{n,p}(R;M)$. Again since the transfers are injective it is easy to see inductively that $f^w$ is injective, and therefore so is $f_\phi$.
\end{proof}

\begin{rem}
We observe that the maps $f_\phi$ are not natural with respect to morphisms of free bimodules, as these are not necessarily compatible with the external Frobenius. They are however natural with respect to those morphisms of bimodules which are induced by a map of bases. This is analogous to the tom-Dieck splitting, which is natural with respect to maps of spaces but not with respect to all the maps of suspension spectra.
\end{rem}

\begin{lemma}\label{iterFrob}
Let $\varphi\colon R\to (R^{\otimes p})^{C_p}$ be an external Frobenius on a ring $R$. Then for every $k\geq 0$
\[\Phi:=\varphi_k=\tau_{k+1}\varphi^{\otimes p^{k}}\colon R^{\otimes p^k}\longrightarrow (R^{\otimes p^{k+1}})^{C_p}\]
 is an external Frobenius on the ring $R^{\otimes p^{k}}$, with higher Frobenius $\Phi_n=\varphi_{k+n}$.
 
If $M$ is an $R$-bimodule with an external Frobenius  $\phi\colon M\to (M^{\otimes p})^{C_p}$ and $k\geq 0$, then 
\[\Phi:=\phi_k=\tau_{k+1}\phi^{\otimes p^{k}}\colon M^{\otimes p^k}\longrightarrow (M^{\otimes p^{k+1}})^{C_p}\]
 is an external Frobenius on the $R^{\otimes p^k}$-bimodule $M^{\otimes p^k}$, with higher Frobenius $\Phi_n=\phi_{k+n}$.
 Moreover this descends to an external Frobenius 
 \[
 \Phi\colon \tens{R}{M}{k}\longrightarrow ((\tens{R}{M}{k})^{\otimes p})^{C_p}
 \]
 on the $R$-bimodule $\tens{R}{M}{k}$.
\end{lemma}
\begin{proof}
We start by proving the claim for the Frobenius of $R^{\otimes p^k}$.  Since $\varphi_k\sigma_k=\sigma_{k+1}\varphi_k$ we have that $\Phi=\varphi_k$ lands in the $C_p$-fixed-points, since the generator of $C_{p}\subset C_{p^{k+1}}$ acts by
\[
\sigma^{p^{k}}_{k+1}\Phi=\Phi\sigma^{p^k}_k=\Phi.
\]
Moreover by Lemma \ref{higherFrob} $\Phi$ is congruent to the $p$-th power map modulo $\tr^{C_p}_e$. Next, we determine the higher Frobenius:
\begin{align*}
\Phi_n&=(\tau_{n+1}\times \id_p)\Phi^{\otimes p^{ n}}=(\tau_{n+1}\times \id_p)(\tau_{k+1}\varphi^{\otimes p^k})^{\otimes p^{n}}=(\tau_{n+1}\times \id_{p^k})(\id_{p^{n}}\times \tau_{k+1})\varphi^{\otimes p^{n+k}}
\\&
=\tau_{n+k+1}\varphi^{\otimes p^{n+k}}
=\varphi_{n+k}.
\end{align*}
The claim for $M^{\otimes p^k}$ is completely analogous, by remarking that $\phi_k$ is indeed a map of $R^{\otimes p^k}$-bimodules. As for $\tens{R}{M}{k}$,
the proof that $\tau_{k+1}\phi^{\otimes p^{k}}$ is well-defined is analogous to the first argument of the proof of \ref{higherFrob}. It remains to verify that it is a map of $R$-bimodules. By definition
\begin{align*}
\tau_{k+1}\phi^{\otimes p^{k}}(r\cdot m_1\otimes\dots\otimes m_{p^{k}})&=\tau_{k+1}(\phi(r\cdot m_1)\otimes\dots\otimes \phi(m_{p^{k}}))
\\&=\tau_{k+1}((\varphi(r)\cdot\phi( m_1))\otimes\dots\otimes \phi(m_{p^{k}}))
\\&=\tau_{k+1}((\varphi(r)\otimes 1\otimes\dots\otimes 1)\cdot\phi^{\otimes p^{k}}(m_1\otimes\dots\otimes m_{p^{k}}))
\\&=\tau_{k+1}((\varphi(r)\otimes 1\otimes\dots\otimes 1))\cdot\tau_{k+1}\phi^{\otimes p^{k}}(m_1\otimes\dots\otimes m_{p^{k}}),
\end{align*}
and the action of $r$ of the left $R$-module structure on $((\tens{R}{M}{k})^{\otimes p})^{C_p}$ induced by $\varphi\colon R\to (R^{\otimes p})^{C_p}$ is precisely multiplication by $\tau_{k+1}((\varphi(r)\otimes 1\otimes\dots\otimes 1))$.
\end{proof}

\phantomsection\addcontentsline{toc}{section}{References} 
\bibliographystyle{amsalpha}
\bibliography{bib}

\end{document}